\documentclass[reqno,10pt]{amsart}

\usepackage{a4wide}
\usepackage{color}
\usepackage[T2A]{fontenc}
\usepackage[utf8]{inputenc}
\usepackage[english]{babel}
\numberwithin{equation}{section}

\usepackage{amsmath, amsfonts, amssymb, amsthm, mathtools, bm, bbm, esint, mathrsfs, nccmath, upgreek, thmtools}
\usepackage[backend=biber,style=numeric-comp,sortcites,giveninits,doi=false,isbn=false,url=false,eprint=false, maxbibnames=99]{biblatex}

 \usepackage{relsize}

\usepackage[colorlinks,citecolor=blue,linkcolor=red]{hyperref}
\usepackage[noabbrev,capitalise,nameinlink]{cleveref}
\hypersetup{colorlinks={true},linkcolor={red},citecolor=blue}
\usepackage{csquotes}

\makeatletter
\newcommand{\bg}{\bBigg@{0.5}}
\newcommand{\biggg}{\bBigg@{3}}
\newcommand{\Biggg}{\bBigg@{3.5}}
\newcommand{\bigggg}{\bBigg@{4}}

\makeatother

\usepackage{graphicx}
\usepackage{resmes}
\newcommand{\rmes}{\,{\mathlarger{\resmes}}}

\AtBeginBibliography{\footnotesize\sloppy}

\DeclareNameAlias{author}{family-given}

\theoremstyle{plain}
\newtheorem{theo}{Theorem}[section]
\newtheorem{lem}[theo]{Lemma}
\newtheorem{prop}[theo]{Proposition}

\theoremstyle{definition}
\newtheorem{defi}[theo]{Definition}
\newtheorem{rem}[theo]{Remark}
\newtheorem{ques}[theo]{Question}
\newtheorem{conj}[theo]{Conjecture}



\usepackage{scalerel}

\newcommand{\lims}{\varlimsup}

\newcommand{\1}{\mathbbm{1}}
\newcommand{\G}{\mathit{\Gamma}}
\newcommand{\Cur}{\mathrm{Cur}}
\renewcommand{\o}{\mathrm{o}}
\newcommand{\Fr}{\mathrm{Fr}}
\newcommand{\GL}{\mathrm{GL}}
\newcommand{\GDis}{\operatorname{{\kern-0.1em}GDis}}
\newcommand{\ADis}{\operatorname{{\kern-0.1em}ADis}}
\newcommand{\Sel}{\operatorname{{\kern-0.1em}Sel}}
\newcommand{\Dual}{\operatorname{{\kern-0.1em}Dual}}
\newcommand{\BLIP}{\mathrm{BLIP}}
\newcommand{\Mat}{\mathbb{M}}
\newcommand{\X}{\mathsf{X}}

\newcommand{\Nat}{\mathbb{N}}
\newcommand{\Y}{\mathsf{Y}}

\newcommand{\dX}{\mathsf{d}_{\mathsf{X}}}
\newcommand{\dY}{\mathsf{d}_{\mathsf{Y}}}

\newcommand{\D}{\mathrm{d}}
\newcommand{\Cl}{\mathrm{Cl}}
\newcommand{\Aue}{\mathrm{Aue}}
\newcommand{\Leb}{\mathrm{L}}

\newcommand{\norm}{\mathsf{n}}
\newcommand{\Norm}{\mathsf{N}}
\newcommand{\ms}{\operatorname{{\kern-0.1em}\mathsf{ms}}}
\newcommand{\ev}{\mathsf{ev}}
\newcommand{\Mod}{\mathrm{Mod}}

\newcommand{\dom}{\mathrm{dom}}
\newcommand{\im}{\mathrm{im}}

\newcommand{\Mas}{\operatorname{{\kern-0.1em}\mathbf{M}}}
\newcommand{\Len}{\operatorname{{\kern-0.1em}Len}}

\newcommand{\Dif}{\mathcal{D}}
\newcommand{\m}{\mathfrak{m}}
\newcommand{\M}{\mathscr{M}}

\newcommand{\R}{\mathbb{R}}

\newcommand{\e}{\varepsilon}
\renewcommand{\P}{\mathrm{P}}
\newcommand{\LIP}{\mathrm{LIP}}

\makeatletter
\let\save@mathaccent\mathaccent
\newcommand*\if@single[3]{%
  \setbox0\hbox{${\mathaccent"0362{#1}}^H$}%
  \setbox2\hbox{${\mathaccent"0362{\kern0pt#1}}^H$}%
  \ifdim\ht0=\ht2 #3\else #2\fi
  }
\newcommand*\rel@kern[1]{\kern#1\dimexpr\macc@kerna}
\newcommand*\widebar[1]{\@ifnextchar^{{\wide@bar{#1}{0}}}{\wide@bar{#1}{1}}}
\newcommand*\wide@bar[2]{\if@single{#1}{\wide@bar@{#1}{#2}{1}}{\wide@bar@{#1}{#2}{2}}}
\newcommand*\wide@bar@[3]{%
  \begingroup
  \def\mathaccent##1##2{%
    \let\mathaccent\save@mathaccent
    \if#32 \let\macc@nucleus\first@char \fi
    \setbox\z@\hbox{$\macc@style{\macc@nucleus}_{}$}%
    \setbox\tw@\hbox{$\macc@style{\macc@nucleus}{}_{}$}%
    \dimen@\wd\tw@
    \advance\dimen@-\wd\z@
    \divide\dimen@ 3
    \@tempdima\wd\tw@
    \advance\@tempdima-\scriptspace
    \divide\@tempdima 10
    \advance\dimen@-\@tempdima
    \ifdim\dimen@>\z@ \dimen@0pt\fi
    \rel@kern{0.6}\kern-\dimen@
    \if#31
      \overline{\rel@kern{-0.6}\kern\dimen@\macc@nucleus\rel@kern{0.4}\kern\dimen@}%
      \advance\dimen@0.4\dimexpr\macc@kerna
      \let\final@kern#2%
      \ifdim\dimen@<\z@ \let\final@kern1\fi
      \if\final@kern1 \kern-\dimen@\fi
    \else
      \overline{\rel@kern{-0.6}\kern\dimen@#1}%
    \fi
  }%
  \macc@depth\@ne
  \let\math@bgroup\@empty \let\math@egroup\macc@set@skewchar
  \mathsurround\z@ \frozen@everymath{\mathgroup\macc@group\relax}%
  \macc@set@skewchar\relax
  \let\mathaccentV\macc@nested@a
  \if#31
    \macc@nested@a\relax111{#1}%
  \else
    \def\gobble@till@marker##1\endmarker{}%
    \futurelet\first@char\gobble@till@marker#1\endmarker
    \ifcat\noexpand\first@char A\else
      \def\first@char{}%
    \fi
    \macc@nested@a\relax111{\first@char}%
  \fi
  \endgroup
}
\makeatother

\newcommand{\wbar}{\widebar}
\makeatother

\title[On one relaxation of the BLD condition]{On one relaxation\\of the bounded-length-distortion condition\\
in the context of metric measure spaces}

\author[R. D. Oleinik]{Roman D. Oleinik}

\address{{International School for Advanced Studies, Trieste, Italy}\newline
\indent {Moscow Institute of Physics and
Technology, Dolgoprudny, Russia}
}

\email{oleinik.r@phystech.edu}

\usepackage{tcolorbox}
\begin{document}
\begin{abstract}
We reformulate the bounded-length-distortion condition for maps between metric spaces in a certain relaxed form that requires the presence of a reference measure on the source space, which makes the new approach more natural from the perspective of maps from metric measure spaces to metric spaces. Within this framework, we establish some mapping results in an entirely singular setting of the following general structure: a metric measure space of finite Hausdorff dimension admits a map with the relaxed bounded-length-distortion condition into a finite-dimensional normed space.
\end{abstract}
\maketitle
\allowdisplaybreaks

\tableofcontents

\newpage
\section{Introduction}
\subsection{Background information}\label{ss:BackInf}
One of the most famous mathematical results is embodied by the Nash embedding theorems (established by J. Nash), a version of which (in a sense, not the most powerful but the most geometrical one) can be formulated in the following way (see \cite[Theorem 2]{N54}). For a smooth manifold $\mathsf{M}$ of dimension $d\in \Nat_0$ equipped with the metric $\mathsf{d}_{\mathsf{M}}$ induced by a $\mathrm{C}^1$-smooth Riemannian structure on $\mathsf{M}$, there exists a $\mathrm{C}^1$-smooth map $\varphi\colon \mathsf{M}\to \mathbb{R}^{2d+1}$ that is also a topological embedding such that
\begin{equation}
    \Len_{\mathsf{p}_2}(\varphi\circ \gamma)=\Len_{\mathsf{M}}(\gamma)\qquad \text{for every $\gamma\in \Cur(\mathsf{M})$}.
\end{equation}
Here and further, the following notation is adopted: given a metric space $\X=(\X,\dX)$, we denote, by $\Cur(\X)$, the family of all curves in $\X$, i.e., continuous maps from compact segments of the real line to $\X$, and, by $\Len_{\X}$, the length functional on $\Cur(\X)$ induced by $\dX$; also, given $D\in \Nat_0$, we mean, by $\R^D$, the $D$-dimensional real coordinate space and, given a norm $p$ on $\R^D$, by $\R^D_p$, the space $\R^D$ equipped with the metric induced by $p$ and, by $\Len_p$, the length functional $\Len_{\R^D_p}$; finally, given $r\in [1,+\infty]$, by $\mathsf{p}_r$, we understand the $\ell_r$-norm on finite-dimensional real coordinate spaces. Thus, the essence of the presented statement is that, from the perspective of lengths of curves, any Riemannian manifold can be viewed as an embedded submanifold of some Euclidean space.

One can abstract the conclusion of the provided theorem into the following notion (see, e.g., \cite[Section 1]{LD13}). For given metric spaces $\X=(\mathsf{X},\dX)$ and $\Y=(\Y,\dY)$, a map $\varphi\colon \X\to \Y$ is called a path-isometry if it is continuous and
\begin{equation}
    \Len_{\Y}(\varphi\circ \gamma)=\Len_{\X}(\gamma)\qquad \text{for every $\gamma\in \Cur(\X)$}.
\end{equation}
As follows directly from this definition, path-isometries form a (typically much larger) superclass of isometries, somewhat better adapted to the ``curvewise'' point of view. A side yet useful observation here is that, given that the source space of a path-isometry is, for instance, length, this map can be deduced to be $1$-Lipschitz.

One classical problem in the theory of metric spaces lies in determining whether a given metric space admits an isometry into a metric space from a certain class; a particular choice for such a one is that of Euclidean spaces. Replacing here correspondingly isometries with path-isometries, we arrive at another well-known problem, which we display separately below.
\begin{ques}\label{ques:PI}
    What metric spaces admit a path-isometry (with predetermined additional properties) into a Euclidean space?
\end{ques}
\noindent Accordingly, the Nash theorem in the indicated form answers the above question positively for smooth Riemannian manifolds equipped with the natural metric. Up to the present time, a considerable number of further results in this direction have been obtained. As one instance, we mention the following statement established by M. Gromov (see, e.g., \cite[Section 2.4.11]{G86}). For a given smooth manifold $\mathsf{M}$ of dimension $d\in \Nat_0$ equipped with the metric $\mathsf{d}_{\mathsf{M}}$ induced by a $\mathrm{C}^1$-smooth Riemannian structure on $\mathsf{M}$, there exists a path-isometry $\varphi\colon \mathsf{M}\to \mathbb{R}^d_{\mathsf{p}_2}$; note the same dimension here in contrast to the Nash theorem. As a further extension of that, A. Petrunin obtained a completely analogous conclusion for metrics on smooth manifolds induced by $\mathrm{C}^1$-smooth sub-Riemannian structures (see \cite[Theorem 1.2]{P11} and the discussion thereafter). In turn, a full sub-Riemannian version of exactly the Nash theorem, i.e., with the requirement for a map to be moreover a topological embedding, was then achieved by E. Le Donne (see \cite[Corollary 2.5]{LD13}).

At the same time, an adjacent problem in the above-mentioned setting consists in finding a corresponding biLipschitz map, rather than an isometry, of the desired form. Ideologically, one replaces here the preservation of the distances between points by a two-sided control on these distances. And an analogous shift can be done for path-isometries as well, which results in the following notion (see, e.g., \cite[Definition 2.10]{LD13}). Given metric spaces $\X=(\X,\dX)$ and $\Y=(\Y,\dY)$, a map $\varphi\colon \X\to \Y$ is called a bounded-length-distortion map, or a BLD map, if it is continuous and there is $c\in (0,+\infty)$ such that
\begin{equation}
    \frac{1}{c}\Len_{\X}(\gamma)\leq\Len_{\Y}(\varphi\circ \gamma)\leq c\Len_{\X}(\gamma)\qquad \text{for every $\gamma\in \Cur(\X)$}.
\end{equation}
Once this definition is formulated, by analogy with \cref{ques:PI}, the following problem naturally emerges.
\begin{ques}\label{ques:BLD}
    What metric spaces admit a BLD map (with predetermined additional properties) into a Euclidean space?
\end{ques}
\noindent As pointed out by E. Le Donne in \cite[Section 2]{LD13}, since any sub-Finsler metric on a compact smooth manifold is biLipschitz equivalent to a sub-Riemannian metric, any metric space of the former form (by the Petrunin theorem) admits a BLD map into a Euclidean space, the fact of which provides a broad class of spaces for which the answer to \cref{ques:BLD} is positive. As a potential far-reaching generalization of this observation, the following conjecture (within the very same work) was put forward.
\begin{conj}
    Any compact, length metric space of finite Hausdorff dimension admits a BLD map that is also a topological embedding into a Euclidean space.
\end{conj}
\noindent Under the stronger assumption that the space is of finite packing dimension, the relevant part here, once the result in \cite[Theorem 2.9]{LD13} of a more topological flavor is taken into account, is exactly the existence of a BLD map into a Euclidean space. And, to our knowledge, there has not been much progress in solving this problem beyond the smooth category, even under various additional restrictions on the source space.

As a clarification for the reader concerning what comes next, we underline the following circumstance. Although the above information is indicated here, it should be considered more as auxiliary data that is quite useful to keep in mind throughout the manuscript than as a direct topic of the discussion therein. Even more, the possibility of relating our eventual results with the presented matter was revealed mostly post factum, i.e., after the former ones were obtained in a slightly different guise. Still, we believe this overview is worth being exhibited at this juncture, as it allows one to view the forthcoming narrative in a much more general context.

\subsection{Framework and results}\label{ss:FrR} In the present work, we wish to address certain problems that can be viewed as loose variations of those described in the previous subsection. Specifically, we choose to endow a given source metric space with an additional structure, namely, to bring a measure thereon into play, and to relax desirable statements in a suitable way. This approach is strongly motivated by the success of recent decades in the theory of Sobolev spaces in the singular setting of metric measure spaces. For detailed information on the subject, we send the reader to \cite{GP20} and \cite{HKST15}, as well as other references therein.

In the Sobolev theory, especially in the context of metric measure spaces, one of the most effective methods of examining first-order properties of functions on a space consists in looking at their behavior along curves therein; for such a purpose though, only ``generic'', rather than all, curves are typically taken into account, which is a standard caveat regarding Sobolev functions. This ``genericity'' is measured (at least in the so-called Newtonian approach) through the notion of modulus, which forms a certain outer measure on the family of all curves in the space, which, in turn, depends itself on the chosen reference measure. In light of all this, it seems sufficiently natural to apply a similar relaxation to \cref{ques:PI} and \cref{ques:BLD}. Thus, unifying and generalizing both of these questions, we can formulate the following (quite vast) one, which should be considered primarily as a reference point for the subsequent exposition.
\begin{ques}\label{ques:GenBLD}
    Let $\X=(\X,\dX)$ be a metric space, let $\mathfrak{C}$ be an outer measure on $\Cur(\X)$. Let $a,b\in (0,+\infty)$, $d\in \mathbb{N}_0$, let $p$ be a norm on $\mathbb{R}^d$. Under what conditions on the objects involved does there exist a continuous map $\varphi\colon \X\to \mathbb{R}^d$ such that
    \begin{equation}\label{eq:GenBLD}
        a\Len_{\X}(\gamma)\leq \Len_p(\varphi\circ \gamma)\leq b\Len_{\X}(\gamma)\qquad \text{for $\mathfrak{C}$-a.e. $\gamma\in \Cur(\X)$}.
    \end{equation}
\end{ques}
\noindent Above, in line with the definitions of path-isometries and BLD maps in \cref{ss:BackInf}, the separate continuity condition is added to guarantee that the pre-compositions of the map with curves are still curves.

The formulation of \cref{ques:GenBLD} is clearly too broad to attack, so we start with some specification. First, for the rest of the subsection, let $\X=(\X,\dX)$ be a separable, locally complete metric space and let $\m$ be a locally finite Borel measure on $\X$; this provides us with the metric measure space $(\X,\m)$. And then, by $\Mod^{\m}_{\infty}$, we denote the $\infty$-modulus on $(\X,\m)$, which, as already mentioned, is a special outer measure on $\Cur(\X)$; its precise definition is given in \cref{ss:MgCf}. As the notation suggests, there are also other moduli, which are even more common in the corresponding literature, namely the moduli $\Mod_r^{\m}$, $r\in [1,+\infty)$; the chosen one is notably the finest among those, i.e., it has fewer negligible sets. With all this, we now possess all the ingredients required by \cref{ques:GenBLD}.

At the current stage, however, we are not able to provide relevant statements (under the selected entries) in the direction of \cref{ques:GenBLD} as it is. Nevertheless, once we introduce one more relaxation into the discussion, we can already achieve several curious results that still appear to be sufficiently useful. Specifically, we need to sacrifice the continuity property of the desired map, which, of course, forces us to make some corresponding modifications. Namely, since the pre-composition of a general map between metric spaces with a curve may not be a curve anymore, the length functional becomes somewhat unsuitably defined on such compositions. To resolve this issue, we appropriately extend this functional to ``discontinuous curves'', with the preservation of the previous notation though; see \cref{ss:MgCf} for the precise meaning behind that. This extension can be seen roughly as the standard total variation (which is known to make sense for arbitrary metric-valued maps defined on segments of the real line) from which the contribution of ``jumps'' (or any other kind of discontinuities) is excluded. In other words, with this choice and after a bit of reasoning, inequalities as in \eqref{eq:GenBLD} can be interpreted as a control on the change of the $1$-dimensional Hausdorff measure (counted with multiplicity) on the given curve under the given map; all this is elaborated further in \cref{ss:MgCf}. As the reader can detect therefrom, this series of simplifications turns the problem under consideration into an essentially ``local'' one, which is, in turn, predictably much easier to tackle.

In spite of the described adjustment, maps emerging within our results are not fully arbitrary and even possess a certain regularity. To be explicit, we are to deal with so-called equi-Luzin--Lipschitz maps, which can be seen as a weakened analog of Lipschitz maps; see \cref{def:LuzLip} for the precise meaning of that. In general, by a Luzin--Lipschitz map, one understands a map from a metric measure space to a metric space for which there is a countable decomposition of the source space into measurable subsets (up to a set of zero measure) such that the map is Lipschitz on each of these subsets. For instances of the use of this notion in such a context, we refer the reader to \cite[Section 2.1]{GT21} and \cite[Definition 2.11]{O25}. Strengthening this definition, one can say that such a map is equi-Luzin--Lipschitz if such a decomposition can be chosen so that, for some $L\in (0,+\infty)$, all the corresponding restrictions are $L$-Lipschitz. In this regard, it would be interesting to investigate whether the main theorems below, in which precisely equi-Luzin--Lipschitz maps appear, can be helpful in obtaining stronger statements with more regular maps involved, potentially with some extra assumptions on the source spaces.

With all of the above being said, we are ready to present the first main result of our manuscript. It can be (quite roughly) viewed as a nonsmooth adaptation to the chosen setting of the mentioned theorems by Gromov and Petrunin in the corresponding ``Finsler'' form. As an explanation for the latter term, we refer to the observation that these theorems, while originally established for Riemannian and sub-Riemannian manifolds, admit also some natural analogs in the context of Finsler and sub-Finsler manifolds; these analogs are notably irrespective of any ``Euclideanity'' of the source and target spaces. And basically the same attribute should be noticed to be inherited by the below statement as well.
\begin{theo}\label{theo:MainRes}
    Suppose the Hausdorff dimension of $\X$ is not greater than $d\in \mathbb{N}_0$. Then, for any $\varepsilon\in (0,1)$, there exists a $1$-Luzin--Lipschitz, Borel map $\varphi_{\varepsilon}\colon \X\to \mathbb{R}_{\mathsf{p}_{\infty}}^d$ with the following property: 
    \begin{equation}\label{eq:MainRes}
       \Len_{\mathsf{p}_{\infty}}(\varphi_{\varepsilon}\circ \gamma)\leq \Len_{\X}(\gamma)\leq (1+\varepsilon)\Len_{\mathsf{p}_1}(\varphi_{\varepsilon}\circ \gamma)\qquad \text{for $\Mod_{\infty}^{\m}$-a.e. $\gamma\in \Cur(\X)$}.
    \end{equation}
\end{theo}
\noindent We point the reader's attention to the fact that, in \eqref{eq:MainRes}, we violate the form suggested in \eqref{eq:GenBLD}, as the order of estimates is changed here; this is done for the sake of pure convenience. Also, in the formulation above, note the different norms on the left-hand and right-hand sides of \eqref{eq:MainRes}; such a choice is made intentionally to provide the optimal form of the corresponding bounds; of course, one could replace $\mathsf{p}_1$ with $\mathsf{p}_{\infty}$ on the right at the cost of introducing an additional factor $d$ there, which would return us to the symmetric formulation in \eqref{eq:GenBLD}.

Let us provide some remarks concerning the above theorem. The first point, which is of great importance in our opinion, is that, while the statement relies heavily on the presence of a reference measure, the condition on the source space itself is purely metric, i.e., it does not involve any assumptions on the measure (aside from some mild ones). And another point is that the statement is sharp with respect to the constants: a careful check of the case where the space $\R^d_{\mathsf{p}_1}$ is taken as $\X$ and the $d$-dimensional Lebesgue measure is taken as $\m$ should demonstrate (whenever $d>0$) that, once the expression on the left in \eqref{eq:MainRes} is multiplied by a constant greater than $1$, for some $\varepsilon\in (0,1)$, there will not exist a corresponding map. All this makes the result quite remarkable in its essence.

The conclusion of \cref{theo:MainRes} can be understood to ignore any possible ``intrinsic geometry'' of the source space; in particular, if this space admitted ``infinitesimal tangent norms'' at its points (in some adequate sense), the ``shapes'' of those (as indicated previously) would therefore be completely irrelevant. As is not difficult to track based on properties of linear maps between finite-dimensional normed spaces, this feature is achieved exactly via the choice of $\mathsf{p}_1$ and $\mathsf{p}_{\infty}$ in their positions. An emerging question then is whether we can replace both of these norms with the Euclidean one by requiring some extra assumptions on the underlying space; for this goal, it would be reasonable (from what we know in the smooth realm) to seek a condition imposing a sort of ``Euclideanity'' thereon. And, in fact, all this is perfectly feasible through the (slightly nonstandard) notion of infinitesimal Euclideanity; see \cref{def:InfEucl} for its precise meaning. It forms a very close analog (better adapted to our exposition) of the more common concept of infinitesimal Hilbertianity, for which we refer the reader to \cite[Section 4.3]{GP20}. As expected from the naming, the presence of this property forces those ``infinitesimal tangent norms'' (which, as elaborated later, indeed make sense in our framework) to be Euclidean, which is known to be a necessary attribute to distinguish ``Riemannian'' spaces within ``Finsler'' ones and hence seems to be precisely what we want.

Correspondingly, the second main result of the manuscript is given below. As discussed already, it constitutes essentially a nonsmooth version of the results by Gromov and Petrunin in the ``Riemannian'' form, i.e., with the presence of Euclideanity in the picture.
\begin{theo}\label{theo:MainResEucl}
       Suppose the Hausdorff dimension of $\X$ is not greater than $d\in \mathbb{N}_0$. Suppose the space $(\X,\m)$ is infinitesimally Euclidean. Then, for any $\varepsilon\in (0,1)$, there exists a $1$-Luzin--Lipschitz, Borel map $\varphi_{\varepsilon}\colon \X\to \mathbb{R}^d_{\mathsf{p}_{\infty}}$ with the following property:
     \begin{equation}\label{eq:MainResEucl}
       \Len_{\mathsf{p}_2}(\varphi_{\varepsilon}\circ \gamma)\leq \Len_{\X}(\gamma)\leq (1+\varepsilon)\Len_{\mathsf{p}_2}(\varphi_{\varepsilon}\circ \gamma)\qquad \text{for $\Mod_{\infty}^{\m}$-a.e. $\gamma\in \Cur(\X)$}.
    \end{equation}
\end{theo} 
\noindent Again, as with \eqref{eq:MainRes}, we indicate that the order of estimates in \eqref{eq:MainResEucl} is changed in comparison with \eqref{eq:GenBLD}. We also point the reader's attention to the choice of the norm for the target space of the map in the theorem.

We wish to make the following collateral remark. A straightforward difference of \cref{theo:MainRes} from \cref{theo:MainResEucl} is that the assumption on the source space in the latter case is no longer purely metric, as the presence of the infinitesimal Euclideanity may easily depend on the chosen measure. Still, if we look instead at the above-mentioned notion of infinitesimal Hilbertianity, notably there exists one known class of metric spaces, namely $\mathrm{CAT}$-spaces, any reasonable measure on which makes them possess the corresponding property; see \cite[Theorem 1.1]{DMGPS21} for this result. While (as we again emphasize for the reader) the two discussed notions are different, they represent the same basic idea, so we believe that an analogous statement on $\mathrm{CAT}$-spaces could be true for the one we adopt as well, which would return us to the nice situation with only metric requirements. We, however, leave this question outside the scope of our work.

Let us give several comments on our proof strategy for the listed theorems. As exactly equi-Luzin--Lipschitz maps (the definition of which is ``local'' in nature) are considered there, both statements basically reduce (after a bit of work) to the possibility of finding a countable cover of the source space by Borel sets (up to a set of zero measure) each of which admits a Lipschitz map from it to the considered finite-dimensional normed space that ``effectively'' tracks the speed of almost every curve passing through this set, with all of this appropriately quantified in a certain uniform way. With this being taken into account, we gain access to the fruitful machinery (developed by D. Bate, S. Eriksson-Bique, and E. Soultanis in \cite{BEBS24}) of so-called fragment-wise differentiable structures; see also \cite{EBS24} for the earlier work by the latter two authors on tightly connected curvewise differentiable structures. One of the key theorems there guarantees for a given metric measure space, under the only assumption of finite Hausdorff dimension of the underlying metric space, the existence of such a structure, which consists of a countable cover of the space (up to a set of zero measure) by special finite-dimensional charts that are, loosely speaking, of a form close to the above-described one but without any uniformity in the corresponding estimates (see \cite[Theorem 5.5]{BEBS24}). Thus, the essence of what we are to carry out in the current work is to construct such charts with the desired uniform control. For this we do the following: first, we perform a preparatory procedure of a similar flavor in the abstract setting of finite-dimensional normed modules (as such ones emerge naturally in the context under discussion) that consists in finding a countable cover of the space by measurable sets (up to a set of zero measure) admitting a certain ``optimal'' (from the perspective of how ``independent'' it is) module basis, which we call an Auerbach basis by analogy with the world of normed vector spaces, within which such bases serve precisely this role; and second, using certain approximation arguments within a suitably chosen normed module, specifically the module of Weaver derivations on the considered space, we obtain charts with the targeted properties. In the end, what we rely on is that, under the mentioned dimensional bound, the proposed module becomes finite-dimensional, the fact of which itself is a highly important consequence of the presence of a fragment-wise differentiable structure (see \cite[Theorem 1.6, Theorem 5.5]{BEBS24}). In other words, while the hidden hardest part of our exposition is established in the earlier literature (mostly within the work by Bate--Eriksson-Bique--Soultanis), many things need to be verified in order to draw the relevant conclusion, and it is exactly the execution of this that constitutes the core of our manuscript.

As a concluding comment here, let us note the following important circumstance. Even though the results given by \cref{theo:MainRes} and \cref{theo:MainResEucl} might fairly seem rather abstract, a ``localized form'' of the former has already found an application in the subsequent work \cite{O26} by the author of the current manuscript. Namely, a certain problem of characterizing metric-valued Sobolev mappings via Lipschitz post-compositions has been resolved therein precisely by exploiting the corresponding intermediate estimates presented in \cref{theo:FinDimBoundsGDis}. In fact, this very problem served as the initial motivation for all the analysis developed in our paper. In our opinion, this supplies a concrete justification for the relevance of the subject under consideration.

\section{Preliminaries}\label{ss:Prelim}

This section is aimed at presenting some preliminary information that is necessary for the subsequent discussion.

\subsection{General notation}\label{ss:GenNot} We start by introducing various notations that we will exploit throughout the manuscript.

We will adopt the following usual ``indeterminacy'' conventions: $(+\infty) \cdot 0=0\cdot(+\infty)=\frac{0}{0}=0$.

The symbols $\mathbb{N}$, $\Nat_0$, $\mathbb{Q}$, $\mathbb{R}$ will denote, respectively, the positive integers, the nonnegative integers, the rationals, the reals. In turn, the symbols $\overline{\R}$, $\overline{\R}_{\geq 0}$, $\R_{\geq 0}$, $\R_{>0}$ will mean, respectively, the sets $[-\infty,+\infty]$, $[0,+\infty]$, $[0,+\infty)$, $(0,+\infty)$. Also, given $d,d'\in \mathbb{N}_0$, we put $\overline{d,d'}\coloneqq [d,d']\cap \mathbb{N}_0$, with which we get $\overline{d,d'}=\varnothing$ whenever $d'<d$. Moreover, given $i,j\in \Nat_0$, we put
\begin{equation}
    \updelta(i,j)\coloneqq \begin{dcases}
        1,\qquad i=j,\\
        0,\qquad \text{otherwise}.
    \end{dcases}
\end{equation}

Given sets $S_0,S,S'$ with $S\subseteq S_0$ and a map $\theta\colon S_0\to S'$, by $\theta|_{S}$, we will mean the restriction of $\theta$ to $S$.

Throughout the manuscript, all appearing vector spaces, unless otherwise explicitly stated, will be assumed to be over the reals. Also, given a vector space $\mathbb{W}$, the symbol $\o_{\mathbb{W}}$ will always denote the corresponding zero vector.

In the paper we will constantly deal with finite systems of elements of a given vector space. It will be convenient to treat them equally, regardless of their size, so we will extend them to infinite sequences by adding the zero vectors. To this end, we introduce the following notation. Let $\mathbb{W}$ be a vector space. By $\Mat^{\infty}(\mathbb{W})$, we will mean the collection of all $\mathbb{W}$-valued sequences, i.e., maps $\mathbb{N}\to \mathbb{W}$, that are eventually equal to $\o_{\mathbb{W}}$; such sequences will be written interchangeably (depending on the context) either just as $w$ or (to highlight the presence of multiple components) in any of the forms $\bm{w}=(w_i)_{i\in \mathbb{N}}$ and $\bm{w}=(w^j)_{j\in \mathbb{N}}$. It is then clear that the collection $\Mat^{\infty}(\mathbb{W})$ forms a vector space under the componentwise operations; the corresponding zero vector will be denoted by $\bm{\o}_{\mathbb{W}}$. Also, given $d\in \mathbb{N}_0$, by $\Mat^d(\mathbb{W})$, we will denote the subcollection of $\Mat^{\infty}(\mathbb{W})$ made of all systems whose components not equal to $\o_{\mathbb{W}}$ are placed in the first $d$ positions; these can be seen to form vector subspaces of $\Mat^{\infty}(\mathbb{W})$. Moreover, given $W\subseteq \mathbb{W}$, by $\Mat^{\infty}(W)$ and $\Mat^d(W)$, $d\in \Nat_0$, we will denote the subcollections of, respectively, $\Mat^{\infty}(\mathbb{W})$ and $\Mat^d(\mathbb{W})$, $d\in \Nat_0$, made of all systems all nonzero elements of which belong to $W$. The following convention will be used as well: if we are given a set $S$ and a map $\theta\colon \mathbb{W}\to S$, then, given $\bm{w}\in \Mat^{\infty}(\mathbb{W})$, by $\theta(\bm{w})$, we will mean the system, i.e., a map $\Nat\to S$, obtained by applying $\theta$ to the components of $\bm{w}$. At the same time, such a notation should not be confused with the situation when a map defined on exactly systems of vectors is under consideration, which would have basically the same appearance; it should be clear from the context what is used at the current point.

Let $\mathbb{W}$, $\mathbb{W}'$, $\mathbb{W}''$ be vector spaces for which there is a given bilinear map
\begin{equation}
    (\cdot,\cdot)\colon \mathbb{W}\times \mathbb{W}'\to \mathbb{W}''.
\end{equation}
Given $\bm{w}\in \Mat^{\infty}(\mathbb{W})$ and $\bm{w}'\in \Mat^{\infty}(\mathbb{W}')$, we will use the notation
\begin{equation}
    \bm{w}\cdot \bm{w}'\coloneqq \sum\limits_{i\in \mathbb{N}}(w_i,w_i').
\end{equation}
In other words, whenever we have a bilinear map on the product of two vector spaces, we transfer it to the bilinear map defined on pairs of systems in the corresponding vector spaces. Note that the chosen notation, with $\cdot$, will be used regardless of the original notation for the map.

For brevity, we put $\R^{\infty}\coloneqq \Mat^{\infty}(\mathbb{R})$ and, for $d\in \mathbb{N}_0$, put $\R^d\coloneqq \Mat^d(\mathbb{R})$ and $\mathbb{Q}^d\coloneqq \Mat^d(\mathbb{Q})$. The latter collections will be occasionally identified with the standard coordinate spaces in the obvious way, which also puts thereon the corresponding topologies.

Let $d\in \mathbb{N}_0$. By $\GL_d$, we will denote the collection of all systems $\bm{u}\in \Mat^d(\mathbb{R}^d)$ the first $d$ elements of which form a basis of $\mathbb{R}^d$; such systems will be called $d$-bases, rather than merely bases, of $\mathbb{R}^d$ in order to indicate the relevant length of those (and a similar convention will be used throughout the paper). As the notation suggests, this object can be identified (which will be implicitly done everywhere further) with the general linear group of degree $d$ and hence possesses a natural structure of a smooth manifold; such ones, in turn, are known to be locally compact Polish spaces. We note that, with this choice of the topology on $\GL_d$, a sequence $(\bm{u}^n)_{n\in \mathbb{N}}\subseteq \GL_d$ converges to $\bm{u}\in \GL_d$ in $\GL_d$ if and only if, for each $i\in \overline{1,d}$, the sequence $(u^n_i)_{n\in \mathbb{N}}$ converges to $u_i$ in $\R^d$. Also, given $\bm{u}\in \GL_d$, by $\Dual_d(\bm{u})$, we will denote the unique system $\bm{\psi}\in \GL_d$ such that
\begin{equation}
    \psi^j\cdot u_i=\updelta(i,j)\qquad \text{for all $i,j\in \overline{1,d}$};
\end{equation}
in this case, the system $(\bm{u},\bm{\psi})$ will be called $d$-biorthogonal.

By $\mathsf{P}^{\infty}$, we will mean the collection of all seminorms on $\mathbb{R}^{\infty}$. Within it, as distinguished elements, we will deal with the $\ell_1$-norm $\mathsf{p}_1$, with the $\ell_2$-norm $\mathsf{p}_2$, and with the $\ell_{\infty}$-norm $\mathsf{p}_{\infty}$, defined in the usual way. Let $d\in \Nat_0$. By $\mathsf{P}^d$, we will denote the collection of all seminorms on $\mathbb{R}^d$, which can be naturally considered as a subcollection of $\mathsf{P}^{\infty}$. The function
\begin{equation}
    \mathsf{P}^d\times \mathsf{P}^d\to \R_{\geq 0}\colon (p,p')\mapsto \sup\limits_{w\in \mathbb{R}^d, \mathsf{p}_{\infty}(w)\leq 1} \big|p'(w)-p(w)\big|
\end{equation}
can be seen then to define a metric on $\mathsf{P}^d$, which, moreover, makes this collection into a separable, complete metric space. Also, by $\mathsf{P}^d_+$, we will denote the subcollection of $\mathsf{P}^d$ made of all norms on $\mathbb{R}^d$, which can be seen to be an open subset of $\mathsf{P}^d$. Then, given $p\in \mathsf{P}^d_+$, by $\R^d_p$, we will mean the space $\R^d$ equipped with the metric induced by $p$. Finally, given $p\in \mathsf{P}^d_+$, by $p^*$, we will denote the norm on $\R^d$ that is dual to $p$. 

Given a set $S$, by $\1_S$, we will mean ``the indicator function'' of $S$, which becomes a real function whenever the ambient set is understood.

Let $S$ be a set. By a measure on $S$, we will always mean a nonnegative, sigma-additive measure defined on a given sigma-algebra of $S$, elements of which will be called measurable. In particular, this term, unless otherwise explicitly stated, will never refer to the Carath{\'e}odory measurability.

For the sequel, let $X$ be a set.

Given a measure $\mu$ on $X$, the pair $(X,\mu)$ will be called a measure space; the measure space $(X,\mu)$ will be called sigma-finite if the measure $\mu$ is sigma-finite.

For what follows, let $\mu$ be a measure on $X$.

Let $S\subseteq X$. By $\P_{\mu}(S)$, we will mean the family of all $\mu$-measurable sets in $X$ that are subsets of $S$. By $\P_{\mu}^+(S)$, we will mean the family of all sets in $\P_{\mu}(S)$ of nonzero $\mu$-measure.

Given $E\in \P_{\mu}(X)$, by a $\mu$-partition of $E$, we will mean a disjoint family $(E_m)_{m\in \mathbb{N}}\subseteq \P_{\mu}(E)$ such that
\begin{equation}
    \mu\bigg(E\Big\backslash \bigcup\limits_{m\in \mathbb{N}} E_m\bigg)=0.
\end{equation}

By $\overline{\Leb}_{\mu}$, we will mean the collection of all equivalence classes up to the $\mu$-a.e. equality of all functions $\rho\colon X\to \overline{\R}$ such that, for some $\mu$-measurable function $\rho'\colon X\to \overline{\R}$, the functions $\rho$ and $\rho'$ coincide $\mu$-a.e. on $X$; moreover, all $\mu$-measurable functions $X\to \overline{\R}$ will be naturally considered as elements of this collection through passing to the corresponding equivalence class. In the usual way, we will treat elements of $\overline{\Leb}_{\mu}$ as $\mu$-a.e. defined $\overline{\R}$-valued functions on $X$. And, given $\varrho\in \overline{\Leb}_{\mu}$, a function $\rho\colon X\to \overline{\R}$ will be called a $\mu$-representative of $\varrho$ if $\rho\in \varrho$, i.e., if the function $\rho$ is equal $\mu$-a.e. on $X$ to $\varrho$. Also, by $\Leb_{\mu}$ and $\Leb_{\mu}^{\infty}$, we will denote the subcollections of $\overline{\Leb}_{\mu}$ made, respectively, of all $\mu$-a.e. finite functions and of all $\mu$-essentially bounded functions, both of which form $\R$-algebras in the obvious way. The function in $\Leb^{\infty}_{\mu}$ that equals zero $\mu$-a.e. on $X$ will be denoted by $\o_{\Leb_{\mu}}$, in agreement with the notation for the zero vector in $\Leb_{\mu}$. Next, by $[\overline{\Leb}_{\mu}]_{\geq 0}$ and $[\Leb_{\mu}]_{\geq 0}$, we will denote the subcollections of, respectively, $\overline{\Leb}_{\mu}$ and $\Leb_{\mu}$ made of all $\mu$-a.e. nonnegative functions. Finally, given $\varrho\in \overline{\Leb}_{\mu}$, by $|\varrho|_{\Leb_{\mu}}$, we will denote the function in $[\overline{\Leb}_{\mu}]_{\geq 0}$ obtained by taking $\mu$-a.e. on $X$ the absolute value of $\varrho$, and, by $\|\varrho\|_{\Leb^{\infty}_{\mu}}$, we will denote the $\mu$-essential supremum of $|\varrho|_{\Leb_{\mu}}$ as valued in $\overline{\R}_{\geq 0}$.

Given $\bm{\varrho}\in \Mat^{\infty}(\Leb_{\mu})$ and $p\in \mathsf{P}^{\infty}$, by $p(\bm{\varrho})$, we will mean the function in $\Leb_{\mu}$ obtained by taking $\mu$-a.e. on $X$ the $p$-seminorm of $\bm{\varrho}$.

Given $\varrho_1,\varrho_2\in \overline{\Leb}_{\mu}$, by $\varrho_1\vee\varrho_2$ and $\varrho_1\wedge \varrho_2$, we will mean the functions in $\overline{\Leb}_{\mu}$ obtained by taking $\mu$-a.e. on $X$, respectively, the pointwise maximum and minimum of $\varrho_1$ and $\varrho_2$; these operations will also be called, respectively, join and meet.

We will need also the standard concepts of join and meet of an arbitrary family of measurable functions. For that, we formulate the separate statement below, only for the former case, as the latter one is completely analogous.
\begin{prop}
Suppose the measure space $(X,\mu)$ is sigma-finite. Let $\Pi$ be a nonempty index set, let $(\varrho_{\pi})_{\pi\in \Pi}\subseteq \overline{\Leb}_{\mu}$. Then there exists a unique function $\varrho\in \overline{\Leb}_{\mu}$, denoted by
\begin{equation}
    \bigvee\limits_{\pi\in \Pi} \varrho_{\pi},
\end{equation}
satisfying the following properties:
\begin{itemize}
    \item[$\rm I)$] it holds, for any $\pi\in \Pi$, that
    \begin{equation}
        \varrho_{\pi}\leq \varrho\qquad \text{$\mu$-a.e. on $X$};
    \end{equation}
    \item[$\rm II)$] given any $\widetilde{\varrho}\in \overline{\Leb}_{\mu}$ such that, for any $\pi\in \Pi$, one has
    \begin{equation}
        \varrho_{\pi}\leq \widetilde{\varrho}\qquad \text{$\mu$-a.e. on $X$},
    \end{equation}
it holds that
    \begin{equation}
        \varrho\leq \widetilde{\varrho}\qquad \text{$\mu$-a.e. on $X$}.
    \end{equation}
\end{itemize}
Moreover, there is a family $(\pi_l)_{l\in \mathbb{N}}\subseteq \Pi$ such that
\begin{equation}
    \varrho(x)=\sup\limits_{l\in \mathbb{N}} \varrho_{\pi_l}(x)\qquad \text{for $\mu$-a.e. $x\in X$}.
\end{equation}
\end{prop}
\noindent In addition, under the same premise as above, by
\begin{equation}
    \bigwedge\limits_{\pi\in \Pi} \varrho_{\pi},
\end{equation}
we will denote the corresponding meet, for which the conclusion is formed analogously by duality. The given statement, as well as its direct consequences, will be used without a specific reference further. Also, to be accurate, if we are given a family of $\mu$-measurable functions $X\to \overline{\R}$, then, by its join and meet, we will mean, respectively, the join and meet of the family of the corresponding equivalence classes up to the $\mu$-a.e. equality of the initial functions.

\subsection{On modules over measure spaces}\label{ss:ModMS} Now we want to give a piece of necessary information about modules over measure spaces.

For the current subsection, we fix a sigma-finite measure space $(\X,\m)$.

    \begin{rem}\label{rem:ModTerm}
       Let us make a comment concerning the terminology used in the context of modules over measure spaces. While the definition we are about to give complies completely with the algebraic notion of module over a ring, in the setting of rings of measurable functions on measure spaces, one typically brings into play extra assumptions (like the locality and gluing conditions, as axiomatized in \cite[Definition 1.2.1]{G18}) to avoid certain pathologies. We, however, will not face any issues in this direction, so we choose a simpler (and more general) definition to work with.
    \end{rem}

Let $\mathcal{A}$ denote here either $\Leb_{\m}$ or $\Leb^{\infty}_{\m}$. By an $\mathcal{A}$-module, we will mean a vector space $\mathcal{M}$ equipped with a bilinear multiplication map
    \begin{equation}
        \mathcal{A}\times \mathcal{M}\to \mathcal{M}\colon (\varrho,w)\mapsto \varrho w
    \end{equation}
    satisfying the following conditions:
    \begin{equation}
        (\varrho' \varrho)w=\varrho'(\varrho w)\qquad \text{for any $\varrho,\varrho'\in \mathcal{A}$ and any $w\in \mathcal{M}$};
    \end{equation}
    \begin{equation}
        \1_{\X}w=w\qquad \text{for any $w\in \mathcal{M}$}.
        \end{equation}
    Of course, such a definition would make perfect sense if one considered instead of $\mathcal{A}$ an arbitrary subalgebra of $\Leb_{\m}$ containing the constant functions. We, however, do not need this level of generality at all.

    We point out that $\Leb_{\m}$-modules can be seen to form a subclass of $\Leb_{\m}^{\infty}$-modules. Thereby, all the further definitions, while given only for the latter ones, will be used for the former ones as well with the obvious meaning. So, for what follows, we fix an $\Leb_{\m}^{\infty}$-module $\mathcal{M}$.

    As usual in the setting of modules over measure spaces, we will adopt the following convention: given $E\in \P_{\m}(\X)$ and $w,w'\in \mathcal{M}$, we write
    \begin{equation}
        w=w'\qquad \text{on $E$}
    \end{equation}
    to mean
    \begin{equation}
        \1_E w=\1_E w'.
    \end{equation}

Let $E\in\P_{\m}(\X)$, $d\in \mathbb{N}_0$, $\bm{w}\in \Mat^{\infty}(\mathcal{M})$. The system $\bm{w}$ will be called $d$-independent over $E$ if, given any $\bm{\varrho}\in \Mat^d(\Leb_{\m}^{\infty})$ with
\begin{equation}
    \bm{\varrho} \cdot \bm{w}=\o_{\mathcal{M}}\qquad \text{on $E$},
\end{equation}
we have
\begin{equation}
    \bm{\varrho}=\bm{\o}_{\Leb_{\m}}\qquad \text{$\m$-a.e. on $E$}.
\end{equation}
The system $\bm{w}$ will be called a $d$-basis of $\mathcal{M}$ over $E$ if it is $d$-independent over $E$ and, for any $w\in \mathcal{M}$, there is an $\m$-partition $(E_m)_{m\in \mathbb{N}}$ of $E$ and there are $\bm{\varrho}_m\in \Mat^d(\Leb_{\m}^{\infty})$, $m\in \mathbb{N}$, such that, for each $m\in \Nat$, we have
\begin{equation}
    w=\bm{\varrho}_m\cdot \bm{w}\qquad \text{on $E_m$}.
\end{equation}
We then make the following simple but important observation that is implied by the given definitions: if the system $\bm{w}$ is $d$-independent over $E$ and there is no $E'\in \P_{\m}^+(E)$ such that the module $\mathcal{M}$ admits a $(d+1)$-independent system over $E'$, then the system $\bm{w}$ forms a $d$-basis of $\mathcal{M}$ over $E$.

Given $E\in \P_{\m}(\X)$, by the dimension of $\mathcal{M}$ over $E$, denoted by $\dim_{\mathcal{M}}(E)$, we will mean the supremum of all $d\in \mathbb{N}_0$ for which there is $E'\in \P_{\m}^+(E)$ such that the module $\mathcal{M}$ admits a $d$-independent system over $E'$. In particular, for $E\in \P_{\m}(\X)$ with $\m(E)=0$, we have $\dim_{\mathcal{M}}(E)=-\infty$.

The following well-known fact about the dimensional decomposition of the space will be convenient to use.
\begin{prop}\label{prop:DimDecompMod}
    Let $E\in \P_{\m}(\X)$. Suppose $\dim_{\mathcal{M}}(E)<+\infty$. Then there exists an $\m$-partition $(E_m)_{m\in \mathbb{N}}$ of $E$ with the following property: for each $m\in \mathbb{N}$ with $\m(E_m)>0$, there is $d_m\in \mathbb{N}_0$ with $d_m\leq \dim_{\mathcal{M}}(E)$ such that the module $\mathcal{M}$ admits a $d_m$-basis over $E_m$.
    \begin{proof}
        The statement can be seen to be a particular case of \cite[Proposition 3.1.1]{K04}. More specifically, using the sigma-finiteness of $\m$, one can recursively exhaust $E$ (unless $\m(E)=0$, in which case everything trivializes) by $\m$-measurable sets over which the module $\mathcal{M}$ admits a basis (up to an $\m$-negligible set). In similar terms, such a reasoning can be found in \cite[Proposition 1.4.5]{G18}. Thereby, we omit further details.
    \end{proof}
\end{prop}

We will need the following terminology concerning subsets of modules. Let $W\subseteq \mathcal{M}$. The set $W$ will be called convex if it is convex as a subset of the underlying vector space of $\mathcal{M}$, i.e., one has
\begin{equation}
    \eta w+(1-\eta)w'\in W\qquad \text{for any $w,w'\in W$ and any $\eta\in [0,1]$}.
\end{equation}
Likewise, the set $W$ will be called disked if it is disked (or absolutely convex in the more common terminology) as a subset of the underlying vector space of $\mathcal{M}$, i.e., it is convex and it holds that
\begin{equation}
    -w\in W\qquad \text{for any $w\in W$}.
\end{equation}
Moreover, we adopt the namings suggested in \cite[Definition 2.1]{CKV15} and introduce the following definitions: the set $W$ will be called stable if
\begin{equation}
    \1_E w+\1_{(\X\backslash E)} w'\in W\qquad \text{for any $w,w'\in W$ and any $E\in \P_{\m}(\X)$};
\end{equation}
the set $W$ will be called $\Leb_{\m}$-convex if, given any $\varrho\in \Leb_{\m}$ with
\begin{equation}
    0\leq \varrho\leq 1\qquad \text{$\m$-a.e. on $\X$},
\end{equation}
we have
\begin{equation}
    \varrho w+(\1_{\X}-\varrho)w'\in W\qquad \text{for any $w,w'\in W$}.
\end{equation}

Let us note here that the vector spaces $\Mat^{\infty}(\mathcal{M})$ and $\Mat^d(\mathcal{M})$, $d\in \Nat_0$, can be seen to possess a natural structure of $\Leb^{\infty}_{\m}$-modules, the fact of which we will always keep in mind.

Further, given $E\in \P_{\m}(\X)$ and $W\subseteq \mathcal{M}$, we will occasionally consider the subset of $\mathcal{M}$ of the form
    \begin{equation}
        W|_E\coloneqq \Big\{\1_E w\bigm|w\in  W\Big\},
    \end{equation}
    which, in the case $W=\mathcal{M}$, can be seen to form an $\Leb^{\infty}_{\m}$-module in its own right.

We now move to the notion of normed modules. Let $\mathcal{A}$ denote here either $\Leb_{\m}$ or $\overline{\Leb}_{\m}$. By an $\mathcal{A}$-norm on $\mathcal{M}$, we will mean a function $\Norm\colon \mathcal{M}\to [\mathcal{A}]_{\geq 0}$ satisfying the following conditions:
\begin{gather}
    \begin{gathered}
    \Norm(w_1+w_2)\leq \Norm(w_1)+\Norm(w_2)\qquad \text{$\m$-a.e. on $\X$}\\
    \text{for any $w_1,w_2\in \mathcal{M}$};
    \end{gathered}\\
    \begin{gathered}
    \Norm(\varrho w)=|\varrho|_{\Leb_{\m}} \Norm(w)\qquad \text{$\m$-a.e. on $\X$}\\
    \text{for any $w\in \mathcal{M}$ and any $\varrho\in \Leb_{\m}^{\infty}$};
    \end{gathered}\\
    \begin{gathered}
    \Norm(w)=\o_{\Leb_{\m}}\qquad \text{$\m$-a.e. on $\X$}\\
    \Downarrow\\
    w=\o_{\M}\\
    \text{for any $w\in \mathcal{M}$}.
    \end{gathered}
\end{gather}
Correspondingly, given an $\mathcal{A}$-norm $\Norm $ on $\mathcal{M}$, the pair $(\mathcal{M},\Norm)$ will be called an $\mathcal{A}$-normed module.

\begin{rem}
    Again, as already mentioned in \cref{rem:ModTerm}, our terminology for modules, normed modules in particular, differs from the standard one in various aspects. We point the reader's attention to this circumstance.
\end{rem}

Clearly, any $\Leb_{\m}$-normed module is also an $\overline{\Leb}_{\m}$-normed module in the obvious way, so all definitions that make sense for the latter ones do so for the former ones. We, however, will eventually work entirely with $\Leb_{\m}$-normed modules.

For what follows, let $\Norm$ be an $\overline{\Leb}_{\m}$-norm on $\mathcal{M}$.

We will occasionally use the following notation:
\begin{equation}
 [\mathcal{M},\Norm]_{\leq 1}\coloneqq \Big\{w\in \mathcal{M} \bigm| \text{$\big(\Norm(w)\big)(x)\leq 1$ for $\m$-a.e. $x\in \X$} \Big\}.
\end{equation}
Whenever the norm $\Norm$ is tightly connected with $\mathcal{M}$, we will omit its indication and write the introduced term simply as $[\mathcal{M}]_{\leq 1}$.

Let $d\in \Nat_0$, $p\in \mathsf{P}_+^d$. Slightly abusing notation, given $\bm{w}\in \Mat^d(\mathcal{M})$, we put
\begin{equation}
    p(\bm{w})\coloneqq \bigvee\limits_{\bm{\varrho}\in \R^d, p^*(\bm{\varrho})\leq 1} \Norm(\bm{\varrho}\cdot \bm{w}).
\end{equation}
This assignment, which will be still denoted by $p$ for simplicity, can be easily seen to define on $\Mat^d(\mathcal{M})$ an $\overline{\Leb}_{\m}$-norm, which, whenever the norm $\Norm$ is an $\Leb_{\m}$-norm, turns into an $\Leb_{\m}$-norm.

Let $E\in \P_{\m}(\X)$. The normed module $(\mathcal{M},\Norm)$ will be called Euclidean over $E$ whenever there exists a map
   \begin{equation}
       \langle \cdot,\cdot \rangle_{\Norm,E}\colon \mathcal{M}|_E\times \mathcal{M}|_E \to \Leb_{\m},
   \end{equation}
   which will be called the $\Leb_{\m}$-inner-product for $\Norm$ over $E$, with the following properties:
   \begin{gather}
   \begin{gathered}
       \big\langle w, \varrho_1 w_1+\varrho_2 w_2\big\rangle_{\Norm,E}=\varrho_1\langle w,  w_1\rangle_{\Norm,E}+\varrho_2\langle w,  w_2\rangle_{\Norm,E}\qquad \text{$\m$-a.e. on $\X$}\\ \text{for any $\varrho_1,\varrho_2\in \Leb_{\m}^{\infty}$ and any $w,w_1,w_2\in \mathcal{M}|_E$}; \end{gathered}\\
       \begin{gathered}
           \langle w,w'\rangle_{\Norm,E}=\langle w',w\rangle_{\Norm,E}\qquad \text{$\m$-a.e. on $\X$}\\
           \text{for any $w,w'\in \mathcal{M}|_E$}; \end{gathered}\\
           \begin{gathered}\label{eq:InProdNorm}
       \langle w,w\rangle_{\Norm,E}=\big(\Norm(w)\big)^2\qquad \text{$\m$-a.e. on $\X$}\\
       \text{for any $w\in \mathcal{M}|_E$}.
       \end{gathered}
   \end{gather}
   Clearly, such a map is necessarily unique. As a direct consequence of the given definition, a Euclidean module over $\X$ is always an $\Leb_{\m}$-normed module. Also, if the normed module $(\mathcal{M}, \Norm)$ is Euclidean over $E$, a system $\bm{w}\in \Mat^{\infty}(\mathcal{M})$ will be called orthogonal over $E$ if it holds, for all $i,i'\in \Nat$ with $i\neq i'$, that
   \begin{equation}
       \big\langle \1_E w_i,\1_E w_{i'}\big\rangle_{\Norm,E}=0\qquad \text{$\m$-a.e. on $E$}
   \end{equation}
   and, given $d\in \Nat_0$, will be called $d$-orthonormal over $E$ if it is orthogonal over $E$ and it holds, for each $i\in \overline{1,d}$, that
   \begin{equation}
       \Norm(w_i)=1\qquad \text{$\m$-a.e. on $E$}.
   \end{equation}

   \begin{rem}
       The notion of Euclidean module should be compared with that of Hilbert module (see, for instance, \cite[Definition 1.2.20]{G18}). Up to the already mentioned caveat about our choice of the terminology for modules, a Hilbert module is nothing but a complete (with respect to the natural metric) Euclidean module.
   \end{rem}

We now explicitly record the following fact about existence of orthonormal bases for Euclidean modules.
   \begin{prop}\label{prop:ExOrthBasEucl}
       Let $E\in \mathrm{P}_{\m}(\X)$. Suppose $\dim_{\mathcal{M}}(E)<+\infty$. Suppose the module $(\mathcal{M},\Norm)$ is Euclidean over $E$. Then there exists an $\m$-partition $(E_m)_{m\in \Nat}$ of $E$ with the following property: for each $m\in \Nat$ with $\m(E_m)>0$, there is $d_m\in \Nat_0$ with $d_m\leq \dim_{\mathcal{M}}(E)$ such that the module $\mathcal{M}$ admits a $d_m$-orthonormal $d_m$-basis over $E_m$.
       \begin{proof}
           The proof can be carried out in essentially the same way as in the realm of Euclidean spaces, i.e., by means of the Gram-Schmidt process. An explicit realization of this scheme, under slightly different definitions though, can be found, for instance, in \cite[Theorem 1.4.11]{G18}. We therefore omit further details for brevity.
       \end{proof}
   \end{prop}

  For the sequel, it is convenient to record the following simple statement.
   \begin{prop}\label{prop:EstNormAbsNorm}
       Let $d\in \Nat_0$, let $\bm{w}\in \Mat^d(\mathcal{M})$ be such that $\Norm(\bm{w})\in \Mat^d(\Leb_{\m})$. Then the assertions below are valid.
       \begin{enumerate}
           \item[$\rm I)$] It holds that
           \begin{equation}
               \mathsf{p}_{\infty}(\bm{w})\leq \mathsf{p}_{\infty}\big(\mathsf{N}(\bm{w})\big)\qquad \text{$\m$-a.e. on $\X$}.
           \end{equation}
           \item[$\rm II)$] Let $E\in \P_{\m}(\X)$. Suppose the module $(\mathcal{M},\Norm)$ is Euclidean over $E$ and suppose the system $\bm{w}$ is orthogonal over $E$. Then it holds that
         \begin{equation}
               \mathsf{p}_2(\bm{w})\leq \mathsf{p}_{\infty}\big(\mathsf{N}(\bm{w})\big)\qquad \text{$\m$-a.e. on $E$}.
           \end{equation}
       \end{enumerate}
        \begin{proof}
               The assertions follow almost directly from all the relevant definitions, so we omit a detailed proof.
           \end{proof}
   \end{prop}

   We now turn to the notion of dual module and related definitions. For that, let $\M$ be an $\Leb_{\m}^{\infty}$-module.

By the dual module of $\M$, we will mean the $\Leb_{\m}$-module, denoted by $\M^{\star}$, obtained as follows. As a set, we declare $\M^{\star}$ to consist of all $\Leb_{\m}^{\infty}$-linear maps from $\M$ to $\Leb_{\m}$, that is, all linear maps $\phi\colon \M\to \Leb_{\m}$ such that
\begin{equation}
    \phi(\xi v)=\xi\phi(v)\qquad \text{for any $\xi\in \Leb_{\m}^{\infty}$ and any $v\in \M$}.
\end{equation}
This set can then be seen to possess a natural structure of an $\Leb_{\m}$-module.

\begin{rem}
   As one more warning for the reader, we note that (far more often) a different definition of the dual module is used (see, for instance, \cite[Definition 1.2.6]{G18}). We, however, prefer the given one for our exposition, particularly since it is purely algebraic, i.e., does not require any fixed norm on the module. Thereby, we introduced the notation for the dual module in a different style, with $\star$ rather than $\ast$, as this should prevent any confusion. Nevertheless, once we restrict ourselves to finite-dimensional modules (which we will do eventually), there will be no conceptual difference between the proposed approaches.
\end{rem}

Let $E\in \P_{\m}(\X)$, $d\in \mathbb{N}_0$, $\bm{v}\in \Mat^{\infty}(\M)$, $\bm{\phi}\in \Mat^{\infty}(\M^{\star})$. The system $(\bm{v},\bm{\phi})$ will be called $d$-biorthogonal over $E$ if, for all $i,j\in \overline{1,d}$, we have
\begin{equation}
    \phi^j(v_i)=\updelta(i,j)\qquad \text{$\m$-a.e. on $E$}.
\end{equation}

We record the version for modules of the standard fact that any basis of a finite-dimensional vector space admits a unique dual system.
\begin{prop}\label{prop:ExistOfFunc}
    Let $E\in \mathrm{P}_{\m}(\X)$, $d\in \mathbb{N}_0$. Let $\bm{u}\in \Mat^d(\M)$ be a system forming a $d$-basis of $\M$ over $E$. Then there exists a unique system $\bm{\psi}\in \Mat^d\big(\M^{\star}|_E\big)$ such that the system $(\bm{u},\bm{\psi})$ is $d$-biorthogonal over $E$; the system $\bm{\psi}$, additionally, forms a $d$-basis of $\M^{\star}$ over $E$.
    \begin{proof}
        The proof is essentially the same as in the setting of vector spaces, so we omit the details.
    \end{proof}
\end{prop}
\noindent The system guaranteed by the above statement will be called the $d$-dual system of $\bm{u}$ over $E$ and denoted by $\Dual_d^E(\bm{u})$.

We will tacitly use the following consequence of \cref{prop:ExistOfFunc}. Given $E\in \P_{\m}(\X)$, $d\in \Nat_0$, $\bm{u}\in \Mat^d(\M)$ such that the system $\bm{u}$ forms a $d$-basis of $\M$ over $E$, putting $\bm{\psi}\coloneqq \Dual^E_d(\bm{u})$, we have the following: for any $v\in \M$ there is an $\m$-partition $(E_m)_{m\in \Nat}$ of $E$ such that, for each $m\in \Nat$, one has $\1_{E_m}\bm{\psi}(v)\in \Mat^d(\Leb_{\m}^{\infty})$ and
\begin{equation}
    v=\big(\1_{E_m}\bm{\psi}(v)\big)\cdot \bm{u}\qquad \text{on $E_m$}.
\end{equation}

For what follows, let $|\cdot|_{\M}$ be an $\Leb_{\m}$-norm on $\M$, which provides us with the $\Leb_{\m}$-normed module $\M=\big(\M,|\cdot|_{\M}\big)$. 

We equip $\M^{\star}$ with a natural $\overline{\Leb}_{\m}$-norm as follows. We consider the function $|\cdot|_{\M^{\star}}\colon \M^{\star}\to [\overline{\Leb}_{\m}]_{\geq 0}$, which will be called the dual norm to $|\cdot|_{\M}$, given for any $\phi\in \M^{\star}$ by
\begin{equation}
    |\phi|_{\M^{\star}}\coloneqq \bigvee\limits_{v\in [\M]_{\leq 1}} \big|\phi(v)\big|_{\Leb_{\m}}.
\end{equation}
It can be then easily checked that this function is indeed an $\overline{\Leb}_{\m}$-norm on $\M^{\star}$. Thus, we get at our disposal the $\overline{\Leb}_{\m}$-normed $\Leb_{\m}$-module $\M^{\star}=\big(\M^{\star},|\cdot|_{\M^{\star}}\big)$. Next, while it is always possible to put on a given normed module a natural topology, we shall use it only for the dual module case. Namely, we endow $\M^{\star}$ with the topology given by the translation-invariant pseudometrics
\begin{equation}
\begin{gathered}
    \M^{\star}\times \M^{\star}\to \R_{\geq 0}\colon (\phi,\phi')\mapsto \int\limits_{\X} \Big(\1_E\wedge |\phi'-\phi|_{\M^{\star}}\Big) \D \m,\\
    \text{$E\in \P_{\m}(\X)$, $\m(E)<+\infty$}.
    \end{gathered}
\end{equation}
This topology can be verified, via the sigma-finiteness of $\m$, to make $\M^{\star}$ into a complete uniform space. In practice, as one can check, this choice of the topology on $\M^{\star}$ implies that a sequence $(\phi_n)_{n\in \mathbb{N}}\subseteq \M^{\star}$ converges to $\phi\in \M^{\star}$ in $\M^{\star}$ if and only if the sequence $\big(|\phi_n-\phi|_{\M^{\star}}\big)_{n\in \mathbb{N}}$ converges to $\o_{\Leb_{\m}}$ locally in $\m$-measure on $\X$. After that, given $d\in \Nat_0$, we also put on $\Mat^d(\M^{\star})$ the natural product topology, which effectively means for a sequence $(\bm{\phi}_n)_{n\in \mathbb{N}}\subseteq \Mat^d(\M^{\star})$ that the following conditions are equivalent: a sequence $(\bm{\phi}_n)_{n\in \mathbb{N}}\subseteq \Mat^d(\M^{\star})$ converges to $\bm{\phi}\in \Mat^d(\M^{\star})$ in $\Mat^d(\M^{\star})$; for each $j\in \overline{1,d}$, the sequence $(\phi_n^j)_{n\in \mathbb{N}}$ converges to $\phi^j$ in $\M^{\star}$; for any $p\in \mathsf{P}^d_+$, the sequence $\big(p(\bm{\phi}_n-\bm{\phi})\big)_{n\in \Nat}$ converges to $\o_{\Leb{\m}}$ locally in $\m$-measure on $\X$. Finally, given $E\in \P_{\m}(\X)$, $d\in \Nat_0$, and $\bm{\Phi}\subseteq \Mat^d(\M^{\star})$, by $\Cl_E(\bm{\Phi})$, we will denote the closure in $\Mat^d(\M^{\star})$ of the family $\bm{\Phi}|_E$.

\begin{rem}
    An important consequence of \cref{prop:ExistOfFunc} is that the dual norm for a finite-dimensional module is in fact an $\Leb_{\m}$-norm, which will always be implicitly used.
\end{rem}

The following terminology, adopted from the setting of normed vector spaces, will be used later. Given $E\in \mathrm{P}_{\m}(\X)$, a set $\Phi\subseteq [\M^{\star}]_{\leq 1}$ will be called norming for $\M$ over $E$ if, given any $v\in \M$, we have
     \begin{equation}
        |v|_{\M}=\bigvee\limits_{\phi\in \Phi} \big|\phi(v)\big|_{\Leb_{\m}}\qquad \text{$\m$-a.e. on $E$}.
    \end{equation}

As we adopt highly nonstandard terminology in our work, we record separately the following proposition, which says essentially that the dual of a Euclidean finite-dimensional module is again Euclidean. Of course, a similar statement, with appropriate modifications, for infinite-dimensional modules is also valid (see, for instance, \cite[Proposition 1.2.24]{G18}) but would require invoking additional machinery, so we prefer to stick with the simpler version below.
   \begin{prop}\label{prop:EuclOrthDualOrth}
       Let $E\in \P_{\m}(\X)$. Suppose $\dim_{\M}(E)<+\infty$. Suppose the module $\M$ is Euclidean over $E$. Then the module $\M^{\star}$ is also Euclidean over $E$. In addition, given $d\in \Nat_0$, if a system $\bm{u}\in \Mat^d(\M)$ forms an orthogonal $d$-basis of $\M$ over $E$, then the system $\Dual_d^E(\bm{u})$ forms an orthogonal $d$-basis of $\M^{\star}$ over $E$.
\begin{proof}
   The proof is quite elementary and is similar to what one would do in the setting of finite-dimensional vector spaces, so we present only its brief outline. To start with, one needs to apply \cref{prop:ExOrthBasEucl} to find an $\m$-partition $(E_m)_{m\in \mathbb{N}}$ of $E$ such that, for each $m\in \Nat$ with $\m(E_m)> 0$, the module $\M$ admits an orthonormal basis, say $\bm{u}^m$, over $E_m$. After that, one needs to define the desired $\Leb_{\m}$-inner-product for $|\cdot|_{\M^{\star}}$ over $E$ by declaring, for each $m\in \Nat$ with $\m(E_m)> 0$, the dual system of $\bm{u}_m$ to be an orthonormal basis of $\M^{\star}$ over $E_m$. This can be checked to determine the corresponding structure completely. As the last step, one needs to verify the condition in \eqref{eq:InProdNorm}, which holds due to the orthogonality of the systems and their dual systems. In turn, the second part of the statement follows easily from the first one.
\end{proof}   
   \end{prop}

\section{On Auerbach bases}\label{ss:AueBas}

In this section, we discuss the notion of so-called Auerbach bases, first in the classical case of vector spaces and then in the context of normed modules. This concept, while not central, and not even fully necessary as mentioned in \cref{rem:RedAueBas}, for the current work, will be used in the sequel, in \cref{ss:ADF} specifically, as an important technical tool for the corresponding proofs. To our knowledge, in the literature this topic has never been addressed in such a framework.
\subsection{Classical notion and auxiliary facts} We start by recalling the classical definition of Auerbach bases and verifying several related statements.

 As can be derived from the definition below, Auerbach bases represent a natural generalization of orthonormal bases in Euclidean spaces to arbitrary normed spaces. We cover only the finite-dimensional case, as it is much simpler and completely suffices for our needs. For a more detailed exposition of this matter, we refer the reader to \cite[Section 1.2]{HMSVZ08}.
\begin{defi}\label{def:AueBasVec}
    Let $d\in \mathbb{N}_0$, $p\in \mathsf{P}^d_+$. A system $\bm{u}\in \Mat^d(\mathbb{R}^d)$ will be called a $p$-Auerbach basis of $\mathbb{R}^d$ if the following conditions hold:
    \begin{enumerate}
        \item[$\rm I)$] the system $\bm{u}$ forms a $d$-basis of $\R^d$;
        \item[$\rm II)$] one has $p(u_i)=1$ for each $i\in \overline{1,d}$;
        \item[$\rm III)$] putting $\bm{\psi}=\Dual_d(\bm{u})$, one has $p^*(\psi^j)=1$ for each $j\in \overline{1,d}$.
    \end{enumerate}
\end{defi}
\noindent In the above terms, by $\Aue_d[p]$, we will denote the subcollection of $\GL_d$ made of all $p$-Auerbach bases of $\mathbb{R}^d$. 

Let us note separately that, as easily follows from the above definition, any Auerbach basis of a Euclidean space is exactly an orthonormal basis therein.

The well-known existence result in the context of Auerbach bases is the following one.
\begin{theo}\label{theo:ClAueExist}
Given $d\in \mathbb{N}_0$ and $p\in \mathsf{P}^d_+$, the set $\Aue_d[p]$ is nonempty.
\begin{proof}
    The proof can be found in \cite[Theorem 1.16]{HMSVZ08}.
\end{proof}
\end{theo}

For what follows, we need the following simple statement. We give a detailed proof for the sake of completeness.
\begin{prop}\label{prop:AueComp}
    Let $d\in \mathbb{N}_0$, $p\in \mathsf{P}^d_+$. Then the set $\Aue_d[p]$ is compact.
    \begin{proof}
        In the case $d=0$ everything trivializes, so we may assume that $d>0$.
        
        As the space $\GL_d$ is Polish, it suffices to check the sequential compactness. Thus, fix an arbitrary sequence $(\bm{u}^n)_{n\in \mathbb{N}}\subseteq \Aue_d[p]$ and then, for each $n\in \mathbb{N}$, put $\bm{\psi}_n\coloneqq \Dual_d(\bm{u}^n)$.
        
        Given any $i\in \overline{1,d}$, we know that $p(u^n_i)=1$ for each $n\in \mathbb{N}$, whence the set of limit points of any subsequence of $(u_i^n)_{n\in \mathbb{N}}$ is nonempty and lies on the unit $p$-sphere. Similar arguments show that, given any $j\in \overline{1,d}$, the set of limit points of any subsequence of $(\psi_n^j)_{n\in \mathbb{N}}$ again is nonempty and lies on the unit $p^*$-sphere. Therefore, as can be readily checked, there exist $\bm{u}\in \Mat^d(\mathbb{R}^d)$ and $\bm{\psi}\in \Mat^d(\mathbb{R}^d)$, the (first $d$) components of which lie on the unit $p$-sphere and $p^*$-sphere, respectively, such that the system $(\bm{u},\bm{\psi})$ is a limit point of $\big((\bm{u}^n,\bm{\psi}_n)\big)_{n\in \mathbb{N}}$ in $\Mat^d(\mathbb{R}^d)\times \Mat^d(\mathbb{R}^d)$.

        We now observe that the duality pairing
        \begin{equation}
            \mathbb{R}^d\times \mathbb{R}^d\to \mathbb{R}\colon (v,\phi)\mapsto \phi(v)
        \end{equation}
        is continuous, whence, with the use of what is said above, it is not hard to get that the system $(\bm{u},\bm{\psi})$, as the limit of $d$-biorthogonal systems, is still $d$-biorthogonal, which, in particular, implies that $\bm{u}\in \GL_d$, as is well known from the standard linear algebra.

        All the conditions from \cref{def:AueBasVec} are satisfied for $\bm{u}$, so the system $\bm{u}$ indeed forms a $p$-Auerbach basis of $\mathbb{R}^d$, as stated. The proof is thus complete.
    \end{proof}
\end{prop}

While the set of Auerbach bases is nonempty for any given norm, for the next subsection we need to know that we can make a Borel selection within this correspondence. This is the essence of the lemma below.
\begin{lem}\label{lem:SelOfAueBas}
    Let $d\in \mathbb{N}_0$. Then there exists a Borel map
    \begin{equation}
    \mathrm{Sel}_d\colon \mathsf{P}^d_+\to \GL_d
    \end{equation}
    such that
    \begin{equation}
    \Sel_d(p)\in \Aue_d[p]\qquad \text{for any $p\in \mathsf{P}^d_+$}.
    \end{equation}
    \end{lem}

    For the proof of the given statement we will exploit the following Borel selection result.
    \begin{theo}\label{theo:BorSelTh}
    Let $\Omega$, $\widetilde{\Omega}$ be Polish spaces, let $\Delta\subseteq \Omega\times \widetilde{\Omega}$ be a Borel set. Suppose the sets
    \begin{equation}\label{eq:BorSelTh0}
        \big\{\widetilde{\omega}\in \widetilde{\Omega}\mid (\omega,\widetilde{\omega})\in \Delta\big\},\qquad \text{$\omega\in \Omega$},
    \end{equation}
    are nonempty and compact. Then there exists a Borel map $\delta\colon \Omega\to \widetilde{\Omega}$ such that
    \begin{equation}
        \big(\omega, \delta(\omega)\big)\in \Delta\qquad \text{for any $\omega\in \Omega$}.
    \end{equation}
    \begin{proof}
        The proof is given in \cite[Theorem 35.46]{K12}.
    \end{proof}
\end{theo}

\begin{rem}
    In fact, the conclusion in the theorem above still holds if one requires from the sets in \eqref{eq:BorSelTh0} that each of those is merely the union of a countable family of compact sets. Thus, as the subsequent proof shows, it would be enough if we proved in \cref{prop:AueComp} only that the sets of Auerbach bases are closed, since any closed subset of a locally compact, Polish space is of the above-described form.
\end{rem}

Now we provide the proof of the stated lemma.
    \begin{proof}[Proof of \cref{lem:SelOfAueBas}]
In the case $d=0$ there is nothing to prove, so further we assume that $d>0$ for simplicity. 

To begin with, we observe that the map
\begin{equation}
    \mathsf{P}^d\times \Mat^d(\mathbb{R}^d)\to \mathbb{R}^d\colon (p,\bm{v})\mapsto p(\bm{v})
\end{equation}
is continuous, as can be readily checked. Likewise, one can easily verify the continuity of the map
\begin{equation}
    \mathsf{P}^d_+\to \mathsf{P}^d_+\colon p\mapsto p^*.
\end{equation}
        Finally, the continuity of the map
        \begin{equation}
            \GL_d\to \GL_d\colon \bm{v}\mapsto \Dual_d(\bm{v})
        \end{equation}
        can be obtained in a straightforward way as well.

        Next, consider the map $\wbar{\Delta}\colon \mathsf{P}^d_+\times \GL_d\to \mathbb{R}^{d}\times \mathbb{R}^d$ given by
        \begin{equation}
          \wbar{\Delta}(p,\bm{v})\coloneqq \Big(p(\bm{v}), p^*\big(\Dual_d(\bm{v})\big)\Big).
        \end{equation}
         With all the above in mind, it follows that this map is continuous and hence Borel. Consider then the Borel set
        \begin{equation}
            \Delta\coloneqq \wbar{\Delta}^{-1}\Big((1,...,1,0,...),(1,...,1,0,...)\Big),
        \end{equation}
        where in each component there are exactly $d$ ones. As can be checked via \cref{def:AueBasVec}, the following equality holds for any given $p\in \mathsf{P}^d_+$:
        \begin{equation}
            \Big\{\bm{u}\in \GL_d \bigm| (p, \bm{u})\in \Delta\Big\}=\Aue_d[p].
        \end{equation}
        By \cref{theo:ClAueExist}, the sets on the left are nonempty. In turn, in light of \cref{prop:AueComp}, those are also compact. We are then exactly in a position to apply \cref{theo:BorSelTh}, which immediately leads us to the conclusion.
    \end{proof}

\subsection{Extension to normed modules} Now we transfer the introduced discussion to the context of normed modules.

For the current subsection, we fix a sigma-finite measure space $(\X,\m)$ and an $\Leb_{\m}$-normed $\Leb^{\infty}_{\m}$-module $\M=\big(\M, |\cdot|_{\M}\big)$.

Below we introduce a natural analog of Auerbach bases in the framework of normed modules.
\begin{defi}\label{def:AuerBasMod}
    Let $E\in \P_{\m}(\X)$, $d\in \mathbb{N}_0$. We say that a system $\bm{u}\in \Mat^d(\M)$ is an Auerbach $d$-basis of $\M$ over $E$ if the following conditions hold:
    \begin{enumerate}
        \item[$\rm I)$] the system $\bm{u}$ forms a $d$-basis of $\M$ over $E$.
        \item[$\rm II)$] one has, for each $i\in \overline{1,d}$, that  
        \begin{equation}
            |u_i|_{\M}=1 \qquad \text{$\m$-a.e. on $E$}.
        \end{equation}
        \item[$\rm III)$] putting $\bm{\psi}\coloneqq \Dual_d^E(\bm{u})$, one has, for each $j\in \overline{1,d}$, that 
        \begin{equation}
            |\psi^j|_{\M^{\star}}=1 \qquad \text{$\m$-a.e. on $E$}.
        \end{equation}
    \end{enumerate}
\end{defi}
\noindent Let us note the following: if the reference set above is of zero measure, then the listed requirements turn out to hold vacuously, regardless of the properties of the other entries. This will be tacitly used in what follows.

With the above definition given, we are in a position to formulate the main result of this section. Namely, it states that once a normed module admits a basis, this module also admits an Auerbach basis. This is quite reasonable to expect, as the latter notion is pointwise in nature. And the subsequent proof is based exactly on this idea. The precise formulation is as follows.
\begin{theo}\label{theo:ExistAuerBasis}
    Let $E\in \mathrm{P}_{\m}(\X)$, $d\in \mathbb{N}_0$. Suppose the module $\M$ admits a $d$-basis over $E$. Then there is an $\m$-partition $(E_m)_{m\in \mathbb{N}}$ of $E$ such that, for each $m\in \mathbb{N}$, the module $\M$ admits an Auerbach $d$-basis over $E_m$.
\end{theo}
\noindent In the statement above we need to pass to such a partition purely because we deal here with merely an $\Leb_{\m}^{\infty}$-module. As the subsequent proof demonstrates, if one considered an $\Leb_{\m}$-module, one could get the existence of an Auerbach basis over the initial set as well.

For the proof of the theorem above we will need the auxiliary lemma below, which can be useful on its own. And while quite similar statements can be found in many sources (see, for instance, \cite[Theorem 10.12]{W18}), we prefer to provide all the details for the sake of completeness.
\begin{lem}\label{prop:ExistPointNorm}
    Let $E\in \mathrm{P}_{\m}(\X)$, $d\in \mathbb{N}_0$, $\bm{v}\in \Mat^d(\M)$. Then there exists an $\m$-measurable map
    \begin{equation}\label{eq:ExistPointNorm0}
        \norm\colon E\to \mathsf{P}^d\colon x\mapsto \norm_x
    \end{equation}
    such that it holds, for any $\bm{\xi}\in \Mat^{d}(\Leb^{\infty}_{\m})$, that
    \begin{equation}\label{eq:ExistPointNorm00}
        \norm_x\big(\bm{\xi}(x)\big)=|\bm{\xi} \cdot \bm{v}|_{\M}(x)\qquad \text{for $\m$-a.e. $x\in E$}.
    \end{equation}
    Moreover, if the system $\bm{v}$ is $d$-independent over $E$, then it holds that
    \begin{equation}
        \norm_x\in \mathsf{P}^d_+\qquad \text{for $\m$-a.e. $x\in E$}.
    \end{equation}
    Furthermore, if the system $\bm{v}$ forms a $d$-basis of $\M$ over $E$, then, putting $\bm{\phi}\coloneqq \Dual_d^E(\bm{v})$, it holds, for any $\bm{\lambda}\in \Mat^d(\Leb_{\m})$, that
    \begin{equation}\label{eq:ExistPointNorm000}
        \norm^*_x\big(\bm{\lambda}(x)\big)=|\bm{\lambda}\cdot\bm{\phi}|_{\M^{\star}}(x)\qquad \text{for $\m$-a.e. $x\in E$}.
    \end{equation}
    \end{lem}

For the proof of (a part of) the above lemma, we shall exploit the following selection result, known as the Kuratowski--Ryll-Nardzewski theorem.
\begin{theo}\label{theo:MeasSelTh}
    Let $X$ be a set, let $\mathfrak{A}$ be a sigma-algebra on $X$, let $E\in \mathfrak{A}$, let $\Omega$ be a Polish space. Suppose we are given, for every $x\in E$, a nonempty, closed set $\Delta(x)\subseteq \Omega$. Suppose that, for any open set $O\subseteq \Omega$, we have
    \begin{equation}
        \Big\{x\in E\bigm| \Delta(x)\cap O\neq \varnothing\Big\}\in \mathfrak{A}.
    \end{equation}
    Then there exists an $\mathfrak{A}$-measurable map $\delta\colon E\to \Omega$ such that
    \begin{equation}
        \delta(x)\in \Delta(x)\qquad \text{for every $x\in E$}.
    \end{equation}
    \begin{proof}
        The proof can be found in \cite[Theorem 5.2.1]{S98}.
    \end{proof}
\end{theo}

We also need the following auxiliary statement.
\begin{prop}\label{prop:WBorSN}
    Let $X$ be a set, let $\mathfrak{A}$ be a sigma-algebra on $X$, let $E\in \mathfrak{A}$, $d\in \mathbb{N}_0$, let $\Norm\colon E\to \mathsf{P}^d\colon x\mapsto \Norm_x$ be a map such that the functions
    \begin{equation}
        E\to \R_{\geq 0}\colon x\mapsto \Norm_x(\xi),\qquad \xi\in \mathbb{Q}^d,
    \end{equation}
   are $\mathfrak{A}$-measurable. Then the map $\Norm$ is $\mathfrak{A}$-measurable.
   \begin{proof}
       The proof can be carried out with the use of quite standard measure-theoretic techniques, so we omit routine details.
   \end{proof}
\end{prop}

Now we return to \cref{prop:ExistPointNorm}.
    \begin{proof}[Proof of \cref{prop:ExistPointNorm}]
         Since the statement is clearly local, we may assume without loss of generality that there is $\zeta\in \R_{>0}$ with
         \begin{equation}\label{eq:ExistPointNorm1}
             \mathsf{p}_{\infty}\big(|\bm{v}|_{\M}\big)\leq \zeta\qquad \text{$\m$-a.e. on $E$}.
         \end{equation}
         Then, for every $\bm{\xi} \in \mathbb{Q}^d$, choose a finite, nonnegative, and $\m$-measurable $\m$-representative, say $\rho_{\bm{\xi}}$, of $\big|(\bm{\xi}\1_{\X})\cdot \bm{v}\big|_{\M}$. Given any $\bm{\xi} \in  \mathbb{Q}^d$, let $E_{\bm{\xi}}$ be the set of all $x\in E$ satisfying the following conditions:
         \begin{gather}
             \rho_{\bm{\xi} +\bm{\xi}'}(x)\leq \rho_{\bm{\xi}}(x)+\rho_{\bm{\xi}'}(x)\qquad \text{for any $\bm{\xi}'\in  \mathbb{Q}^d$};\label{eq:ExistPointNorm2}\\
            \rho_{\eta\bm{\xi}}(x)=|\eta|\rho_{\bm{\xi}}(x)\qquad \text{for any $\eta\in \mathbb{Q}$};\label{eq:ExistPointNorm3}\\
             \rho_{\bm{\xi}}(x)\leq \zeta \mathsf{p}_1(\bm{\xi})\label{eq:ExistPointNorm4}.
         \end{gather}
         The defining properties of $\Leb_{\m}$-norms, together with \eqref{eq:ExistPointNorm1}, easily imply that $\m\big(E\backslash E_{\bm{\xi}}\big)=0$ for any $\bm{\xi}\in \mathbb{Q}^d$, whence, with
         \begin{equation}
             \widetilde{E}\coloneqq \bigcap\limits_{\bm{\xi}\in \mathbb{Q}^d} E_{\bm{\xi}},
         \end{equation}
         we also have $\m\big(E\backslash \widetilde{E}\big)=0$. By \eqref{eq:ExistPointNorm2} and \eqref{eq:ExistPointNorm3}, we know that, for any $x\in \widetilde{E}$, the map
         \begin{equation}
             \mathbb{Q}^d\to \mathbb{R}_{\geq 0}\colon \bm{\xi}\mapsto \rho_{\bm{\xi}}(x)
         \end{equation}
         is a $\mathbb{Q}$-seminorm on $\mathbb{Q}^d$, i.e., a $\mathbb{Q}$-homogeneous, subadditive function $\mathbb{Q}^d\to \R_{\geq 0}$, which, by \eqref{eq:ExistPointNorm4}, can be seen to be continuous with respect to the topology on $\mathbb{R}^d$. Such ones are known to be uniquely extendable to ordinary seminorms, i.e., $\R$-seminorms, on $\mathbb{R}^d$. So, for any $x\in \widetilde{E}$, let $\norm_x$ be the corresponding seminorm obtained by this procedure. In turn, for any $x\in \big(E\backslash \widetilde{E}\big)$, we declare $\mathsf{n}_x$ to be the trivial seminorm on $\mathbb{R}^d$. All this gives us the required map $\norm\colon E\to \mathsf{P}^d$.
         
         To check that the map $\mathsf{n}$ is $\m$-measurable, we notice that
         \begin{equation}
             \norm_x(\bm{\xi})=\rho_{\bm{\xi}}(x)\qquad \text{for any $\bm{\xi}\in \mathbb{Q}^d$ and any $x\in \widetilde{E}$},
         \end{equation}
         which implies that the functions
         \begin{equation}
             \widetilde{E}\to \mathbb{R}_{\geq 0}\colon x\mapsto \norm_x(\bm{\xi}),\qquad \text{$\bm{\xi}\in \mathbb{Q}^d$},
         \end{equation}
         are $\m$-measurable. As can be deduced from \cref{prop:WBorSN}, the map $\norm$ itself is then $\m$-measurable.

         We move to the second part, so we assume further that the system $\bm{v}$ is $d$-independent over $E$. Arguing by contradiction, we may suppose that there is $E'\in \P^+_{\m}(\widetilde{E})$ such that
         \begin{equation}\label{eq:ExistPointNorm6}
             \mathsf{n}_x\notin \mathsf{P}^d_+\qquad \text{for any $x\in E'$}.
         \end{equation}
         Then, for any $x\in E'$, we define the set
         \begin{equation}\label{eq:ExistPointNorm7}
             \bm{\Xi}(x)\coloneqq \Big\{\bm{\xi}\in \R^d \bigm| \norm_x(\bm{\xi})=0, \mathsf{p}_{\infty}(\bm{\xi})=1\Big\},
         \end{equation}
         which is nonempty by \eqref{eq:ExistPointNorm6} and, by trivial reasons, also compact. We would like to find a measurable selection of this multifunction, so we need to check that all the conditions of \cref{theo:MeasSelTh} are met. To this end, given any $n\in \mathbb{N}$, we define, for any $x\in E'$, the set
         \begin{equation}
             \bm{\Xi}_n(x)\coloneqq \Big\{\bm{\xi}\in \R^d\bigm| \norm_x(\bm{\xi})\leq \tfrac{1}{n}, \big|\mathsf{p}_{\infty}(\bm{\xi})-1\big|\leq\tfrac{1}{n}\Big\},
         \end{equation}
         the set
         \begin{equation}
             \bm{\Xi}_n\coloneqq \Big\{\bm{\xi}\in \mathbb{Q}^d\bigm| \big|\mathsf{p}_{\infty}(\bm{\xi})-1\big|\leq\tfrac{1}{n}\Big\},
         \end{equation}
         and, for any $\bm{\xi} \in \mathbb{Q}^d$, the set
         \begin{equation}
             E'(\bm{\xi},n)\coloneqq \Big\{x\in E'\bigm| \norm_x(\bm{\xi})\leq\tfrac{1}{n}\Big\},
         \end{equation}
         which can be seen to be $\m$-measurable by the first part of the current lemma. Now, pick an open set $O\subseteq \mathbb{R}^d$ and find a family $(O_m)_{m\in \Nat}$ of open subsets of $O$ that covers $O$ such that, for each $m\in \Nat$, the closure of $O_m$ in $\R^d$ is a compact subset of $O$. Using then the continuity of seminorms on $\mathbb{R}^d$, we can write that
         \begin{equation}
         \begin{gathered}
             \Big\{x\in E'\bigm| \bm{\Xi}(x)\cap O\neq \varnothing\Big\}=\bigcup\limits_{m\in \Nat}\bigcap\limits_{n\in \mathbb{N}} \bigcup\limits_{\bm{\xi}\in \mathbb{Q}^d\cap O_m} \Big\{x\in E'\bigm| \bm{\xi}\in \bm{\Xi}_n(x)\Big\}=\\
             =\bigcup\limits_{m\in \Nat}\bigcap\limits_{n\in \mathbb{N}} \bigcup\limits_{\bm{\xi}\in \bm{\Xi}_n\cap O_m} E'(\bm{\xi},n),
             \end{gathered}
         \end{equation}
         which immediately gives us that the set on the left is $\m$-measurable. We are exactly in a position to apply \cref{theo:MeasSelTh}, so, after unfolding the conclusion appropriately, we can find $\bm{\xi}\in \Mat^d(\Leb^{\infty}_{\m})$ such that
         \begin{equation}\label{eq:ExistPointNorm8}
             \bm{\xi}(x)\in \bm{\Xi}(x)\qquad \text{for $\m$-a.e. $x\in E'$},
         \end{equation}
         whence, with the use of \eqref{eq:ExistPointNorm00}, we get that
         \begin{equation}
            | \bm{\xi}\cdot  \bm{v}|_{\M}(x)=\norm_x\big(\bm{\xi}(x)\big)=0 \qquad \text{for $\m$-a.e. $x\in E'$}.
         \end{equation}
         By the defining properties of $\Leb_{\m}$-norms, this implies immediately that
         \begin{equation}
             \bm{\xi}\cdot \bm{v}=\o_{\M}\qquad \text{on $E'$},
         \end{equation}
         whence, as the system $\bm{v}$ is $d$-independent over $E$, we have
         \begin{equation}
             \bm{\xi}=\bm{\o}_{\Leb_{\m}}\qquad \text{$\m$-a.e. on $E'$},
         \end{equation}
         which contradicts \eqref{eq:ExistPointNorm7} and \eqref{eq:ExistPointNorm8}, as declared.

         We are left with the third part, so we assume further that the system $\bm{v}$ forms a $d$-basis of $\M$ over $E$. Put then $\bm{\phi}\coloneqq  \Dual^E_d(\bm{v})$.

         Fix $\bm{\lambda}\in \Mat^d(\Leb_{\m})$. As the system $(\bm{v},\bm{\phi})$ is $d$-biorthogonal over $E$, for any $\bm{\xi}\in \Mat^d(\Leb^{\infty}_{\m})$, we have
         \begin{equation}
             (\bm{\lambda}\cdot \bm{\xi})(x)=\big((\bm{\lambda}\cdot \bm{\phi})(\bm{\xi}\cdot \bm{v})\big)(x)\qquad \text{for $\m$-a.e. $x\in E$}.
         \end{equation}
       In view of the first part of the current lemma, this implies that, for any $\bm{\xi}\in \Mat^d(\Leb^{\infty}_{\m})$, we have
         \begin{equation}
             \big|\bm{\lambda}(x)\cdot\bm{\xi}(x)\big|\leq |\bm{\lambda}\cdot\bm{\phi}|_{\M^{\star}}(x) |\bm{\xi}\cdot \bm{v}|_{\M}(x)=|\bm{\lambda}\cdot\bm{\phi}|_{\M^{\star}}(x) \norm_x\big(\bm{\xi}(x)\big)\qquad \text{for $\m$-a.e. $x\in E$},
         \end{equation}
         whence it clearly holds, for $\m$-a.e. $x\in E$, that
         \begin{equation}
             \big|\bm{\lambda}(x)\cdot \bm{\xi}\big|\leq |\bm{\lambda}\cdot \bm{\phi}|_{\M^{\star}}(x) \norm_x(\bm{\xi})\qquad \text{for any $\bm{\xi}\in \mathbb{Q}^d$}
         \end{equation}
         and hence that
         \begin{equation}
             \norm_x^*\big(\bm{\lambda}(x)\big)\leq |\bm{\lambda}\cdot \bm{\phi}|_{\M^{\star}}(x).
         \end{equation}
         For the opposite inequality, we fix $v\in \M$ and find an $\m$-partition $(E_m)_{m\in \mathbb{N}}$ of $E$ such that, for each $m\in \Nat$, we have $\1_{E_m}\bm{\phi}(v)\in \Mat^d(\Leb_{\m}^{\infty})$ and
         \begin{equation}
             v=\big(\1_{E_m}\bm{\phi}(v)\big)\cdot  \bm{v}\qquad \text{on $E_m$}.
         \end{equation}
         Then, for each $m\in \mathbb{N}$, we can write that
         \begin{equation}
         \begin{gathered}
            \big|(\bm{\lambda}\cdot \bm{\phi})(v)\big|_{\Leb_{\m}}(x)\leq \norm_x^*\big(\bm{\lambda}(x)\big) \norm_x\Big(\big(\bm{\phi}(v)\big)(x)\Big)=\\
            =\norm_x^*\big(\bm{\lambda}(x)\big)\Big|\big(\1_{E_m}\bm{\phi}(v)\big)\cdot \bm{v}\Big|_{\M}(x)=\norm_x^*\big(\bm{\lambda}(x)\big)|v|_{\M}(x)\qquad \text{for $\m$-a.e. $x\in E_m$}.
            \end{gathered}
         \end{equation}
         This, in turn, implies that
         \begin{equation}
              \big|(\bm{\lambda}\cdot \bm{\phi})(v)\big|_{\Leb_{\m}}(x)\leq \norm_x^*\big(\bm{\lambda}(x)\big)|v|_{\M}(x)\qquad \text{for $\m$-a.e. $x\in E$}.
         \end{equation}
         By the arbitrariness of $v$ we get that
         \begin{equation}
            |\bm{\lambda}\cdot \bm{\phi}|_{\M^{\star}}(x)\leq  \norm_x^*\big(\bm{\lambda}(x)\big) \qquad \text{for $\m$-a.e. $x\in E$}.
         \end{equation}
         Thus, the equality in \eqref{eq:ExistPointNorm000} is verified.

         The proof is complete.
    \end{proof}

We are ready to provide the proof of \cref{theo:ExistAuerBasis}.
\begin{proof}[Proof of \cref{theo:ExistAuerBasis}.]
In the case $d=0$ there is nothing to prove, so we assume further that $d>0$ for the sake of simplicity.

By the given assumptions, we can find a system $\bm{v}\in \Mat^d(\M)$ that forms a $d$-basis of $\M$ over $E$. Put then $\bm{\phi}\coloneqq \Dual^E_d(\bm{v})$. Let also $\Sel_d$ be the map constructed in \cref{lem:SelOfAueBas}. Finally, let $\norm$ be the map as in \cref{prop:ExistPointNorm} constructed with respect to $E$ and $\bm{v}$ and let $E_+\in \P_{\m}(E)$ be a set with $\m\big(E\backslash E_+\big)=0$ such that $\mathsf{n}_x\in \mathsf{P}^d_+$ for every $x\in E_+$.

Now we look at the map $\Sel_d\circ (\norm|_{E_+})$ and juxtapose \cref{lem:SelOfAueBas} and \cref{prop:ExistPointNorm}. Rewriting the outcome in more explicit terms, we get the following: there are $\bm{\xi}_i\in \Mat^d(\Leb_{\m})$, $i\in \overline{1,d}$, and $\bm{\lambda}^j\in \Mat^d(\Leb_{\m})$, $j\in \overline{1,d}$, such that, for all $i,j\in \overline{1,d}$, we have
\begin{gather}
    \norm_x\big(\bm{\xi}_i(x)\big)=1 \qquad \text{for $\m$-a.e. $x\in E$},\label{eq:ExistAuerBasis2}\\
    \norm_x^*\big(\bm{\lambda}^j(x)\big)=1\qquad \text{for $\m$-a.e. $x\in E$},\label{eq:ExistAuerBasis3}\\
   \bm{\lambda}^j(x)\cdot\bm{\xi}_i(x)=\updelta(i,j)\qquad \text{for $\m$-a.e. $x\in E$}\label{eq:ExistAuerBasis4}.
\end{gather}
With this, we then find an $\m$-partition $(E_m)_{m\in \mathbb{N}}$ of $E$ such that $\1_{E_m}\bm{\xi}_i\in \Mat^d(\Leb_{\m}^{\infty})$ and $\1_{E_m}\bm{\lambda}^j\in \Mat^d(\Leb_{\m}^{\infty})$ for all $i,j\in \overline{1,d}$ and for each $m\in \mathbb{N}$.

Next, we fix $m\in \mathbb{N}$ and, for each $i\in \overline{1,d}$, put
\begin{equation}\label{eq:ExistAuerBasis5}
    u^m_i\coloneqq (\1_{E_m}\bm{\xi}_i) \cdot \bm{v},
\end{equation}
which gives us the system $\bm{u}^m\in \Mat^d(\M)$. We aim to show that the system $\bm{u}^m$ forms an Auerbach $d$-basis of $\M$ over $E_m$. Indeed, condition $\rm I)$ of \cref{def:AuerBasMod} can be deduced from the fact that the system $\bm{v}$ forms a $d$-basis of $\M$ over $E$ via \eqref{eq:ExistAuerBasis4} and \eqref{eq:ExistAuerBasis5}. In turn, condition $\rm II)$ of \cref{def:AuerBasMod} follows from \eqref{eq:ExistPointNorm00} in \cref{prop:ExistPointNorm} via \eqref{eq:ExistAuerBasis2} and \eqref{eq:ExistAuerBasis5}. It remains to check condition $\rm III)$ of \cref{def:AuerBasMod}. For that, put $\bm{\psi}_m\coloneqq \Dual_d^{E_m}(\bm{u}^m)$. We now notice that, for each $j\in \overline{1,d}$, we have
\begin{equation}
    \psi^j_m=\bm{\lambda}^j\cdot \bm{\phi}\qquad \text{on $E_m$},
\end{equation}
as follows from \eqref{eq:ExistAuerBasis4} and \eqref{eq:ExistAuerBasis5}. We then use \cref{prop:ExistPointNorm} once again to get, for any $j\in \overline{1,d}$, that
\begin{equation}
    |\psi^j_m|_{\M^{\star}}=1\qquad \text{$\m$-a.e. on $E_m$}.
\end{equation}
Thus, the system $\bm{u}^m$ indeed forms an Auerbach $d$-basis of $\M$ over $E_m$.

Due to the arbitrariness of $m$, we come exactly to the desired conclusion.
\end{proof}

Let us also conclude here with the following fact concerning Auerbach bases of Euclidean modules.
\begin{prop}\label{prop:EuclModONBAB}
    Let $E\in \P_{\m}(\X)$. Suppose the module $\M$ is Euclidean over $E$. Let $d\in \Nat_0$, $\bm{u}\in \Mat^d(\M)$. Then the system $\bm{u}$ forms an Auerbach $d$-basis of $\M$ over $E$ if and only if the system $\bm{u}$ forms a $d$-orthonormal basis of $\M$ over $E$.
    \begin{proof}
        This can be verified quite easily by direct calculations, so we omit the details.
    \end{proof}
\end{prop}

\begin{rem}\label{rem:ExONB}
    The proposition above, in combination with \cref{theo:ExistAuerBasis}, establishes the existence of orthonormal bases for Euclidean modules. It is worth noting, however, that, as is briefly described in \cref{prop:ExOrthBasEucl}, orthonormal bases of Euclidean modules can be shown to exist by essentially the same methods as in the case of Euclidean spaces. Thus, in other words, for Euclidean modules the existence of Auerbach bases can be obtained by much simpler means, which is useful to keep in mind.
\end{rem}

\section{Algebraic distortion factors}\label{ss:ADF}

This section is devoted to introducing objects, as well as examining certain properties thereof, that we will call algebraic distortion factors. As the main outcome of the section, we obtain two results. The first one, \cref{theo:MainEstADis}, states that, by looking at precisely Auerbach bases of a given module, we can always achieve certain bounds, which will be clarified below, on the algebraic distortion factors, and exactly these bounds will constitute the corresponding estimates in \cref{theo:MainRes} and \cref{theo:MainResEucl}. The second one, \cref{theo:LACisApprox}, provides us basically with the necessary flexibility in the choice of a system in the dual module delivering such bounds, which will form a crucial step for the proof of our main theorems.

Throughout all of the section, a sigma-finite measure space $(\X,\m)$ and an $\Leb_{\m}$-normed $\Leb^{\infty}_{\m}$-module $\M=\big(\M,|\cdot|_{\M}\big)$ are assumed fixed.

\subsection{Definition and main estimates}\label{ss:DefMEst}

We begin by defining the central objects of interest for this section and discussing certain key estimates formulated in terms of those.

\begin{rem}\label{rem:MotAnDis}
    Let us give a supporting comment on the forthcoming definition. Let $\big(\mathbb{V}, \|\cdot\|_{\mathbb{V}}\big)$ and $\big(\widetilde{\mathbb{V}}, \|\cdot\|_{\widetilde{\mathbb{V}}}\big)$ be normed vector spaces, let $A\colon \mathbb{V}\to \widetilde{\mathbb{V}}$ be a linear map. One of many equivalent ways to define the operator norm of $A$ is via the following formula:
    \begin{equation}\label{eq:MotAnDis0}
       \|A\|_{\mathrm{on}}\coloneqq \inf\limits\Big\{L\in \R_{\geq 0} \bigm| \text{$\big\|A(v)\big\|_{\widetilde{\mathbb{V}}}\leq L\|v\|_{\mathbb{V}}$ for any $v\in \mathbb{V}$}\Big\}.
    \end{equation}
    One can consider though a somewhat dual notion, sometimes called the minimum modulus of $A$, namely the following quantity:
    \begin{equation}\label{eq:MotAnDis00}
        \|A\|_{\mathrm{mm}}\coloneqq \sup\Big\{L\in \R_{\geq 0} \bigm| \text{$\big\|A(v)\big\|_{\widetilde{\mathbb{V}}}\geq L\|v\|_{\mathbb{V}}$ for any $v\in \mathbb{V}$}\Big\}.
    \end{equation}
    It obviously measures the nondegeneracy degree of $A$ and, whenever nonzero, is equal to the reciprocal of the operator norm of the left inverse of the operator (as defined on its image). Both these quantities admit natural analogs in the realm of normed modules, which, however, must already be functions defined on the given underlying space, rather than merely numbers. And exactly the module counterpart of both the quantities given in \eqref{eq:MotAnDis0} and \eqref{eq:MotAnDis00} (in a form adapted to our exposition) is what we wish to consider below. Intentionally violating any common terminology for these, we will call such objects algebraic distortion factors. This label is chosen in coordination with that of the subsequent notion of geometric distortion factors, which will appear in \cref{ss:GDF}, to draw parallels between the two. In turn, the name for the latter ones is directly motivated by the implicit presence of those, which should be easy to detect, in the definition of BLD maps from the introduction.
\end{rem}

Bearing in mind the above remark, we give the following definition.
\begin{defi}\label{def:DistCoef}
  Let $p\in \mathsf{P}^{\infty}$, $\bm{\phi} \in \Mat^{\infty}(\M^{\star})$. By the upper algebraic distortion factor of $\bm{\phi}$ with respect to $p$, we will mean the function in $[\overline{\Leb}_{\m}]_{\geq 0}$ obtained by
  \begin{equation}
      \ADis_p^{\mathrm{u}}[\bm{\phi}]\coloneqq \bigwedge\limits_{\varrho}\varrho,
  \end{equation}
  where the meet is taken over all $\varrho\in [\overline{\Leb}_{\m}]_{\geq 0}$ such that, for every $v\in \M$, we have
\begin{equation}\label{eq:DistCoef0}
  \varrho |v|_{\M}\geq   p\big(\bm{\phi}(v)\big)\qquad \text{$\m$-a.e. on $\X$}.
\end{equation}
  By the lower algebraic distortion factor of $\bm{\phi}$ with respect to $p$, we will mean the function in $[\overline{\Leb}_{\m}]_{\geq 0}$ obtained by
\begin{equation}
    \ADis^{\mathrm{l}}_{p}[\bm{\phi}]\coloneqq \bigvee\limits_{\varrho}\varrho,
\end{equation}
where the join is taken over all $\varrho\in [\overline{\Leb}_{\m}]_{\geq 0}$ such that, for every $v\in \M$, we have
\begin{equation}\label{eq:DistCoef00}
    \varrho|v|_{\M}\leq p\big(\bm{\phi}(v)\big)\qquad \text{$\m$-a.e. on $\X$}.
\end{equation}
\end{defi}

Let $p\in \mathsf{P}^{\infty}$, $\bm{\phi}\in \Mat^{\infty}(\M^{\star})$. We note first that, as is not hard to get from \cref{def:DistCoef}, the corresponding algebraic distortion factors themselves satisfy the condition in \eqref{eq:DistCoef0} and \eqref{eq:DistCoef00}, i.e., for every $v\in \M$, we have
\begin{equation}
    \ADis^{\mathrm{u}}_p[\bm{\phi}]|v|_{\M}\geq p\big(\bm{\phi}(v)\big)\qquad \text{$\m$-a.e. on $\X$}
\end{equation}
and
\begin{equation}
    \ADis^{\mathrm{l}}_p[\bm{\phi}]|v|_{\M}\leq p\big(\bm{\phi}(v)\big)\qquad \text{$\m$-a.e. on $\X$}.
\end{equation}
Also, one more easy consequence of \cref{def:DistCoef} is that algebraic distortion factors are absolutely $\Leb_{\m}$-homogeneous in the following sense: given any $E\in \P_{\m}(\X)$ and any $\lambda\in \Leb_{\m}$ with
\begin{equation}\label{eq:CorHomADis}
    \lambda(x)\neq 0\qquad \text{for $\m$-a.e. $x\in E$},
\end{equation}
we have
\begin{equation}
    \ADis^{\mathrm{u}}_p[\lambda \bm{\phi}]=|\lambda|_{\Leb_{\m}}\ADis^{\mathrm{u}}_p[\bm{\phi}]\qquad \text{$\m$-a.e. on $E$}
\end{equation}
and
\begin{equation}
   \ADis^{\mathrm{l}}_p[\lambda \bm{\phi}]=|\lambda|_{\Leb_{\m}}\ADis^{\mathrm{l}}_p[\bm{\phi}]\qquad \text{$\m$-a.e. on $E$}.
\end{equation}
The condition in \eqref{eq:CorHomADis} is necessary here to avoid problems caused by our infinity conventions.

\begin{rem}\label{rem:GoalAnDis}
    We want to give some more explanations concerning the role of algebraic distortion factors in our exposition. As can be seen, any system $\bm{\phi}\in \Mat^{\infty}(\M^{\star})$ gives rise to an obvious $\Leb_{\m}^{\infty}$-linear map $\M\to \Mat^{\infty}(\Leb_{\m})$, which is useful to think of loosely (especially in light of the sequel) as something defined infinitesimally on $\X$. Such an ``algebraic'' map, as we get to the appropriate setting of metric measure spaces in \cref{ss:AppMMS}, should be considered, in turn, as a proxy for a ``geometric'' map of the form appearing in \cref{theo:MainRes} and \cref{theo:MainResEucl}. Thus, in order to achieve the desired ``geometric'' two-sided control in our main results, we need first to obtain the corresponding ``algebraic'' one. And algebraic distortion factors are responsible exactly for the latter one, which determines our forthcoming actions in this regard.
\end{rem}

Let us give the following straightforward estimate of the upper algebraic distortion factors.
\begin{prop}\label{prop:UpADisRef}
    Let $d\in \Nat_0$, $p\in \mathsf{P}^d_+$, $\bm{\phi}\in \Mat^d(\M^{\star})$. Then one has
    \begin{equation}\label{eq:UpADisRef0}
        \ADis^{\mathrm{u}}_p[\bm{\phi}]\leq p(\bm{\phi})\qquad \text{$\m$-a.e. on $\X$}.
    \end{equation}
    \begin{proof}
        Fix $v\in \M$ and simply write, via the relevant definitions, that
        \begin{equation}
            p\big(\bm{\phi}(v)\big)=\bigvee\limits_{\bm{\lambda}\in \R^d, p^*(\bm{\lambda})\leq 1} \big|\bm{\lambda }\cdot \bm{\phi}(v)\big|_{\Leb_{\m}}\leq \bigvee\limits_{\bm{\lambda}\in \R^d, p^*(\bm{\lambda})\leq 1} |\bm{\lambda}\cdot \bm{\phi} |_{\M^{\star}}|v|_{\M}=p(\bm{\phi})|v|_{\M},
        \end{equation}
        which, due to the arbitrariness of $v$, clearly implies the desired estimate.
    \end{proof}
\end{prop}

As one can anticipate, a control from below on lower distortion factors of a system can be given in terms of the dual norm of the corresponding dual system, the fact of which is recorded in the proposition below.
\begin{prop}\label{prop:ADisEst}
    Let $E\in \mathrm{P}_{\m}(\X)$, $d\in \mathbb{N}_0$, $p\in \mathsf{P}^d_+$. Let $\bm{u}\in \Mat^{d}(\M)$ be a system forming a $d$-basis of $\M$ over $E$. Put $\bm{\psi}\coloneqq \Dual^E_d(\bm{u})$. Then it holds that
    \begin{gather}
        \ADis^{\mathrm{l}}_{p}[\bm{\psi}]\geq \frac{1}{p^*(\bm{u})}\qquad \text{$\m$-a.e. on $E$}.
    \end{gather}
    \begin{proof}
Fix $v\in \M$ and find a corresponding $\m$-partition $(E_m)_{m\in \mathbb{N}}$ of $E$ such that, for each $m\in \mathbb{N}$, we have $\1_{E_m}\bm{\psi}(v)\in \Mat^d(\Leb_{\m}^{\infty})$ and
\begin{equation}
    v=\big(\1_{E_m}\bm{\psi}(v)\big)\cdot \bm{u}\qquad \text{on $E_m$}.
\end{equation}
It then easily follows that
\begin{equation}
    |v|_{\M}=\Big|\big(\1_{E_m}\bm{\psi}(v)\big)\cdot \bm{u}\Big|_{\M}\leq p\big(\bm{\psi}(v)\big) p^*(\bm{u})\qquad \text{$\m$-a.e. on $E_m$}.
\end{equation}
Juxtaposing this with \cref{def:DistCoef}, we directly come to the conclusion.
    \end{proof}
\end{prop}

We now combine the previous statements with \cref{prop:EstNormAbsNorm} into the lemma below, which establishes quite specific estimates on the algebraic distortion factors of given systems in terms of their dual systems. From what is below, one can already extract the importance of Auerbach bases in this context.
\begin{lem}\label{lem:ADisEst}
    Let $E\in \mathrm{P}_{\m}(\X)$, $d\in \mathbb{N}_0$. Let $\bm{u}\in \Mat^d(\M)$ be a system forming a $d$-basis of $\M$ over $E$. Put $\bm{\psi}\coloneqq \Dual^E_d(\bm{u})$. Then the assertions below are valid.
     \begin{enumerate}
        \item[$\rm I)$] The following properties hold:
    \begin{gather}
    \ADis^{\mathrm{u}}_{\mathsf{p}_{\infty}}[\bm{\psi}]\leq \mathsf{p}_{\infty}\big(|\bm{\psi}|_{\M^{\star}}\big)\qquad \text{$\m$-a.e. on $E$};\\
        \ADis^{\mathrm{l}}_{\mathsf{p}_1}[\bm{\psi}]\geq \frac{1}{\mathsf{p}_{\infty}\big(|\bm{u}|_{\M}\big)}\qquad \text{$\m$-a.e. on $E$}.
    \end{gather}
    \item[$\rm II)$] Suppose the module $\M$ is Euclidean over $E$ and suppose the system $\bm{u}$ is orthogonal over $E$. Then the following properties hold:
       \begin{gather}
    \ADis^{\mathrm{u}}_{\mathsf{p}_2}[\bm{\psi}]\leq \mathsf{p}_{\infty}\big(|\bm{\psi}|_{\M^{\star}}\big)\qquad \text{$\m$-a.e. on $E$};\\
        \ADis^{\mathrm{l}}_{\mathsf{p}_2}[\bm{\psi}]\geq \frac{1}{\mathsf{p}_{\infty}\big(|\bm{u}|_{\M}\big)}\qquad \text{$\m$-a.e. on $E$}.
    \end{gather}
    \end{enumerate}
    \begin{proof}
    For the proof it suffices to juxtapose \cref{prop:UpADisRef} and \cref{prop:ADisEst} with \cref{prop:EstNormAbsNorm} and \cref{prop:EuclOrthDualOrth}.
    \end{proof}
\end{lem}

We are in a position to formulate the first main result of this section. Apart from some routine technicalities, the proof is nothing but a conjunction of \cref{theo:ExistAuerBasis} with the lemma above.
\begin{theo}\label{theo:MainEstADis}
    Let $E\in \P_{\m}(\X)$, $d\in \mathbb{N}_0$. Suppose $\dim_{\M}(E)\leq d$. Then there exists $\bm{\psi}\in \Big[\Mat^d\big(\M^{\star}|_E\big),\mathsf{p}_{\infty}\Big]_{\leq 1}$ such that the assertions below are valid.
    \begin{enumerate}
    \item[$\rm I)$] The following properties hold:
    \begin{gather}
        \ADis^{\mathrm{u}}_{\mathsf{p}_{\infty}}[\bm{\psi}]\leq  1\qquad \text{$\m$-a.e. on $E$}\label{eq:MainEstADis0};\\
        \ADis^{\mathrm{l}}_{\mathsf{p}_1}[\bm{\psi}]\geq 1\qquad \text{$\m$-a.e. on $E$} \label{eq:MainEstADis00}.
    \end{gather}
    \item[$\rm II)$] Suppose the module $\M$ is Euclidean over $E$. Then the following properties hold:
    \begin{gather}
       \ADis^{\mathrm{u}}_{\mathsf{p}_2}[\bm{\psi}]\leq 1\qquad \text{$\m$-a.e. on $E$}\label{eq:MainEstADis000};\\
        \ADis^{\mathrm{l}}_{\mathsf{p}_2}[\bm{\psi}]\geq 1\qquad \text{$\m$-a.e. on $E$}\label{eq:MainEstADis0000}.
    \end{gather}
    \end{enumerate}
    \begin{proof}
        From the premise it follows, with the use of \cref{prop:DimDecompMod}, that there is an $\m$-partition $(E_m)_{m\in \mathbb{N}}$ of $E$ such that, for each $m\in \mathbb{N}$ with $\m(E_m)>0$, the module $\M$ admits a $d_m$-basis over $E_m$ for some $d_m\in \overline{0,d}$.
        
        Fix $m\in \mathbb{N}$. By \cref{theo:ExistAuerBasis}, we know that there is an $\m$-partition $(E_{m,m'})_{m'\in \mathbb{N}}$ of $E_m$ such that, for each $m'\in \Nat$ with $\m(E_{m,m'})>0$, the module $\M$ admits also an Auerbach $d_m$-basis over $E_{m,m'}$, say $\bm{u}^{m,m'}$. Picking $m'\in \Nat$ with $\m(E_{m,m'})>0$ and putting $\bm{\psi}_{m,m'}\coloneq \Dual_{d_m}^{E_{m,m'}}(\bm{u}^{m,m'})$, we have, by \cref{def:AuerBasMod} and assertion I) of \cref{lem:ADisEst}, that
         \begin{equation}\label{eq:MainEstADis1}
         \ADis^{\mathrm{u}}_{\mathsf{p}_{\infty}}[\bm{\psi}_{m,m'}]\leq 1\qquad \text{$\m$-a.e. on $E_{m,m'}$}
        \end{equation}
        and that
        \begin{equation}\label{eq:MainEstADis2}
          \ADis^{\mathrm{l}}_{\mathsf{p}_1}[\bm{\psi}_{m,m'}]\geq 1\qquad \text{$\m$-a.e. on $E_{m,m'}$}.
        \end{equation}

        We now define the desired system $\bm{\psi}\in \Big[\Mat^d\big(\M^{\star}|_E\big),\mathsf{p}_{\infty}\Big]_{\leq 1}$ by the following formula:
        \begin{equation}
            \bm{\psi}\coloneqq \sum\limits_{\substack{m,m'\in \mathbb{N}\\ \m(E_{m,m'})>0}} \1_{E_{m,m'}} \bm{\psi}_{m,m'}.
        \end{equation}
        The corresponding correctness is easy to detect from all the information we have. It can also be readily checked, via \eqref{eq:MainEstADis1} and \eqref{eq:MainEstADis2}, that the properties in \eqref{eq:MainEstADis0} and \eqref{eq:MainEstADis00} indeed hold. This proves the first assertion. For the second one, we just need to apply assertion II) of \cref{lem:ADisEst}, in combination with \cref{prop:EuclModONBAB}, to get the estimate as in \eqref{eq:MainEstADis1} and \eqref{eq:MainEstADis2} but with $\mathsf{p}_2$ instead of, respectively, $\mathsf{p}_{\infty}$ and $\mathsf{p}_1$, which immediately implies the properties in \eqref{eq:MainEstADis000} and \eqref{eq:MainEstADis0000}. This finishes the proof.
    \end{proof}
\end{theo}

\begin{rem}\label{rem:RedAueBas}
    Let us make a comment on the following matter. Roughly speaking, the essence of \cref{theo:MainEstADis} lies in the fact that any given finite-dimensional normed module can be ``embedded'' into a normed vector space of the same dimension with universal bounds on both the ``direct'' and ``inverse'' pointwise norms. The importance of the first assertion in the theorem is that it holds regardless of the properties of the module, while the second one depends highly on the Euclideanity of the module. At the same time, the Euclidean part does not really rely on any of our results concerning Auerbach bases, since, as discussed in \cref{rem:ExONB}, the existence of orthonormal bases for Euclidean modules can be established by quite elementary means. The importance of this observation becomes clear in light of the following result (presented, for instance, in \cite[Lemma 10.17]{W18}): any finite-dimensional normed module admits an equivalent, up to a universal dimensional multiplicative constant, Euclidean norm. This is basically the module version of the well-known John ellipsoid theorem. In other words, all the machinery concerning Auerbach bases is somewhat redundant for our exposition from the utilitarian point of view. Nevertheless, we believe it helps to highlight the essence of our approach.
\end{rem}

\subsection{Approximation constituents} We turn our attention to establishing the necessary basis for our approximation arguments in the context of algebraic distortion factors.

\begin{rem}
It may seem that in the previous subsection we achieved the best possible situation, from the perspective of what is mentioned in \cref{rem:GoalAnDis}, concerning algebraic distortion factors. Namely, as \cref{theo:MainEstADis} demonstrates, we can always find a system whose upper algebraic distortion factor is bounded from above and whose lower algebraic distortion factor is bounded from below, with both estimates being sharp in a suitable sense. At the same time, in practice, one could be interested in finding merely an ``almost optimal'' (in regard to these bounds) system but which has some special form. And this will be exactly the case for us in the next section, which creates a need for obtaining additional approximation properties.
\end{rem}

We start by recording the following continuity property of algebraic distortion factors.
\begin{prop}\label{prop:LSCDisFF}
   Let $d\in \mathbb{N}_0$, $p_{\mathrm{u}},p_{\mathrm{l}}\in \mathsf{P}^d_+$, $\bm{\phi}\in \Big[\Mat^{d}(\M^{\star}),\mathsf{p}_{\infty}\Big]_{\leq 1}$, $(\bm{\phi}_n)_{n\in \Nat}\subseteq\Big[\Mat^{d}(\M^{\star}),\mathsf{p}_{\infty}\Big]_{\leq 1} $. Suppose the sequence $(\bm{\phi}_n)_{n\in \Nat}$ converges to $\bm{\phi}$ in $\Mat^d(\M^{\star})$. Then the following properties hold:
   \begin{gather}
   \text{the sequence $\Big(\ADis^{\mathrm{u}}_{p_{\mathrm{u}}}[\bm{\phi}_n]\Big)_{n\in \Nat}$ converges to $\ADis^{\mathrm{u}}_{p_{\mathrm{u}}}[\bm{\phi}]$ locally in $\m$-measure on $\X$};\label{eq:LSCDisFF0}\\
       \text{the sequence $\Big(\ADis^{\mathrm{l}}_{p_{\mathrm{l}}}[\bm{\phi}_n]\Big)_{n\in \Nat}$ converges to $\ADis^{\mathrm{l}}_{p_{\mathrm{l}}}[\bm{\phi}]$ locally in $\m$-measure on $\X$}.\label{eq:LSCDisFF00}
   \end{gather}
   \begin{proof}
   We will prove only the second property, as the first one can be verified in a completely analogous way.

   Fix $n\in \mathbb{N}$. Given any $v\in \M$, we can write, after recalling the relevant definitions, that
\begin{equation}
    \Big|p_{\mathrm{l}}\big(\bm{\phi}_n(v)\big)-p_{\mathrm{l}}\big(\bm{\phi}(v)\big)\Big|\leq p_{\mathrm{l}}(\bm{\phi}_n-\bm{\phi})|v|_{\M}\qquad \text{$\m$-a.e. on $\X$},
\end{equation}
which, with the help of \cref{def:DistCoef}, implies both that
\begin{equation}
    p_{\mathrm{l}}\big(\bm{\phi}_n(v)\big)\geq \Big(\ADis^{\mathrm{l}}_{p_{\mathrm{l}}}[\bm{\phi}]-p_{\mathrm{l}}(\bm{\phi}_n-\bm{\phi})\Big)|v|_{\M}\qquad \text{$\m$-a.e. on $\X$}
\end{equation}
and that
\begin{equation}
    p_{\mathrm{l}}\big(\bm{\phi}(v)\big)\geq \Big(\ADis^{\mathrm{l}}_{p_{\mathrm{l}}}[\bm{\phi}_n]-p_{\mathrm{l}}(\bm{\phi}_n-\bm{\phi})\Big)|v|_{\M}\qquad \text{$\m$-a.e. on $\X$}.
\end{equation}
Recalling \cref{def:DistCoef} again, we immediately get, due to the arbitrariness of $v$, that
\begin{equation}
    \ADis^{\mathrm{l}}_{p_{\mathrm{l}}}[\bm{\phi}_n]\leq \ADis^{\mathrm{l}}_{p_{\mathrm{l}}}[\bm{\phi}]+p_{\mathrm{l}}(\bm{\phi}_n-\bm{\phi})\qquad \text{$\m$-a.e. on $\X$}
\end{equation}
and that
\begin{equation}
    \ADis^{\mathrm{l}}_{p_{\mathrm{l}}}[\bm{\phi}]\leq \ADis^{\mathrm{l}}_{p_{\mathrm{l}}}[\bm{\phi}_n]+p_{\mathrm{l}}(\bm{\phi}_n-\bm{\phi})\qquad \text{$\m$-a.e. on $\X$}.
\end{equation}
The arbitrariness of $n$ and the convergence of $(\bm{\phi}_n)_{n\in \Nat}$ to $\bm{\phi}$ in $\Mat^d(\M^{\star})$ clearly give us the conclusion.

The proof is complete.
   \end{proof}
\end{prop}

The lemma below constitutes the core of the current subsection. Namely, the statement therein provides a sufficient condition for a family of systems to be dense in the corresponding ``unit ball'' (with respect to the $\mathsf{p}_{\infty}$-norm specifically) over a reference set, which has a clear importance in regard to the previous proposition.
\begin{lem}\label{lem:DensOfPurElem}
    Let $E\in \P_{\m}(\X)$, $d\in \Nat_0$. Suppose $\dim_{\M}(E)\leq d$. Let $\Phi\subseteq [\M^{\star}]_{\leq 1}$ be a disked, stable set that is norming for $\M$ over $E$. Then the following equality holds:
    \begin{equation}\label{eq:DensOfPurElem0}
\Cl_E\big(\Mat^d(\Phi)\big)=\Big[\Mat^d\big(\M^{\star}|_E\big),\mathsf{p}_{\infty}\Big]_{\leq 1}.
    \end{equation}
\end{lem}

To prove the above lemma, we need some preparation first.

As one can imagine, for \cref{lem:DensOfPurElem} to hold, a somewhat ``random version'' of the hyperplane separation theorem must be valid, in finite dimensions at least. In the specific case of ``coordinate modules'' this is exactly the essence of the following lemma.
\begin{lem}\label{lem:SepThConcr}
    Let $E\in \P_{\m}(\X)$, $d\in \mathbb{N}_0$. Let $\bm{W},\widetilde{\bm{W}}\subseteq \Mat^d(\Leb_{\m})$ be nonempty, $\Leb_{\m}$-convex sets with
    \begin{equation}\label{eq:SepThConcr0}
        \bigwedge\limits_{\bm{w}\in \bm{W},\widetilde{\bm{w}}\in \widetilde{\bm{W}}} \mathsf{p}_{\infty}(\widetilde{\bm{w}}-\bm{w})>0 \qquad \text{$\m$-a.e. on $E$}.
    \end{equation}
    Then there exists $\widehat{\bm{w}}\in \Mat^d(\Leb^{\infty}_{\m})$ such that
    \begin{equation}\label{eq:SepThConcr00}
        \bigvee\limits_{\bm{w}\in \bm{W}} \widehat{\bm{w}}\cdot  \bm{w} < \bigwedge\limits_{\widetilde{\bm{w}}\in \widetilde{\bm{W}}} \widehat{\bm{w}}\cdot  \widetilde{\bm{w}} \qquad \text{$\m$-a.e. on $E$}.
    \end{equation}
    \begin{proof}
        The statement can be derived from \cite[Theorem 6.1]{CKV15}.
    \end{proof}
\end{lem}

We now transfer the above lemma to the setting of ``abstract'' normed modules, in the form that we will exploit subsequently. Even though such a formulation is probably present in the literature and the proof below lies essentially in a straightforward application of \cref{lem:SepThConcr}, we prefer to present the necessary details for the completeness of the exposition.
\begin{lem}\label{lem:SepMod}
    Let $E\in \P_{\m}(\X)$, $d\in \mathbb{N}_0$. Suppose the module $\M$ admits a $d$-basis over $E$. Let $\Phi,\widetilde{\Phi}\subseteq \M^{\star}$ be nonempty, $\Leb_{\m}$-convex sets with
    \begin{equation}\label{eq:SepMod0}
        \bigwedge\limits_{\phi\in \Phi,\widetilde{\phi}\in \widetilde{\Phi}}\big|\widetilde{\phi}-\phi\big|_{\M^{\star}}>0\qquad \text{$\m$-a.e. on $E$}.
    \end{equation}
    Then there exists $\widehat{v}\in [\M]_{\leq 1}$ such that
    \begin{equation}\label{eq:SepMod00}
        \bigvee\limits_{\phi\in \Phi} \phi(\widehat{v}) < \bigwedge\limits_{\widetilde{\phi}\in \widetilde{\Phi}}\widetilde{\phi}(\widehat{v})\qquad \text{$\m$-a.e. on $E$}.
    \end{equation}
    \begin{proof}
            By the premise, there is a system $\bm{u}\in \Mat^d(\M)$ forming a $d$-basis of $\M$ over $E$. We then put $\bm{\psi}\coloneqq \Dual^E_d(\bm{u})$ and consider the families
            \begin{gather}
                \bm{\Lambda}\coloneqq \Big\{\bm{\lambda}\in \Mat^{d}(\Leb_{\m}) \bigm| \bm{\lambda}\cdot \bm{\psi}\in \Phi|_E\Big\},\\
                 \widetilde{\bm{\Lambda}}\coloneqq \Big\{\widetilde{\bm{\lambda}} \in \Mat^{d}(\Leb_{\m}) \bigm| \widetilde{\bm{\lambda}}\cdot \bm{\psi}\in \widetilde{\Phi}|_E\Big\},
            \end{gather}
            which can be easily seen to be $\Leb_{\m}$-convex. We now want to apply \cref{lem:SepThConcr}, with respect to the above-introduced sets, for the possibility of which we just need to check the validity of \eqref{eq:SepThConcr0}. To this end, we argue as follows. It is a direct consequence of the finite-dimensionality of $\M$ over $E$ that there is $\varrho\in [\Leb_{\m}]_{\geq 0}$ (depending, in fact, only on $d$ and $|\bm{\psi}|_{\M^{\star}}$) such that, given $\bm{\lambda}\in  \bm{\Lambda}$ and $\widetilde{\bm{\lambda}}\in  \widetilde{\bm{\Lambda}}$, we have
            \begin{equation}\label{eq:SepMod1}
                \Big|\big(\widetilde{\bm{\lambda}}-\bm{\lambda}\big)\cdot \bm{\psi}\Big|_{\M^{\star}}\leq \varrho \mathsf{p}_{\infty}\big(\widetilde{\bm{\lambda}}-\bm{\lambda}\big)\qquad \text{$\m$-a.e. on $E$}.
            \end{equation}
            Bearing in mind that the system $\bm{\psi}$ forms a $d$-basis of $\M^{\star}$ over $E$, once we juxtapose \eqref{eq:SepMod0} and \eqref{eq:SepMod1}, we indeed get the necessary condition. Hence, after looking at \eqref{eq:SepThConcr00}, we can easily gain the existence of $\widehat{\bm{\xi}}\in \Mat^d(\Leb^{\infty}_{\m})$ such that
            \begin{equation}\label{eq:SepMod2}
                \bigvee\limits_{\bm{\lambda}\in \bm{\Lambda}}\bm{\lambda}\cdot \widehat{\bm{\xi}}< \bigwedge\limits_{\widetilde{\bm{\lambda}}\in \widetilde{\bm{\Lambda}}} \widetilde{\bm{\lambda}}\cdot \widehat{\bm{\xi}}\qquad \text{$\m$-a.e. on $E$}.
            \end{equation}
    It remains only to put
    \begin{equation}
       \widehat{v}\coloneqq \frac{\widehat{\bm{\xi}}\cdot \bm{u}}{\1_{\X}\vee \big|\widehat{\bm{\xi}}\cdot \bm{u}\big|_{\M}},
    \end{equation}
    which is clearly well-posed, and notice that the condition in \eqref{eq:SepMod2} implies exactly the desired condition in \eqref{eq:SepMod00}. This finishes the proof.
    \end{proof}
\end{lem}

Let us record also the following auxiliary fact.
\begin{prop}\label{prop:ClConv}
    Let $E\in \P_{\m}(\X)$, let $\Phi\subseteq [\M^{\star}]_{\leq 1}$ be a convex, stable set. Then the set $\Cl_E(\Phi)$ is $\Leb_{\m}$-convex.
    \begin{proof}
       To begin with, we notice that the set $\Cl_E(\Phi)$ is still convex and stable, as can be readily checked via the relevant definitions. Then, we observe that, given any $\lambda\in \mathbb{R}$, any $E'\in \P_{\m}(\X)$, and any $\phi,\phi'\in \M^{\star}$, we have the following:
       \begin{equation}
           (\lambda\1_{E'}) \phi+(\1_{\X}-\lambda\1_{E'})\phi'=\1_{E'}\Big(\lambda \phi+(1-\lambda)\phi'\Big)+\1_{(\X\backslash E')}\phi'.
       \end{equation}
       With this in mind, a routine verification with the use of our assumptions demonstrates that, given any finitely valued function $\lambda\in \Leb_{\m}$ with
       \begin{equation}\label{eq:ClConv1}
           0\leq\lambda\leq 1\qquad \text{$\m$-a.e. on $\X$}
       \end{equation}
       and any $\phi,\phi'\in \Cl_E(\Phi)$, we have the inclusion
        \begin{equation}
          \lambda\phi+(\1_{\X}- \lambda)\phi'\in \Cl_E(\Phi).
       \end{equation}
       As any function $\lambda\in \Leb_{\m}$ satisfying the condition in \eqref{eq:ClConv1} can be approximated, in the sense of the local convergence in $\m$-measure on $\X$, by finitely valued functions in $\Leb_{\m}$ satisfying the same condition, we easily come to the desired conclusion.
    \end{proof}
\end{prop}

Using the proposition above, we now obtain a sufficient condition for a subset of a normed module to be dense in the corresponding unit ball. This can also be seen to be a $1$-dimensional version of \cref{lem:DensOfPurElem}. In particular, the key reasoning behind the lemma lies exactly in the statement below.
\begin{prop}\label{prop:DensNormSet}
    Let $E\in \P_{\m}(\X)$. Suppose $\dim_{\M}(E)<+\infty$. Let $\Phi\subseteq [\M^{\star}]_{\leq 1}$ be a disked, stable set that is norming for $\M$ over $E$. Then the following equality holds:
    \begin{equation}\label{eq:DensNormSet0}
        \Cl_E(\Phi)=\big[\M^{\star}|_E\big]_{\leq 1}.
    \end{equation}
    \begin{proof}
    Put $\wbar{\Phi}\coloneqq \Cl_E(\Phi)$. Arguing by contradiction, we suppose that the equality in \eqref{eq:DensNormSet0} does not hold, whence we can find $\phi_0\in \Big(\big[\M^{\star}|_E\big]_{\leq 1}\big\backslash \wbar{\Phi}\Big)$. After a routine check, this implies, with the use of the closedness of $\overline{\Phi}$, that, for some $E'\in \P^+_{\m}(E)$, we have
   \begin{equation}
       \bigwedge\limits_{\phi\in \widebar{\Phi}}|\phi-\phi_0|_{\M^{\star}}>0\qquad \text{$\m$-a.e. on $E'$}.
   \end{equation}
   Next, by \cref{prop:DimDecompMod}, we can find $E''\in \P_{\m}^+(E')$ such that the module $\M$ admits a basis over $E''$. By \cref{prop:ClConv}, we know that the set $\widebar{\Phi}$ is $\Leb_{\m}$-convex and hence that the set $\widebar{\Phi}|_{E''}$ is $\Leb_{\m}$-convex, so, applying \cref{lem:SepMod} with respect to $\widebar{\Phi}|_{E''}$ and $\{\1_{E''}\bm{\phi}_0\}$, we can guarantee the existence of $\widehat{v}\in [\M]_{\leq 1}$ such that
   \begin{equation}
        \bigvee\limits_{\phi\in \Phi} \phi(\widehat{v})\leq \bigvee\limits_{\phi\in \widebar{\Phi}} \phi(\widehat{v})<\phi_0(\widehat{v})\qquad \text{$\m$-a.e. on $E''$}.
   \end{equation}
   Exploiting our assumptions, we can get that
   \begin{equation}
      |\widehat{v}|_{\M}<\big|\phi_0(\widehat{v})\big|_{\Leb_{\m}}\qquad \text{$\m$-a.e. on $E''$}
   \end{equation}
   and hence that
    \begin{equation}
       |\widehat{v}|_{\M}<|\widehat{v}|_{\M}\qquad \text{$\m$-a.e. on $E''$},
   \end{equation}
   which is the desired contradiction. The proof is thus complete.
    \end{proof}
\end{prop}

Now we are in a position to provide the (rather straightforward) proof for \cref{lem:DensOfPurElem}.
\begin{proof}[Proof of \cref{lem:DensOfPurElem}]

Pick $\bm{\phi}\in \Big[\Mat^d\big(\M^{\star}|_E\big),\mathsf{p}_{\infty}\Big]_{\leq 1}$ and fix $j\in \overline{1,d}$. As it clearly holds that $\phi^j\in \big[\M^{\star}|_E\big]_{\leq 1}$, by \cref{prop:DensNormSet}, we can find a sequence $(\phi^j_n)_{n\in \Nat}\subseteq \Phi$ that converges in $\M^{\star}$ to $\phi^j$. For each $n\in \Nat$, let $\bm{\phi}_n$ be given by the components $\phi_n^j$, $j\in \overline{1,d}$, which can be seen to belong to $\Mat^d(\Phi)$. We thus immediately deduce that the sequence $(\bm{\phi}_n)_{n\in \Nat}$ converges to $\bm{\phi}$ in $\Mat^d(\M^{\star})$, as desired.
\end{proof}

Now, based on what is already obtained in this subsection, we formulate the second main result of the current section.
\begin{theo}\label{theo:LACisApprox}
    Let $E\in \mathrm{P}_{\m}(\X)$, $d\in \Nat_0$. Suppose $\dim_{\M}(E)\leq d$. Let $\Phi\subseteq [\M^{\star}]_{\leq 1}$ be a disked, stable set that is norming for $\M$ over $E$, let $\bm{\phi}\in \big[\Mat^d(\M^{\star}),\mathsf{p}_{\infty}\big]_{\leq 1}$, $p_{\mathrm{u}},p_{\mathrm{l}}\in \mathsf{P}^d_+$, $\zeta\in (0,1)$. Then there exists an $\m$-partition $(E_m)_{m\in \Nat}$ of $E$ and a family $(\bm{\phi}_m)_{m\in \Nat}\subseteq \Mat^d(\Phi)$ such that, for each $m\in \Nat$, the following properties hold:
    \begin{gather}
    \zeta\ADis_{p_{\mathrm{u}}}^{\mathrm{u}}[\bm{\phi}_m]\leq \ADis_{p_{\mathrm{u}}}^{\mathrm{u}}[\bm{\phi}]\qquad \text{$\m$-a.e. on $E_m$};\\
        \ADis_{p_{\mathrm{l}}}^{\mathrm{l}}[\bm{\phi}_m]\geq \zeta\ADis_{p_{\mathrm{l}}}^{\mathrm{l}}[\bm{\phi}]\qquad \text{$\m$-a.e. on $E_m$}.
    \end{gather}
    \begin{proof}
        The statement follows almost immediately from a combination of \cref{prop:LSCDisFF} and \cref{lem:DensOfPurElem}, so we omit routine details.
    \end{proof}
\end{theo}

\section{Application to metric measure spaces}\label{ss:AppMMS}
This section is dedicated to the proof of our main results, namely \cref{theo:MainRes} and \cref{theo:MainResEucl}. More specifically, here we introduce objects that we name geometric distortion factors, which we then relate with algebraic distortion factors, considered with respect to a suitably chosen normed module, the module of so-called Weaver derivations. And after that we combine this with the obtained information concerning bounds on algebraic distortion factors, which leads us to the desired final statements.

Until the end of the section, we assume fixed a separable, locally complete metric space $\X=(\X,\dX)$. Exactly in terms of it, our eventual results are to be formulated.

\begin{rem}
    In what follows, we will often refer to the previously mentioned work \cite{BEBS24} by Bate--Eriksson-Bique--Soultanis, where the authors deal specifically with metric spaces that are separable and complete. At the same time, essentially all the discussion therein can be transferred directly to the locally complete setting too, so we will take this for granted in some cases and apply the corresponding results to our situation as well.
\end{rem}

\subsection{Metric spaces and curve fragments}\label{ss:MgCf}
As a start, we wish to provide some necessary information concerning metric spaces, especially what surrounds the analysis with so-called curve fragments.

In the sequel, by $\mathcal{L}^1$, we will mean the $1$-dimensional Lebesgue measure on $\mathbb{R}$. Also, given a set $S$, by $\mathscr{N}(S)$, we will understand the cardinality of $S$ as valued in $\Nat_0\cup \{+\infty\}$.

By a Borel measure on a topological space $\mathsf{T}$, we will mean a measure on $\mathsf{T}$ whose corresponding sigma-algebra coincides with the Borel sigma-algebra of $\mathsf{T}$.

We now recall several notions related to the $1$-dimensional area formula in general metric spaces. So, let $\Y=(\Y,\dY)$ be a metric space.

By $\mathcal{H}^1_{\Y}$, we will mean the $1$-dimensional Hausdorff measure on $\Y$. If we are given $d\in \Nat_0$ and $p\in \mathsf{P}^d_+$, we also put $\mathcal{H}^1_p\coloneqq \mathcal{H}^1_{\R^d_p}$ for brevity.

Let $T\subseteq \R$, let $\beta\colon T\to \Y$ be a continuous map. In what follows, we will keep in mind the following fact: given a Borel set $T'\subseteq \R$ with $T'\subseteq T$, the map
\begin{equation}
    \Y\to \overline{\R}_{\geq 0}\colon y\mapsto \mathscr{N}\big(\beta^{-1}(y)\cap T'\big)
\end{equation}
is Carath{\'e}odory $\mathcal{H}^1_{\Y}$-measurable. For the proof of this, we refer the reader to the reasoning given in \cite[Theorem 2.10.10]{F14}. Next, by the metric speed of $\beta$, we will mean the function $\ms_{\beta}\colon T\to \R$ given for any $t\in T$ by the following rule:
    \begin{equation}\label{eq:MetrSp}
            \ms_{\beta}(t)\coloneqq \begin{dcases}
            \lim\limits_{\substack{t'\to t \\ t'\in T}} \frac{\dY(\beta_t,\beta_{t'})}{|t'-t|},\quad \parbox{20em}{\center{the point $t$ is a limit point of $T$\\ and the limit exists in $\mathbb{R}$,}}\\
            0,\qquad\qquad\qquad\qquad\qquad \text{otherwise;}
        \end{dcases}
    \end{equation}
    this function can be checked by standard means to be Borel. For more information on this object, we send the reader, for instance, to \cite[Section 4.1]{AT04}. We wish also to emphasize that, as follows from the Kirchheim extension of the Rademacher theorem given in \cite[Theorem 2]{K94}, if the map $\beta$ is Lipschitz, then, for $\mathcal{L}^1$-a.e. $t\in T$, exactly the first alternative in \eqref{eq:MetrSp} is applied. This should always be borne in mind, as eventually we will deal with the metric speed of Lipschitz maps only.

    We also point out the following simple consequence of the given definitions: given $T,T'\subseteq \R$ with $T'\subseteq T$ and a continuous map $\beta\colon T\to \Y$, once we put $\beta'\coloneqq \beta|_{T'}$, we get that
    \begin{equation}\label{eq:MSUniq}
    \begin{gathered}
        \ms_{\beta'}(t)=\ms_{\beta}(t)\\
        \text{for every $t\in T'$}\\
        \text{which is a limit point of $T'$ and at which the first alternative in \eqref{eq:MetrSp} for $\beta$ is applied};
        \end{gathered}
    \end{equation}
    so, in particular, if the map $\beta$ is Lipschitz, then, from what is said above, we get that the equality in \eqref{eq:MSUniq} holds for $\mathcal{L}^1$-a.e. $t\in T'$.

It is well known that the metric speed plays the role of Jacobian for $1$-dimensional Lipschitz maps from the perspective of the area formula. We incorporate this in the following proposition.
    \begin{prop}\label{prop:ArForm}
        Let $T\subseteq \R$, let $\beta\colon T\to \Y$ be a Lipschitz map. Then, given a Borel set $T'\subseteq \R$ with $T'\subseteq T$, one has
        \begin{equation}\label{eqref:ArForm0}
            \int\limits_{T'} \ms_{\beta}(t)\D \mathcal{L}^1(t)=\int\limits_{\Y} \mathscr{N}\big(\beta^{-1}(y)\cap T'\big)\D \mathcal{H}^1_{\Y}(y).
        \end{equation}
        \begin{proof}
            The statement can be extracted from \cite[Theorem 7]{K94}.
        \end{proof}
    \end{prop}

\begin{rem}\label{rem:MultFIntVar}
It is useful to mention for the subsequent discussion that the integral as in the right-hand side of \eqref{eqref:ArForm0} still measures the variation of the corresponding map (which is represented by exactly the left-hand side therein) even if that map is merely continuous, at least in the case when its initial domain is an interval. This can be deduced, for instance, on the basis of \cite[Theorem 2.10.13]{F14}.
\end{rem}

Relying on the previous proposition and remark, we now introduce the notion of length for $1$-dimensional maps that are not assumed to be continuous, which (as mentioned in \cref{ss:FrR}) is a crucial step for our exposition. So, given $T\subseteq \R$ and a map $\beta\colon T\to \Y$, by the length of $\beta$, we will mean the quantity
\begin{equation}
    \Len_{\Y}(\beta)\coloneqq \sup\limits_{\kappa} \int\limits_{\Y} \kappa(y)\D \mathcal{H}^1_{\Y}(y),
\end{equation}
where the supremum is taken over all Carath{\'e}odory $\mathcal{H}^1_{\Y}$-measurable functions $\kappa\colon \Y\to \overline{\R}_{\geq 0}$ with
\begin{equation}
    \kappa(y)\leq \mathscr{N}\big(\beta^{-1}(y)\big)\qquad\text{for $\mathcal{H}^1_{\Y}$-a.e. $y\in \Y$}.
\end{equation}
In other words, the presented object is nothing but the lower $\mathcal{H}^1_{\Y}$-integral of the standard multiplicity function of the given map; this is indeed a necessary modification here due to the possible lack of the measurability of the latter one. In particular, we can conclude thereby that the new definition matches the usual one (involving the total variation) for continuous maps (defined on intervals). At the same time, as can be extracted from the subsequent part, for those $1$-dimensional maps that will be in place in our further proofs, the corresponding multiplicity function will, in fact, be measurable, so the modified length is needed mostly for simplicity of formulations.

Again, if we are given $d\in \Nat_0$ and $p\in \mathsf{P}^d_+$, we put $\Len_{p}\coloneqq \Len_{\R^d_p}$ for brevity.

We return to discussing notions that will be considered solely with respect to $\X$.

Given $S\subseteq \X$, by $\mathrm{B}(S)$, we will mean the family of all Borel sets in $\X$ that are subsets of $S$. Also, given $Q\in \mathrm{B}(\X)$, by a Borel partition of $Q$, we will mean a disjoint family $(Q_m)_{m\in \mathbb{N}}\subseteq \mathrm{B}(Q)$ such that
\begin{equation}
    Q=\bigcup\limits_{m\in \Nat}Q_m.
\end{equation}

Given $S\subseteq \X$, we will use the notation $\dim_{\mathrm{H}}(S)$ for the Hausdorff dimension of $S$.

Given a Borel measure $\mu$ on $\X$ and $Q\in \mathrm{B}(\X)$, by $\mu_{\rmes Q}$, we will mean the Borel measure on $\X$ given for any $Q'\in \mathrm{B}(\X)$ by
\begin{equation}
    \mu_{\rmes Q}(Q')\coloneqq \mu(Q'\cap Q).
\end{equation}

Let $\mu$, $\mu_0$ be Borel measures on $\X$. We will write $\mu\ll \mu_0$ whenever the measure $\mu$ is absolutely continuous with respect to $\mu_0$. If both measures are sigma-finite and $\mu\ll \mu_0$, by $\dfrac{\D \mu}{\D \mu_0}$, we will mean the corresponding Radon--Nikodym derivative as an element of $[\Leb_{\mu_0}]_{\geq 0}$.

Below we introduce the necessary terminology concerning the notion of curve fragments and record several related statements.

By a curve fragment in $\X$, or just fragment in what follows, we will mean a continuous map from a nonempty, compact subset of $\mathbb{R}$ to $\X$. Given a fragment $\gamma$, by $\dom(\gamma)$ and $\im(\gamma)$, we will denote its domain and image, respectively. A fragment $\gamma$ will be called a curve if the set $\dom(\gamma)$ is an interval. Given fragments $\gamma$ and $\gamma'$, we will say that the fragment $\gamma'$ is a subfragment of $\gamma$ if $\dom(\gamma')\subseteq \dom(\gamma)$ and $\gamma|_{\dom(\gamma')}=\gamma'$. Given a fragment $\gamma$ and $t\in \dom(\gamma)$, the value of $\gamma$ at $t$ will be denoted by $\gamma_t$. Fragments $\gamma$ and $\gamma'$ will be called equivalent if, for some compact set $T\subseteq \R$ and for some nondecreasing, continuous functions $\sigma,\sigma'\colon T\to \R$ with $\sigma(T)=\dom(\gamma)$ and $\sigma'(T)= \dom(\gamma')$, we have $\gamma'\circ \sigma'=\gamma\circ \sigma$. A fragment $\gamma$ will be called Lipschitz if the map $\gamma$ is Lipschitz. A fragment $\gamma$ will be called biLipschitz if the map $\gamma$ is biLipschitz. A fragment $\gamma$ will be called rectifiable if it is equivalent to a Lipschitz fragment.

\begin{rem}
Let us note that the introduced notion of equivalence of fragments is nothing but a natural extension of the analogous concept in the realm of curves. In turn, the definition of rectifiable fragments is justified by the well-known fact that any rectifiable curve is equivalent to a Lipschitz curve, through its natural parameterization for instance. All this allows us to deal with curve fragments in a parameterization-invariant way.
\end{rem}

We record the following important fact concerning equivalence of fragments.
\begin{prop}\label{prop:FromRectToLip}
    Let $\gamma_1,\gamma_2$ be rectifiable fragments that are equivalent. Then it holds that
    \begin{equation}
\mathscr{N}\big(\gamma_1^{-1}(x)\big)=\mathscr{N}\big(\gamma_2^{-1}(x)\big)\qquad \text{for $\mathcal{H}^1_{\X}$-a.e. $x\in \X$}.
    \end{equation}
    \begin{proof}
    The statement is folklore in the case of curves. The formulation with curve fragments can be proved by essentially the same means, so we omit further details.
    \end{proof}
\end{prop}

\begin{rem}
    As can be seen from the proposed definition of curve fragments, those form a more natural replacement for curves in spaces lacking the latter, and exactly due to this caveat we prefer to deal everywhere further with curve fragments rather than curves only. For a reference source for us in this context, we send the reader to \cite[Section 2.3]{BEBS24}. At the same time, our exposition differs from the one therein, where curve fragments are a priori understood as being biLipschitz in our terminology. This (essentially insignificant) modification is made by us to achieve consistency with the problem formulations in \cref{ss:BackInf}, as those deal with all possible curves, rather than, e.g., Lipschitz ones, and therefore are parametrization-invariant. In this regard, we also warn the reader that different parts of the further discussion (depending on our needs) operate with different types of curve fragments, namely with one of the above introduced classes, which is always advisable to track. 
\end{rem}

Let $S\subseteq \X$. The families of all fragments, of all rectifiable fragments, of all Lipschitz fragments, and of all biLipschitz fragments whose image lies in $S$ will be denoted, respectively, by $\Fr(S)$, $\Fr_{\mathrm{r}}(S)$, $\Fr_{\mathrm{L}}(S)$, and $\Fr_{\mathrm{bL}}(S)$. Also, given $n\in \mathbb{N}$, we let $\Fr_{\mathrm{bL}}^n(S)$ stand for the family of all $\gamma\in \Fr(S)$ such that
\begin{equation}
    \frac{1}{n} |t'-t|\leq \dX(\gamma_t,\gamma_{t'})\leq n|t'-t|\qquad \text{for all $t,t'\in \dom(\gamma)$}.
\end{equation}
We thus have that
\begin{equation}
    \Fr_{\mathrm{bL}}(S)=\bigcup\limits_{n\in \mathbb{N}}\Fr_{\mathrm{bL}}^n(S).
\end{equation}
Finally, for the sake of consistency, we let $\Cur(\X)$ denote the family of all curves, which forms then a subfamily of $\Fr(\X)$. It, however, will not be of any explicit use later, as all the statements will deal with curve fragments directly, but of course curves will be implicitly present everywhere below.

Given $\G\subseteq \Fr(\X)$ and a closed set $C\subseteq \X$, we put
\begin{equation}
    \G|_C\coloneqq \Big\{\gamma|_{\gamma^{-1}(C)}\bigm| \text{$\gamma\in \G$, $\gamma^{-1}(C)\neq \varnothing$}\Big\},
\end{equation}
which can be seen to define a subfamily of $\Fr(C)$.

Given $Q\in \mathrm{B}(\X)$, we put
\begin{equation}
    \Gamma^+_Q(\X)\coloneqq \Bigg\{\gamma\in \Fr_{\mathrm{r}}(\X) \biggm| \int\limits_Q \mathscr{N}\big(\gamma^{-1}(x)\big)\D \mathcal{H}^1_{\X}(x)>0\Bigg\}.
\end{equation}
As a clarification, this family basically consists of all rectifiable fragments that cover ``positive length'' in the given set.

We shall topologize the family $\Fr(\X)$, for which we do the following. Let $\mathscr{H}(\X\times \R)$ denote the collection of all nonempty, compact subsets of $\X\times \R$ endowed with the topology induced by the Hausdorff distance therein, which can be verified to be a Polish space under our assumptions; for some brief information on all this, we refer the reader to \cite[Section 4.F]{K12}. We then identify $\Fr(\X)$ with a subset of $\mathscr{H}(\X\times \R)$ via the injective map $\Fr(\X)\to \mathscr{H}(\X\times \R)$ that sends $\gamma\in \Fr(\X)$ to its graph, i.e., the set $\big\{(\gamma_t,t)\mid t\in \dom(\gamma)\big\}$. With this, one can check that the families $\Fr_{\mathrm{L}}(\X)$ and $\Fr_{\mathrm{bL}}(\X)$ form Borel subfamilies of $\Fr(\X)$. As one more important feature of this construction, which can be easily checked via the definitions involved, let us mention the following: given a sequence $(\gamma^n)_{n\in \mathbb{N}}\subseteq \Fr(\X)$ that converges to $\gamma\in \Fr(\X)$ in $\Fr(\X)$, it holds that the sequence $\big(\dom(\gamma^n)\big)_{n\in \mathbb{N}}$ converges to $\dom(\gamma)$ with respect to the Hausdorff distance on $\mathbb{R}$.

At some moment later, we will need the following simple consequence of the given definitions: the function
\begin{equation}
    \Fr(\X)\to \mathbb{R}\colon \gamma\mapsto \mathcal{L}^1\big(\dom(\gamma)\big)
\end{equation}
is upper-semicontinuous. This fact is standard and can be checked by quite elementary means. As a side comment, we note that, in general, while the $1$-Hausdorff measure is not always upper-semicontinuous, the so-called $1$-Hausdorff content, which turns out to coincide specifically on $\mathbb{R}$ with the former and hence with the $1$-dimensional Lebesgue measure, does possess this feature; for this fact, see \cite[Lemma 1.5]{AGPP22}.

Next, we introduce the collections
\begin{gather}
    \overline{\Fr}(\X)\coloneqq \Big\{(\gamma,t)\in \Fr(\X)\times \mathbb{R} \bigm| t\in \dom(\gamma)\Big\},\\
    \overline{\Fr}_{\mathrm{L}}(\X)\coloneqq \overline{\Fr}(\X)\cap\big(\Fr_{\mathrm{L}}(\X)\times \R\big),\\
    \overline{\Fr}_{\mathrm{bL}}(\X)\coloneqq \overline{\Fr}(\X)\cap\big(\Fr_{\mathrm{bL}}(\X)\times \R\big)
\end{gather}
and equip them with the topologies induced from the product topology on $\Fr(\X)\times \R$. Here we note that the family $\overline{\Fr}_{\mathrm{bL}}(\X)$ forms a Borel subfamily of $\overline{\Fr}(\X)$. In addition to this, we consider the natural evaluation map
\begin{equation}
    \ev\colon \overline{\Fr}(\X)\to \X\colon (\gamma,t)\mapsto\gamma_t,
\end{equation}
which, as can be readily verified via the corresponding definitions, turns out to be continuous and hence Borel. Moreover, given a Borel measure $\nu$ on $\overline{\Fr}(\X)$, by $\ev_{\#}(\nu)$, we will mean the Borel measure on $\X$ obtained as the pushforward of $\nu$ under $\ev$, that is, we have
\begin{equation}
    \ev_{\#}(\nu)(Q)=\nu\big(\ev^{-1}(Q)\big)\qquad \text{for any $Q\in \mathrm{B}(\X)$}.
\end{equation}

    It is convenient to introduce the metric speed functional, by which we will mean the function
    \begin{equation}
        \ms\colon \overline{\Fr}(\X)\to \R_{\geq 0}\colon(\gamma,t)\mapsto \begin{dcases}
            \ms_{\gamma}(t),\quad \gamma\in \Fr_{\mathrm{L}}(\X),\\
            0,\quad \quad \quad \;\text{otherwise}.
        \end{dcases}
    \end{equation}
    A folklore fact in this context is that the resulting function is Borel, which can be verified via all the exploited definitions. Here, we notably ignore all fragments that are not Lipschitz, which, however, completely suffices for our needs.

We now recall the notion of line integration, in a somewhat nonstandard form, as we need to suitably adapt it to the setting of curve fragments. The corresponding motivation for the forthcoming definition can be extracted again from \cref{prop:ArForm}. So, given $\gamma\in \Fr(\X)$ and an $\mathcal{H}^1_{\X}$-measurable function $\rho\colon \X\to \overline{\R}_{\geq 0}$, we put
\begin{equation}
    \int\limits_{\gamma}\rho\coloneqq \begin{dcases} \int\limits_{\X} \rho(x)\mathscr{N}\big(\gamma^{-1}(x)\big) \D \mathcal{H}^1_{\X}(x),\quad\gamma\in \Fr_{\mathrm{r}}(\X),\\
    +\infty,\quad \quad\quad\quad\qquad\qquad\;\;\;\quad\quad\text{otherwise}.
    \end{dcases}
\end{equation}
With this definition, we have, given any $Q\in \mathrm{B}(\X)$, that
\begin{equation}\label{eq:LinIntInjFr}
\begin{gathered}
    \int\limits_{\gamma}\1_{Q}=\mathcal{H}^1_{\X}\big(\im(\gamma)\cap Q\big)\\ \text{for every $\gamma\in \Fr_{\mathrm{r}}(\X)$ such that the map $\gamma$ is injective}.
    \end{gathered}
\end{equation}

We need to introduce some notation and terminology related to Lipschitz functions on $\X$.
    
By $\LIP(\X)$, we will denote the collection of all Lipschitz functions $\X\to \mathbb{R}$, which clearly forms a vector space. Then, by $\LIP_1(\X)$, we will denote the subcollection of $\LIP(\X)$ made of all $1$-Lipschitz functions. Finally, by $\BLIP(\X)$, we will denote the subcollection of $\LIP(\X)$ made of all Lipschitz functions that are also bounded, which can be additionally seen to form an $\mathbb{R}$-algebra.

Let us note separately that, given $d\in \mathbb{N}_0$, any system $\bm{f}\in \Mat^d\big(\LIP(\X)\big)$ can be identified in the obvious way (which we will always tacitly do) with a function $\X\to \mathbb{R}^d$ becoming Lipschitz with respect to any choice of the norm on the target space, for which we will moreover use the same symbol.

    Given $\gamma\in \Fr(\X)$ and $f\in \LIP(\X)$, by the derivative of $f\circ \gamma$, we will mean the function $\partial(f\circ \gamma)\colon \dom(\gamma)\to \mathbb{R}$ given for any $t\in \dom(\gamma)$ by the following rule:
     \begin{equation}
           \big(\partial(f\circ \gamma)\big)(t)\coloneqq \begin{dcases}
            \lim\limits_{\substack{t'\to t \\ t'\in \dom(\gamma)}} \frac{f(\gamma_{t'})-f(\gamma_t)}{t'-t},\quad \parbox{18em}{\center{the point $t$ is a limit point of $\dom(\gamma)$\\ and the limit exists in $\mathbb{R}$,}}\\
            0,\qquad\qquad\qquad\qquad\qquad\quad \;\;\text{otherwise.}
        \end{dcases}
    \end{equation}
    It can be routinely verified, essentially as with the metric speed functional, that, given $f\in \LIP(\X)$, the function
    \begin{equation}
        \overline{\Fr}(\X)\to \R\colon (\gamma,t)\mapsto \begin{dcases}
            \big(\partial(f\circ \gamma)\big)(t),\quad \gamma\in \Fr_{\mathrm{L}}(\X),\\
            0,\qquad \qquad\qquad\text{otherwise},
        \end{dcases}
    \end{equation}
    is Borel. Moreover, given $p\in \mathsf{P}^{\infty}$, $\bm{f}\in \Mat^{\infty}\big(\LIP(\X)\big)$, and $\gamma\in \Fr_{\mathrm{L}}(\X)$, we put
    \begin{equation}
        \ms^{p}_{\bm{f}\circ \gamma}\coloneqq p\big(\partial(\bm{f}\circ \gamma)\big),
    \end{equation}
    which can be seen to be consistent with the previous notation, since, whenever $p\in \mathsf{P}^d_+$ and $f\in \Mat^d\big(\LIP(\X)\big)$ for some $d\in \Nat_0$, the above quantity ($\mathcal{L}^1$-a.e. on $\dom(\gamma)$) is nothing but the metric speed of $\bm{f}\circ \gamma$ as a map $\dom(\gamma)\to \R^d_p$. Thus, also the function
    \begin{equation}
        \overline{\Fr}_{\mathrm{L}}(\X)\to \R_{\geq 0}\colon (\gamma,t)\mapsto \ms_{\bm{f}\circ \gamma}^p(t)
    \end{equation}
    is once again Borel.

    At a certain point in the sequel, we will exploit the following standard fact: there can be found $(g_l)_{l\in \mathbb{N}}\subseteq \LIP_1(\X)$ such that it holds, for any $\gamma\in \Fr_{\mathrm{L}}(\X)$, that
    \begin{equation}
        \ms_{\gamma}(t)=\sup\limits_{l\in \mathbb{N}} \big(\partial (g_l\circ \gamma)\big)(t)\qquad \text{for $\mathcal{L}^1$-a.e. $t\in \dom(\gamma)$}.
    \end{equation}
    Typically, one considers the absolute value on the right in the above formula, but, by adding the negatives of the functions in the chosen family thereto, one can achieve what is above as well. For the proposed statement, see, for instance, \cite[Theorem 1.1.2]{AGS05}.

    In addition, let us record the following fact: given $d\in \Nat_0$, $p\in \mathsf{P}^d_+$ and $\bm{f}\in \Mat^d\big(\LIP(\X)\big)$, it holds, for any $\gamma\in \Fr_{\mathrm{L}}(\X)$, that
    \begin{equation}
        \ms^p_{\bm{f}\circ \gamma}(t)=\sup\limits_{\bm{\xi}\in \R^d, p^*(\bm{\xi})\leq 1} \bigg|\big(\partial \big((\bm{\xi}\cdot \bm{f})\circ \gamma\big)\Big)(t)\bigg|\qquad \text{for $\mathcal{L}^1$-a.e. $t\in \dom(\gamma)$}.
    \end{equation}
    This should be kept in mind later on.

For the remainder of the subsection, let $\mu$ be a given locally finite Borel measure on $\X$. As the space $\X$ is separable, we deduce easily that the measure $\mu$ is also sigma-finite.

Although our previous notation and terminology contain all the necessary information about $\mu$ according to our conventions, we prefer to redefine here some of those for the reader's convenience. So, given $S\subseteq \X$, we note first that $\mathrm{B}(S)=\P_{\mu}(S)$ and put then $\mathrm{B}_{\mu}^+(S)\coloneqq \P_{\mu}^+(S)$. In addition, a $\mu$-partition of $S$ will always be explicitly called a Borel $\mu$-partition of $S$.

We proceed with introducing the notion of fragment plans, the importance of which will become clear later, and discussing several related constructions. The corresponding presentation relies primarily on \cite[Section 2.5]{BEBS24} but involves some minor adaptations necessary to match our choice of the definition of curve fragments.

Let $\mathbb{P}$ be a finite Borel measure on $\Fr(\X)$. The assignment $\mathbb{P}^{\#}$ given by
\begin{equation}\label{eq:BarCMeas}
    \mathbb{P}^{\#}(Q)\coloneqq \int\limits_{\ev^{-1}(Q)} \ms\D(\mathbb{P}\otimes \mathcal{L}^1), \qquad \text{$Q\in \mathrm{B}(\X)$},
\end{equation}
can be checked to define a Borel measure on $\X$, which is typically called the barycenter of $\mathbb{P}$. With this, the measure $\mathbb{P}$ will be called a fragment plan on $(\X,\mu)$ if $\mathbb{P}\big(\Fr(\X)\backslash \Fr_{\mathrm{bL}}(\X)\big)=0$, $\mathbb{P}^{\#}(\X)<+\infty$, and $\mathbb{P}^{\#}\ll \mu$. Thus, fragment plans in our exposition, by definition, are concentrated on biLipschitz fragments only, which should always be kept in mind further. Moreover, we put in correspondence to $\mathbb{P}$ the respective Borel measure $\overline{\mathbb{P}}$ on $\overline{\Fr}(\X)$, which can be seen to be finite whenever the measure $\mathbb{P}$ is a fragment plan on $(\X,\mu)$, given for every Borel family $\mathcal{G}\subseteq \overline{\Fr}(\X)$ by
\begin{equation}\label{eq:ExtPlan}
    \overline{\mathbb{P}}(\mathcal{G})\coloneqq  \int\limits_{\mathcal{G}} \ms\D(\mathbb{P}\otimes \mathcal{L}^1).
\end{equation}
With all these definitions, it is possible to detect that
\begin{equation}
    \ev_{\#}(\overline{\mathbb{P}})=\mathbb{P}^{\#},
\end{equation}
which will be tacitly used in what follows.

Let $\mathbb{P}$ be a fragment plan on $(\X,\mu)$. Then, given a Borel set $\G\subseteq \Fr(\X)$, the measure $\mathbb{P}_{\rmes \G}$ on $\Fr(\X)$ given for any Borel family $\G'\subseteq \Fr(\X)$ by
\begin{equation}
    \mathbb{P}_{\rmes \G}(\G')\coloneqq \mathbb{P}(\G'\cap \G)
\end{equation}
can be easily checked to be a fragment plan on $(\X,\mu)$ as well.

For the sequel, we shall invoke the notion of disintegration of measures with respect to maps, in application to fragment plans specifically, so we provide the following statement.
\begin{prop}
 Let $\mathbb{P}$ be a fragment plan on $(\X,\mu)$. Then there exists a $\mu$-a.e. unique family $(\overline{\mathbb{P}}_x)_{x\in \X}$ of finite Borel measures on $\overline{\Fr}(\X)$ with the following properties:
   \begin{enumerate}
       \item[$\rm I)$] it holds that
       \begin{equation}
           \overline{\mathbb{P}}_x\Big(\overline{\Fr}(\X)\big\backslash \big(\ev^{-1}(x)\cap \overline{\Fr}_{\mathrm{bL}}(\X)\big)\Big)=0\qquad \text{for $\mu$-a.e. $x\in \X$};
       \end{equation}
       \item[$\rm II)$] for a Borel $\mu$-representative $\rho$ of $\dfrac{\D \mathbb{P}^{\#}}{\D \mu}$, it holds, with $Q\coloneqq \rho^{-1}(\R_{>0})$, that
       \begin{gather}
\overline{\mathbb{P}}_x\big(\overline{\Fr}(\X)\big)=\1_Q(x)\qquad \text{for $\mu$-a.e. $x\in \X$};
       \end{gather}
       \item[$\rm III)$] given any Borel family $\mathcal{G}\subseteq \overline{\Fr}(\X)$, the function
       \begin{equation}
           \X\to \mathbb{R}\colon x\mapsto \overline{\mathbb{P}}_x(\mathcal{G})
       \end{equation}
       is Carath{\'e}odory $\mathbb{P}^{\#}$-measurable;
       \item[$\rm IV)$] given any Borel family $\mathcal{G}\subseteq \overline{\Fr}(\X)$, one has
       \begin{equation}
\overline{\mathbb{P}}(\mathcal{G})=\int\limits_{\X}\overline{\mathbb{P}}_x(\mathcal{G})\D \mathbb{P}^{\#}(x).
       \end{equation}
   \end{enumerate}
   \begin{proof}
       The necessary reasoning is given in \cite[Section 4]{BEBS24}; see also \cite[Theorem 5.3.1]{AGS05}.
   \end{proof}
   \end{prop}

Finally, we recall the notion of $\infty$-modulus; we refer to \cite[Section 2.4]{BEBS24} as the corresponding main reference, in comparison with which, however, we again have some implicit differences of the nature discussed earlier (exactly as in the case of fragment plans). For a comprehensive study of moduli in general, we send the reader to \cite[Section 5]{HKST15}.

Given $\G\subseteq \Fr(\X)$, a Borel function $\rho\colon \X\to \overline{\R}_{\geq 0}$ is said to be admissible for $\G$ if
\begin{equation}
    \int\limits_{\gamma}\rho \geq 1\qquad \text{for every $\gamma\in \G$}.
\end{equation}
With this in mind, given $\G\subseteq \Fr(\X)$, by the $\infty$-modulus of $\G$ with respect to $\mu$, one means the quantity
\begin{equation}
    \Mod^{\mu}_{\infty}(\G)\coloneqq \inf\limits_{\rho}\|\rho\|_{\Leb^{\infty}_{\mu}},
\end{equation}
where the infimum is taken over all Borel functions $\rho\colon \X\to \overline{\R}_{\geq 0}$ that are admissible for $\G$. It turns out that the $\infty$-modulus as a function defined on subfamilies of $\Fr(\X)$ is an outer measure. As an important consequence of the given definition, we also immediately notice that
\begin{equation}\label{eq:ModConcRFr}
    \Mod^{\mu}_{\infty}\big(\Fr(\X)\backslash \Fr_{\mathrm{r}}(\X)\big)=0.
\end{equation}
In other words, the $\infty$-modulus is concentrated on rectifiable fragments, which will always be taken into account.

\begin{rem}
    In light of the attribute given in \eqref{eq:ModConcRFr}, we wish to mention the following. This property implies that, in the formulations of \cref{theo:MainRes} and \cref{theo:MainResEucl}, it would change nothing if the main inequalities there were required to hold ``generically'' for rectifiable curves only rather than for all curves. Put differently, the use of the $\infty$-modulus reduces the consideration immediately to rectifiable curves (or, more generally, curve fragments) by its very nature. In this regard, it would perhaps be more ``honest'' to compare the central property of the maps within our results, i.e., the presence of controls as in \eqref{eq:GenBLD}, not with the path-isometry and BLD conditions but with the so-called weak path-isometry and BLD conditions, respectively. Namely, unlike for the former ones, for the latter ones corresponding inequalities on the lengths must hold for rectifiable curves only, which is known to be indeed a weaker assumption; for more on this, see \cite[Definition 4.1]{P11} specifically and the content therein in general.
\end{rem}

We will need further the following simple statement.
\begin{prop}\label{prop:SubsZerMod}
    Let $Q\in \mathrm{B}(\X)$ be such that $\mu(Q)=0$. Then one has
    \begin{equation}
    \Mod_{\infty}^{\mu}\big(\Gamma^+_Q(\X)\big)=0.
    \end{equation}
    \begin{proof}
        The proof can be found in \cite[Lemma 2.7]{BEBS24}.
    \end{proof}
\end{prop}

We prefer to explicitly record also the following useful fact.
\begin{prop}\label{prop:SubFrag}
    Let $\G_0,\G\subseteq \Fr(\X)$ be such that any fragment $\gamma\in \G_0$ contains a subfragment lying in $\G$. Then one has
    \begin{equation}
        \Mod^{\mu}_{\infty}(\G_0)\leq \Mod^{\mu}_{\infty}(\G).
    \end{equation}
    \begin{proof}
        One just needs to notice that any function admissible for $\G$ is also admissible for $\G_0$.
    \end{proof}
\end{prop}

We finalize this subsection with the definition of equi-Luzin--Lipschitz maps, which are exactly the maps appearing in our main results.
\begin{defi}\label{def:LuzLip}
    Given $Q\in \mathrm{B}(\X)$, $d\in \Nat_0$, $p\in \mathsf{P}^d_+$, and $L\in \R_{\geq 0}$, a map $\varphi\colon Q\to \mathbb{R}^d_p$ is said to be $L$-Luzin--Lipschitz if there is a Borel $\mu$-partition $(Q_m)_{m\in \mathbb{N}}$ of $Q$ such that, for each $m\in \mathbb{N}$, the map $\varphi|_{Q_m}$ is $L$-Lipschitz.
\end{defi}

\subsection{On Weaver derivations}\label{ss:WD} Here we give some important information on the notion of Weaver derivations necessary for our exposition.

\begin{rem}
    Over the last decades, there has been successfully developed a fruitful theory of “nonsmooth differential geometry” in the setting of metric measure spaces, within which one can speak, in particular, about specific ``vector fields'' (of different kinds) on such spaces, perfectly in line with the classical differential geometry. Moreover, the collection of such objects (subject to a chosen approach) forms naturally a certain normed module, which is typically referred to as a ``tangent module'' for the underlying space; remarkably, this allows one to operate rigorously, for instance, with ``infinitesimal tangent norms'', as announced in \cref{ss:FrR}. From the perspective of the very Sobolev calculus, this program was accomplished mostly by N. Gigli, the comprehensive presentation of which can be found in \cite{G18}. The foundation for all that, in turn, lay in earlier works by N. Weaver, for which we send the reader to \cite[Section 10]{W18} and references therein. One of the central objects within the corresponding analysis is what are known as Weaver derivations, which can be viewed as a nonsmooth generalization of derivations on smooth manifolds. And exactly these, which we find the most suitable to exploit (among other tools to define a ``tangent module'' in our framework), fall under the forthcoming consideration; as the main source on this matter though, we primarily rely on \cite{BEBS24}, as it is more closely related to our work.
\end{rem}

Let $\mu$ be a locally finite Borel measure on $\X$. By a Weaver $\mu$-derivation on $\X$, or just $\mu$-derivation in what follows, we will mean a linear map $\Dif\colon \LIP(\X)\to \Leb^{\infty}_{\mu}$ with the following properties:
\begin{enumerate}
\item[$\rm I)$] it is bounded, i.e., one has
\begin{equation}
    \sup\limits_{g\in \LIP_1(\X)} \big\|\Dif(g)\big\|_{\Leb^{\infty}_{\mu}}<+\infty;
\end{equation}
    \item[$\rm II)$] it satisfies the Leibniz rule, i.e., given $g_1,g_2\in \BLIP(\X)$, it holds that
    \begin{equation}
        \Dif(g_1g_2)=g_1\Dif(g_2)+g_2\Dif(g_1) \qquad \text{$\mu$-a.e. on $\X$};
    \end{equation}
    \item[$\rm III)$] it is $*$-weakly continuous, i.e., given $g\in \LIP_1(\X)$ and $(g_n)_{n\in \mathbb{N}}\subseteq \LIP_1(\X)$ such that the sequence $(g_n)_{n\in \mathbb{N}}$ converges pointwise on $\X$ to $g$, the sequence $\big(\Dif(g_n)\big)_{n\in \mathbb{N}}$ converges $*$-weakly in $\Leb^{\infty}_{\mu}$ to $\Dif(g)$.
\end{enumerate}
The collection of all $\mu$-derivations will be denoted by $\mathcal{X}_{\mu}$. It can be seen to possess a natural structure of an $\Leb_{\mu}^{\infty}$-module. Moreover, the assignment
\begin{equation}
   |\Dif|_{\mathcal{X}_{\mu}}\coloneqq \bigvee\limits_{g\in \LIP_1(\X)} \big|\Dif(g)\big|_{\Leb_{\mu}},\qquad \text{$\Dif\in \mathcal{X}_{\mu}$},
\end{equation}
results in an $\Leb_{\mu}$-norm $|\cdot|_{\mathcal{X}_{\mu}}$ on $\mathcal{X}_{\mu}$, as can be easily checked. With all this, we have at our disposal the $\Leb_{\mu}$-normed $\Leb_{\mu}^{\infty}$-module $\mathcal{X}_{\mu}=\big(\mathcal{X}_{\mu},|\cdot|_{\mathcal{X}_{\mu}}\big)$. Exactly to this module we will subsequently apply all the results established in the previous sections.

For the purposes of our work, the cornerstone result that we will exploit later (and which is essentially the only ``black box'' for the whole exposition) concerns exactly Weaver derivations. Loosely speaking, it states that the dimension of the corresponding module, which can itself depend quite strongly on the underlying measure, can be bounded from above by the Hausdorff dimension of the underlying metric space, which remarkably knows nothing about the measure involved. More accurately, this result reads as follows.
\begin{theo}\label{theo:FromHausToWeav}
    Let $\mu$ be a locally finite Borel measure on $\X$. Then, given any $Q\in \mathrm{B}(\X)$, the following estimate holds:
    \begin{equation}
        \dim_{\mathcal{X}_{\mu}}(Q)\leq \dim_{\mathrm{H}}(Q).
    \end{equation}
    \begin{proof}
        The provided statement stems basically from several deep results established within \cite{BEBS24} by Bate--Eriksson-Bique--Soultanis. For the convenience of the reader though, we prefer to give an outline of the proof.
        
        Fix $Q\in \mathrm{B}(\X)$ and put $d\coloneqq \dim_{\mathrm{H}}(Q)$ for brevity. In the case $d=+\infty$ there is nothing to prove, so further we may clearly assume that $d<+\infty$.
        
        To fit completely into the setting of \cite{BEBS24}, we find a nondecreasing family of open subsets $(O_l)_{l\in \mathbb{N}}$ of $\X$ that covers $\X$ with $\mu(O_l)<+\infty$ for each $l\in \mathbb{N}$ such that, for each $l\in \Nat$, the closure of $O_l$ in $\X$ forms, under the metric induced from $\X$, a complete metric space. This is clearly possible as the space $\X$ is separable, and hence Lindel{\"o}f, and locally complete and the measure $\mu$ is locally finite. Fix then $l\in \mathbb{N}$ and consider the measure $\mu_l\coloneqq \mu_{\rmes Q\cap O_l}$, which is now finite and Borel and whose support is moreover a separable, complete metric space. At a conceptual level at least, all this allows us to apply results from \cite{BEBS24} to the space $(\X,\mu_{l})$.
        
        It then follows from \cite[Theorem 5.5]{BEBS24} that the space $(\X,\mu_l)$ admits a so-called fragment-wise differentiable structure, with the corresponding charts being moreover of dimension not greater than $d$ in our case. Informally speaking, this means that the space $\X$ can be $\mu_l$-partitioned into a countable collection of special finite-dimensional charts, each of which consists of a Borel set in $\X$ and a system of real-valued Lipschitz functions on $\X$ with some desirable properties. This feature, in turn, is a relatively straightforward consequence of the profound result in \cite[Theorem 5.3]{BKO23} allowing one to relate exactly the Hausdorff dimension of the space (visible by the measure) and the number of ``independent directions'' in the space, which basically equals (after applying appropriate definitions) the dimension of charts thereon. Then, as explained in \cite[Section 5.4]{BEBS24}, the presence of a fragment-wise differentiable structure on the space gives rise to a certain finite-dimensional measurable bundle $\mathcal{T}_{\mu_l}$ over $(\X,\mu_l)$, which can be viewed as a ``tangent bundle'' of $(\X,\mu_l)$. Based on this, one can consider the $\Leb_{\mu_l}$-normed $\Leb^{\infty}_{\mu_l}$-module $\Gamma_{\infty}(\mathcal{T}_{\mu_l})$ of $\Leb^{\infty}_{\mu_l}$-sections of $\mathcal{T}_{\mu_l}$, as discussed again in \cite[Section 5.4]{BEBS24}. This module turns out to be finite-dimensional, and, in fact, its dimension does not exceed $d$. While all this is not explicitly demonstrated in \cite{BEBS24}, a completely analogous property is verified in \cite[Section 6.1]{EBS24}, within an earlier work by Eriksson-Bique--Soultanis on morally the same subject.

        One more pivotal result by Bate--Eriksson-Bique--Soultanis, namely \cite[Theorem 1.6]{BEBS24}, states that, under the existence of a fragment-wise differentiable structure on $(\X,\mu_l)$, the modules $\Gamma_{\infty}(\mathcal{T}_{\mu_l})$ and $\mathcal{X}_{\mu_l}$ are isomorphic as normed modules (in a natural sense), which, together with the dimensional bound for $\Gamma_{\infty}(\mathcal{T}_{\mu_l})$, yields that
        \begin{equation}
            \dim_{\mathcal{X}_{\mu_l}}(Q)=\dim_{\Gamma_{\infty}(\mathcal{T}_{\mu_l})}(Q)\leq d.
        \end{equation}
        
       As the last step, it remains only to exploit the locality of Weaver derivations, provided in \cite[Theorem 10.37]{W18}, which means basically, in our terms, that the modules $ \mathcal{X}_{\mu}|_{Q\cap O_l}$ and $\mathcal{X}_{\mu_l}$ can be naturally identified. As we clearly have, from the corresponding definitions, that
        \begin{equation}
            \dim_{\mathcal{X}_{\mu}}(Q)=\sup\limits_{l\in \mathbb{N}}\dim_{\mathcal{X}_{\mu}|_{Q\cap O_l}}(Q),
        \end{equation}
        we finally end up with the desired bound.
    \end{proof}
\end{theo}

In terms of Weaver derivations, we now introduce the following (rather nonstandard) notion, which is exactly the one that appears in \cref{theo:MainResEucl}.
\begin{defi}\label{def:InfEucl}
    Let $\mu$ be a locally finite Borel measure on $\X$. The space $(\X,\mu)$ will be called infinitesimally Euclidean if the module $\mathcal{X}_{\mu}$ is Euclidean over $\X$.
\end{defi}

\begin{rem}
    The proposed notion (as mentioned already in \cref{ss:FrR}) should be compared to, but not confused with, that of infinitesimal Hilbertianity. The latter one represents the idea that a given metric measure space will turn out to possess Euclidean ``infinitesimal tangent norms'' at its points once we require the Sobolev $\mathrm{W}^{1,2}$-space thereon, which is a priori just a Banach space, to be a Hilbert space. Practically, that condition requires a certain normed module, which is fair to call a ``cotangent module'', over the given space to be Euclidean, which basically coincides with what we do, with the only difference lying in the fact that, for our purposes, another module is more natural to consider. It would be interesting to relate these two notions, yet we do not include this question within the scope of our work, as it is not really central thereto.
\end{rem}

Before proceeding to the next section, we recall the notion of exterior derivative from the context of Weaver derivations. A detailed exposition on this can be found in \cite[Section 10.3]{W18}. 

Let $\mu$ be a locally finite Borel measure on $\X$. By the exterior derivative on $(\X,\mu)$, we will mean the map 
\begin{equation}\label{eq:ExtDer1}
    \D_{\mu}\colon \LIP(\X)\to \mathcal{X}_{\mu}^{\star}\colon g\mapsto \D_{\mu}g
\end{equation}
given by the assignment
\begin{equation}\label{eq:ExtDer2}
    (\D_{\mu} g)(\mathcal{D})\coloneqq \mathcal{D}(g),\qquad \text{$g\in \LIP(\X)$}.
\end{equation}
It is then quite easy to see that the map $\D_{\mu}$ is well posed and linear, as well as that
\begin{equation}
    \D_{\mu} \big(\LIP_1(\X)\big)\subseteq [\mathcal{X}_{\mu}^{\star}]_{\leq 1}.
\end{equation}
As a direct, yet important, consequence of that, we note the following: given $\bm{f}\in \Mat^{\infty}\big(\LIP_1(\X)\big)$, we have
\begin{equation}
    \mathsf{p}_{\infty}\big(\D_{\mu}(\bm{f})\big)\leq 1\qquad \text{$\mu$-a.e. on $\X$}. 
\end{equation}

\begin{rem}
    We warn the reader about the following notational aspect. In the current manuscript, we actively refer to \cite{BEBS24}, where (as already discussed in the course of the proof of \cref{theo:FromHausToWeav}) the authors establish (under suitable conditions) the existence of a certain fragment-wise differentiability structure on a given space. Along with this, there emerges the corresponding fragment-wise differential that also acts on Lipschitz functions on the space. As one of the main results in the provided work, namely \cite[Theorem 1.6]{BEBS24}, demonstrates, the latter differential (under the same conditions as necessary for its existence) can be naturally identified with the exterior derivative in the sense we introduced. We, however, will never use either this correspondence or the fragment-wise differential itself. Put differently, we will deal in what follows only with the exterior derivative in the above sense.
\end{rem}

\subsection{On geometric distortion factors}\label{ss:GDF} Now we turn our attention to one more key notion of the current work, namely that of geometric distortion factors.

Throughout this subsection, a locally finite Borel measure $\m$ on $\X$ is assumed fixed.

\begin{rem}
    The concept that we are about to present is useful to compare with several other ones appearing in the modern mathematical literature, notably in slightly different contexts. First, we mention the notions of (minimal) weak upper gradient (comprehensively presented, for instance, in \cite[Section 6]{HKST15}) and (minimal) $\ast$-upper gradient (covered, for instance, in \cite[Section 2.4]{BEBS24}), which correspond to the ``upper'' part in the forthcoming definition. And second, we add here the notion of (maximal) weak subslope from the nonsmooth Lorentzian geometry proposed in \cite[Section 3.3]{BBCGMORS24}, which, in turn, is related to the ``lower'' part in the forthcoming definition. The main common feature of these objects, including the ones introduced below, lies in the fact that they measure the extreme rate of change of the speed of curves (or curve fragments) under a given map.
\end{rem}

The relevant definition is as follows. We draw the reader's attention to direct parallels of it with \cref{def:DistCoef}.
\begin{defi}\label{def:GeoDisFac}
    Let $p\in \mathsf{P}^{\infty}$, $\bm{f}\in \Mat^{\infty}\big(\LIP(\X)\big)$. By the upper geometric distortion factor of $\bm{f}$ with respect to $p$, we will mean the element of $[\overline{\Leb}_{\m}]_{\geq 0}$ obtained by
    \begin{equation}\label{eq:GeoDisFac0}
        \GDis^{\mathrm{u}}_{p}[\bm{f}]\coloneqq \bigwedge\limits_{\rho}\rho,
    \end{equation}
    where the meet is taken over all Borel functions $\rho\colon \X\to \overline{\R}_{\geq 0}$ such that it holds, for $\Mod^{\m}_{\infty}$-a.e. $\gamma\in \Fr_{\mathrm{L}}(\X)$, that
    \begin{equation}\label{eq:GeoDisFac00}
        \ms^p_{\bm{f}\circ \gamma}(t)\leq \rho(\gamma_t)\ms_{\gamma}(t) \qquad \text{for $\mathcal{L}^1$-a.e. $t\in \dom(\gamma)$}.
    \end{equation}
    By the lower geometric distortion factor of $\bm{f}$ with respect to $p$, we will mean the element of $[\overline{\Leb}_{\m}]_{\geq 0}$ obtained by
    \begin{equation}\label{eq:GeoDisFac000}
        \GDis^{\mathrm{l}}_{p}[\bm{f}]\coloneqq \bigvee\limits_{\rho}\rho,
    \end{equation}
    where the join is taken over all Borel functions $\rho\colon \X\to\overline{\R}_{\geq 0}$ such that it holds, for $\Mod^{\m}_{\infty}$-a.e. $\gamma\in \Fr_{\mathrm{L}}(\X)$, that
    \begin{equation}\label{eq:GeoDisFac0000}
        \ms^{p}_{\bm{f}\circ \gamma}(t)\geq \rho(\gamma_t)\ms_{\gamma}(t) \qquad \text{for $\mathcal{L}^1$-a.e. $t\in \dom(\gamma)$}.
    \end{equation}
\end{defi}

\begin{rem}
    We point out the following technical difference between \cref{def:DistCoef} and \cref{def:GeoDisFac}: in the former case, we take the join and meet over equivalence classes up to the $\m$-a.e. equality of functions $\X\to \overline{\R}_{\geq 0}$, while in the latter case, the join and meet are taken over genuine Borel functions $\X\to \overline{\R}_{\geq 0}$. The reason for that lies in the standard subtlety of working with (in fact, any kind of) modulus that (unlike in the closely related case of plans) the composition of an $\m$-a.e. defined function on $\X$ with a curve (or a curve fragment) is not generally well defined.
\end{rem}

Analogously to the situation with \cref{def:DistCoef}, given $p\in \mathsf{P}^{\infty}$ and $\bm{f}\in \Mat^{\infty}\big(\LIP(\X)\big)$, one can readily verify, via \cref{prop:SubsZerMod}, the following: letting $\rho^{\mathrm{u}}$ and $\rho^{\mathrm{l}}$ be arbitrary nonnegative Borel $\m$-representatives, respectively, of $\GDis^{\mathrm{u}}_p[\bm{f}]$ and of $\GDis^{\mathrm{l}}_p[\bm{f}]$, one has, for $\Mod_{\infty}^{\m}$-a.e. $\gamma\in \Fr_{\mathrm{L}}(\X)$, that
\begin{equation}\label{eq:GeoDisFac00000}
    \ms_{\bm{f}\circ \gamma}^p(t)\leq \rho^{\mathrm{u}}(\gamma_t)\ms_{\gamma}(t) \qquad \text{for $\mathcal{L}^1$-a.e. $t\in \dom(\gamma)$}
\end{equation}
and that
\begin{equation}\label{eq:GeoDisFac000000}
    \ms_{\bm{f}\circ \gamma}^{p}(t)\geq \rho^{\mathrm{l}}(\gamma_t)\ms_{\gamma}(t)\qquad \text{for $\mathcal{L}^1$-a.e. $t\in \dom(\gamma)$}.
\end{equation}
Also, we notice that both functions given by \eqref{eq:GeoDisFac0} and \eqref{eq:GeoDisFac000} are homogeneous with respect to the multiplication by scalars from $\R_{>0}$. We again exclude zero to prevent some obvious infinity issues.

The definition of geometric distortion factors is given in terms of modulus, rather than in terms of fragment plans, as it is slightly more natural and gives finer information about the family of all ``bad'' curves for a given Lipschitz system, i.e., those that violate the conditions in \eqref{eq:GeoDisFac00000} and \eqref{eq:GeoDisFac000000}. It is much more convenient though to deal with fragment plans, as these are genuinely sigma-additive on Borel subsets of the space, so it is those that will be exploited in the next subsection for obtaining necessary estimates on geometric distortion factors. For this reason, we provide the lemma below, which shows basically that the two approaches lead to the same object. For more on connections between plans and moduli, we send the reader to \cite{EKMM21}.
\begin{lem}\label{lem:ModToPlan}
    Let $Q\in \mathrm{B}_{\m}^+(\X)$, $p\in \mathsf{P}^{\infty}$, $\bm{f}\in \Mat^{\infty}\big(\LIP(\X)\big)$, $\zeta\in \R_{> 0}$. Then the assertions below are valid.
    \begin{enumerate}
    \item[$\rm I)$] Suppose that
    \begin{equation}\label{eq:ModToPlan0}
        \GDis^{\mathrm{u}}_p[\bm{f}]> \zeta\qquad \text{$\m$-a.e. on $Q$}.
    \end{equation}
    Then there exists a nonzero fragment plan $\mathbb{P}$ on $(\X,\m)$ such that, for $\mathbb{P}$-a.e. $\gamma\in \Fr_{\mathrm{bL}}(\X)$, one has
    \begin{equation}\label{eq:ModToPlan00}
       \mathcal{L}^1\bigg(\Big\{t\in \gamma^{-1}(Q)\bigm| \ms^p_{\bm{f}\circ \gamma}(t)> \zeta\ms_{\gamma}(t)\Big\}\bigg)>0.
    \end{equation}
      \item[$\rm II)$] Suppose that
    \begin{equation}\label{eq:ModToPlan000}
        \GDis^{\mathrm{l}}_{p}[\bm{f}]< \zeta\qquad \text{$\m$-a.e. on $Q$}.
    \end{equation}
    Then there exists a nonzero fragment plan $\mathbb{P}$ on $(\X,\m)$ such that, for $\mathbb{P}$-a.e. $\gamma\in \Fr_{\mathrm{bL}}(\X)$, one has
    \begin{equation}\label{eq:ModToPlan0000}
       \mathcal{L}^1\bigg(\Big\{t\in \gamma^{-1}(Q)\bigm| \ms^{p}_{\bm{f}\circ \gamma}(t)< \zeta\ms_{\gamma}(t)\Big\}\bigg)>0.
    \end{equation}
    \end{enumerate}
\end{lem}

For the proof of the statement above, we shall exploit the following result, referred to as Rainwater's lemma in \cite[Section 4.1]{BEBS24}.
\begin{lem}\label{lem:RainWater}
    Let $\mu$ be a finite Borel measure on $\X$, let $\G\subseteq \Fr_{\mathrm{bL}}(\X)$ be a compact family with $\Mod_{\infty}^{\mu}(\G)>0$ such that, for every $\gamma\in \G$, one has $\mathcal{L}^1\big(\dom(\gamma)\big)>0$. Then there exists a compact set $K\in \mathrm{B}_{\mu}^+(\X)$ and a fragment plan $\mathbb{P}$ on $(\X,\mu)$ with $\mathbb{P}^{\#}(K)>0$ such that $\mathbb{P}\big(\Fr_{\mathrm{bL}}(\X)\backslash \G|_K\big)=0$.
    \begin{proof}
        While the proof is provided in \cite[Lemma 4.2]{BEBS24}, let us give some necessary brief explanation. In the presented lemma, there are two alternatives. The first one, under our assumptions, is impossible, as can be verified via \cite[Lemma 2.7]{BEBS24}, so the second one must take place. The corresponding statement, as it is in \cite[Lemma 4.2]{BEBS24}, is weaker than what we put in our formulation. Nevertheless, the very beginning of the proof therein actually demonstrates the validity of our conclusion as well.
    \end{proof}
\end{lem}

We now return to \cref{lem:ModToPlan}. The driving strategy of the corresponding reasoning below essentially mimics that of \cite[Proposition 4.3]{BEBS24}, where the authors verify a statement of a quite similar flavor.
\begin{proof}[Proof of \cref{lem:ModToPlan}]
For the sake of brevity, we prove only the second part, as the consideration of the first part is mostly the same.

    To begin with, let us find a nondecreasing family $(C_l)_{l\in \Nat}\subseteq \mathrm{B}_{\m}^+(Q)$ of closed subsets of $\X$ such that
    \begin{equation}
        \m\bigg(Q\Big\backslash \bigcup\limits_{l\in \Nat} C_l\bigg)=0.
    \end{equation}
    Then, we find a nondecreasing family $(O_l)_{l\in \mathbb{N}}$ of open subsets of $\X$ such that $\m(O_l)<+\infty$ for each $l\in \mathbb{N}$ and
    \begin{equation}
        \bigcup\limits_{l\in \mathbb{N}} O_l=\X.
    \end{equation}
    Both of these steps are clearly possible due to the assumptions on $\X$ and $\m$. With this, for each $l\in \Nat$ we put
    \begin{equation}
        \G_{l}\coloneqq \Bigg\{\gamma\in \Fr_{\mathrm{L}}(O_l) \biggm|\mathcal{L}^1\bigg(\Big\{t\in \gamma^{-1}(C_l)\bigm| \ms^{p}_{\bm{f}\circ \gamma}(t)< \zeta\ms_{\gamma}(t)\Big\}\bigg)>0\Bigg\}.
    \end{equation}
Now, taking \cref{def:GeoDisFac} with \cref{prop:SubsZerMod} and \cref{prop:SubFrag} into account and applying the subadditivity of $\Mod^{\m}_{\infty}$, we can derive from \eqref{eq:ModToPlan000} the existence of $l_0\in \Nat$ such that $\Mod^{\m}_{\infty}(\G_{l_0})>0$. Let us put $\G\coloneqq \G_{l_0}$ for brevity and consider the Borel measure $\mu\coloneqq \m_{\rmes O_{l_0}}$, which can be seen to be finite, and observe easily that, as $\G\subseteq \Fr(O_{l_0})$, we have
    \begin{equation}\label{eq:ModToPlan2}
        \Mod^{\mu}_{\infty}(\G)=\Mod_{\infty}^{\m}(\G)>0.
    \end{equation}
    This reasoning was necessary to reduce everything to the case of a finite underlying measure.

    Next, for each $n\in \mathbb{N}$, we let $\G(n)$ denote the family of all $\gamma\in \Fr(\X)$ such that
    \begin{equation}\label{eq:ModToPlan3}
        \mathcal{L}^1\big(\dom(\gamma)\big)\geq \frac{1}{n},
    \end{equation}
    which turns out to be closed as a subset of $\Fr(\X)$. After that, as the Borel measure $\mu$ is finite on a Polish space, we can find a nondecreasing family $(K_m)_{m\in \mathbb{N}}$ of compact subsets of $O_{l_0}$ such that, with
    \begin{equation}
        Q_0\coloneqq \bigg(\X\Big\backslash \bigcup\limits_{m\in \mathbb{N}} K_m\bigg),
    \end{equation}
    we have $\mu(Q_0)=0$. Finally, for all $n,m\in \mathbb{N}$, we let $\G^n_m$ be the family of all $\gamma\in \G(n)\cap \Fr_{\mathrm{bL}}^n(C_{l_0}\cap K_m)$ such that
    \begin{gather}
      p\big(\bm{f}(\gamma_{t'})-\bm{f}(\gamma_t)\big)\leq \frac{n\zeta}{n+1}\dX(\gamma_t,\gamma_{t'})\qquad \text{for all $t,t'\in \dom(\gamma)$},\label{eq:ModToPlan4}\\
      \dom(\gamma)\subseteq [-n,n]\label{eq:ModToPlan5}.
    \end{gather}
    Reasoning as for the Arzel{\`a}--Ascoli theorem, one can verify, using the relevant definitions, that the families $\G^n_m$, $n,m\in \mathbb{N}$, are compact as subsets of $\Fr(\X)$.

    We now fix $\gamma\in \big(\G\backslash \Gamma^+_{Q_0}(\X)\big)$ and claim the following: there are $n,m\in \mathbb{N}$ such that the fragment $\gamma$ contains a subfragment lying in $\G^n_m$.

    First, as $\gamma\in \G$, we can clearly find $\epsilon\in (0,1)$ and a compact set $T\subseteq \gamma^{-1}(C_{l_0})$ with $\mathcal{L}^1(T)>0$ such that
    \begin{equation}\label{eq:ModToPlan7}
        \ms^{p}_{\bm{f}\circ \gamma}(t)< \epsilon\zeta\ms_{\gamma}(t)\qquad \text{for any $t\in T$}.
    \end{equation}
    In turn, since $\gamma\in \Fr_{\mathrm{L}}(\X)$ and $\gamma\not\in \Gamma_{Q_0}^+(\X)$, we know that
    \begin{equation}
        \int\limits_{\gamma} \1_{Q_0}=0,
    \end{equation}
    whence, by \cref{prop:ArForm}, we easily get that
    \begin{equation}
        \ms_{\gamma}(t)=0\qquad \text{for $\mathcal{L}^1$-a.e. $t\in \gamma^{-1}(Q_0)$},
    \end{equation}
    which, together with \eqref{eq:ModToPlan7}, implies that there is $m\in \mathbb{N}$ such that, with $T'\coloneqq T\cap \gamma^{-1}(K_m)$, we have $\mathcal{L}^1(T')>0$ and
    \begin{equation}
        \ms_{\gamma}(t)>0\qquad \text{for $\mathcal{L}^1$-a.e. $t\in T'$}.
    \end{equation}
    For the sake of simplicity, we restrict further to a compact set $T''\subseteq T'$ without isolated points such that one has $\mathcal{L}^1(T'')>0$ and the condition above, with $T'$ replaced by $T''$, still also holds. As the next step, we recall that
    \begin{gather}
        \ms_{\gamma}(t)=\lim\limits_{\substack{t'\to t \\ t'\in \dom(\gamma)}}\frac{\dX(\gamma_t,\gamma_{t'})}{|t'-t|},\qquad \text{for $\mathcal{L}^1$-a.e. $t\in T''$},\\
        \ms^{p}_{\bm{f}\circ \gamma}(t)=\lim\limits_{\substack{t'\to t \\ t'\in \dom(\gamma)}}\frac{p\big(\bm{f}(\gamma_{t'})-\bm{f}(\gamma_{t})\big)}{|t'-t|},\qquad \text{for $\mathcal{L}^1$-a.e. $t\in T''$}.
    \end{gather}
    As the maps $\gamma$ and $\bm{f}\circ \gamma$ are continuous, with the help of the Egorov theorem and after some standard reasoning, we can get the existence of a compact set $T'''\subseteq T''$ with $\mathcal{L}^1(T''')>0$ such that
    \begin{gather}
    \inf\limits_{t\in T'''}\ms_{\gamma}(t)>0,\\
        \lims\limits_{\tau \to 0}\sup\limits_{\substack{t,t'\in T'''\\ 0<|t'-t|\leq \tau}}\bigg|\frac{\dX(\gamma_t,\gamma_{t'})}{|t'-t|}-\ms_{\gamma}(t)\bigg|=0,\\
        \lims\limits_{\tau \to 0}\sup\limits_{\substack{t,t'\in T'''\\ 0<|t'-t|\leq \tau}}\Bigg|\frac{p\big(\bm{f}(\gamma_{t'})-\bm{f}(\gamma_{t})\big)}{|t'-t|}-\ms^{p}_{\bm{f}\circ \gamma}(t)\Bigg|=0.
    \end{gather}
With this in mind, it is possible to find $\tau_0\in \R_{>0}$, $\epsilon_0\in (\epsilon,1)$, $L_0\in \R_{>0}$ such that
    \begin{equation}\label{eq:ModToPlan6}
    \begin{gathered}
       \text{$p\big(\bm{f}(\gamma_{t'})-\bm{f}(\gamma_{t})\big)\leq \epsilon_0\zeta\dX(\gamma_t,\gamma_{t'})\quad $ and $\quad  \dX(\gamma_t,\gamma_{t'})\geq \frac{1}{L_0}|t'-t| $}\\
       \text{for any $t,t'\in T'''$ with $|t'-t|\leq \tau_0$}.
        \end{gathered}
    \end{equation}
  Finally, we find $t_0\in T'''$ such that, with
    \begin{equation}
        T_0\coloneqq T'''\cap [t_0,t_0+\tau_0],
    \end{equation}
    we have $\mathcal{L}^1(T_0)>0$, which is obviously possible. Choosing $n\in \mathbb{N}$ appropriately, in correspondence with \eqref{eq:ModToPlan3}, \eqref{eq:ModToPlan4}, and \eqref{eq:ModToPlan6}, we can guarantee that $\gamma|_{T_0}\in \G^n_m$, as desired.

    The stated claim is thus verified, whence, by \cref{prop:SubsZerMod} and \cref{prop:SubFrag}, we get, again using the subadditivity of $\Mod_{\infty}^{\mu}$, that
    \begin{equation}
    \begin{gathered}
        \Mod^{\mu}_{\infty}(\G)=\Mod^{\mu}_{\infty}\big(\G\backslash \Gamma^+_{Q_0}(\X)\big)\leq\\
        \leq \Mod^{\mu}_{\infty}\Bigg(\bigcup\limits_{n,m\in \mathbb{N}} \G^n_m\Bigg)\leq \sum\limits_{n,m\in \mathbb{N}}\Mod^{\mu}_{\infty}(\G^n_m),
        \end{gathered}
    \end{equation}
    from which, together with \eqref{eq:ModToPlan2}, there follows the existence of $m_0,n_0\in \mathbb{N}$ such that
    \begin{equation}
        \Mod^{\mu}_{\infty}(\G^{n_0}_{m_0})>0.
    \end{equation}
    As the family $\G^{n_0}_{m_0}$ is compact and contains only ``visible'' biLipschitz fragments, we are able to apply \cref{lem:RainWater} to it, which gives us the existence of a compact set $K\in \mathrm{B}_{\mu}^+(\X)$ and a fragment plan $\mathbb{P}_0$ on $(\X,\mu)$ with $\mathbb{P}_0^{\#}(K)>0$ such that $\mathbb{P}_0\big(\Fr_{\mathrm{bL}}(\X)\backslash \G^{n_0}_{m_0}|_K\big)=0$. We then observe that, by the very definition of fragment plans, the measure $\mathbb{P}_0$ is also a fragment plan on $(\X,\m)$. 
    
    As the finalizing step, we claim the following: there is $n\in \mathbb{N}$ such that
    \begin{equation}
        \mathbb{P}_0\big(\G(n)\big)>0.
    \end{equation}
    Assuming the contrary, we immediately get that
    \begin{equation}
        \mathcal{L}^1\big(\dom(\gamma)\big)=0\qquad \text{for $\mathbb{P}_0$-a.e. $\gamma\in \Fr_{\mathrm{bL}}(\X)$}.
    \end{equation}
    At the same time, from the condition above, we easily obtain, with the use of \eqref{eqref:ArForm0}, \eqref{eq:LinIntInjFr}, and \eqref{eq:BarCMeas}, that
    \begin{equation}
        \mathbb{P}_0^{\#}(\X)=\int\limits_{\Fr_{\mathrm{bL}}(\X)} \mathcal{H}^1_{\X}\big(\im(\gamma)\big)\D \mathbb{P}_0(\gamma)=0,
    \end{equation}
    which contradicts the fact that $\mathbb{P}_0^{\#}(\X)>0$. So, a number with the required property indeed exists; let $\widehat{n}\in \mathbb{N}$ be one such number.

    To conclude the proof, it remains only to put $\mathbb{P}\coloneqq {\mathbb{P}_0}_{\rmes \G(\widehat{n})}$, which yields the desired nonzero fragment plan on $(\X,\m)$, and witness, on the basis of what we obtained previously, that one has $\mathbb{P}\Big(\Fr_{\mathrm{bL}}(\X)\big\backslash \big(\G_{m_0}^{n_0}|_K \cap \G(\widehat{n})\big)\Big)=0$ and $\big(\G_{m_0}^{n_0}|_K \cap \G(\widehat{n})\big)\subseteq \G_{l_0}$, that is, the measure $\mathbb{P}$ is nonzero and is concentrated on $\G_{l_0}$, which can be seen to be exactly what we wanted. The proof is thus fully complete.
    \end{proof}

\subsection{Relations between distortion factors} Here we aim at obtaining a cornerstone connection between the two introduced kinds of distortion factors.

Until the end of the subsection, we fix a locally finite Borel measure $\m$ on $\X$.

Speaking loosely, what we established in \cref{ss:DefMEst} is that any abstract finite-dimensional normed module admits a system in its dual with some uniform controls of a specific kind. Our eventual goal, from the perspective of \cref{theo:MainRes} and \cref{theo:MainResEucl}, is to transfer this result appropriately to the setting of metric measure spaces and Lipschitz functions thereon. As the central ingredient for that, we present the following intermediate lemma, which seems to be of independent importance.
\begin{lem}\label{lem:GeoDisAnDis}
    Let $d\in \Nat_0$, $p_{\mathrm{u}},p_{\mathrm{l}}\in \mathsf{P}^d_+$, $\bm{f}\in \Mat^d\big(\LIP(\X)\big)$. Then the following properties hold:
    \begin{gather}
        \GDis^{\mathrm{u}}_{p_{\mathrm{u}}}[\bm{f}]\leq \ADis^{\mathrm{u}}_{p_{\mathrm{u}}}\big[\D_{\m} \bm{f}\big]\qquad \text{$\m$-a.e. on $\X$};\label{eq:GeoDisAnDis0}\\
        \GDis^{\mathrm{l}}_{p_{\mathrm{l}}}[\bm{f}]\geq \ADis^{\mathrm{l}}_{p_{\mathrm{l}}}\big[\D_{\m} \bm{f}\big]\qquad \text{$\m$-a.e. on $\X$}.\label{eq:GeoDisAnDis00}
    \end{gather}
\end{lem}
\noindent It is worth noting that the above statement holds regardless of any dimensional bounds on the underlying space.

\begin{rem}
    Informally, the properties in \eqref{eq:GeoDisAnDis0} and \eqref{eq:GeoDisAnDis00} hold because (as will be elaborated below) any fragment plan gives rise to a certain Weaver derivation. Thus, one would expect that if it was true that any such derivation induced a fragment plan with some satisfactory properties, then the opposite estimate in \eqref{eq:GeoDisAnDis0} would also take place. At this stage, however, we do not know whether this is the case. As it does not affect our exposition, we omit any further discussion on this matter.
\end{rem}

In order to prove \cref{lem:GeoDisAnDis}, we need to verify several facts first, which, while being auxiliary, contain all the essence lying behind the corresponding statement.

Below we provide a lemma allowing one to construct explicitly a special Weaver derivation from a given fragment plan.
\begin{lem}\label{lem:ConstrDeriv}
    Let $\mathbb{P}$ be a fragment plan on $(\X,\m)$, let $\chi\colon \overline{\Fr}_{\mathrm{bL}}(\X)\to \R$ be a bounded, Borel function. Then there exists a unique derivation $\Dif\in \mathcal{X}_{\m}$ such that it holds, for any $g\in \LIP(\X)$, that
    \begin{equation}\label{eq:ConstrDeriv0}
        \big(\Dif(g)\big)(x)=\int\limits_{\overline{\Fr}_{\mathrm{bL}}(\X)}\chi(\gamma,t) \frac{\big(\partial(g\circ \gamma)\big)(t)}{\ms_{\gamma}(t)}\D \overline{\mathbb{P}}_x(\gamma,t)\qquad \text{for $\m$-a.e. $x\in \X$}.
    \end{equation}
  \begin{proof}
      The proof is given in \cite[Lemma 7.1]{BEBS24}.
  \end{proof}
\end{lem}
\noindent Let us note here that fractions as in the integrand in \eqref{eq:ConstrDeriv0}, while not being well defined everywhere, still make sense almost everywhere, with respect to both the measure on the underlying space and the disintegration measures on the collection of fragment graphs coming from fragment plans, the fact of which will be tacitly used in what follows.

In the next proposition, we demonstrate how to build, on the basis of a given Lipschitz function, a Weaver derivation of a ``sufficiently small'' norm that respects a certain ``visible'' (by some fragment plan) relation between the derivatives of a function along curves and the metric speeds of these curves.
\begin{prop}\label{prop:CurvToDerivU}
    Let $Q\in \mathrm{B}^+_{\m}(\X)$, $f\in \LIP(\X)$, $\zeta\in \R_{>0}$. Consider the family
    \begin{equation}\label{eq:CurvToDerivU0}
       \mathcal{G}_Q\coloneqq \Big\{(\gamma,t)\in \ev^{-1}(Q)\cap \overline{\Fr}_{\mathrm{bL}}(\X)\bigm| \big(\partial(f\circ \gamma)\big)(t)\geq  \zeta \ms_{\gamma}(t)\Big\}.
    \end{equation}
    Let $\mathbb{P}$ be a fragment plan on $(\X,\m)$ with $\overline{\mathbb{P}}(\mathcal{G}_Q)>0$, let $\varepsilon\in (0,1)$. Then there exists $\Dif_{\varepsilon}\in \mathcal{X}_{\m}$ such that, for some $Q_{\varepsilon}\in \mathrm{B}^+_{\m}(Q)$, the following properties hold:
    \begin{gather}
        |\Dif_{\e}|_{\mathcal{X}_{\m}}\leq 1+\varepsilon\qquad \text{$\m$-a.e. on $Q_{\varepsilon}$};\label{eq:CurvToDerivU00}\\
       \big|\Dif_{\e}(f)\big|_{\Leb_{\m}}\geq \zeta \qquad \text{$\m$-a.e. on $Q_{\varepsilon}$}.\label{eq:CurvToDerivU000}
    \end{gather}
    \begin{proof}
        Let us notice that the family $\mathcal{G}_Q$ can indeed be checked to be Borel, which is necessary for the correctness of our reasoning.
        
        Since
    \begin{equation}
        \int\limits_{\X} \overline{\mathbb{P}}_x(\mathcal{G}_Q)\D \mathbb{P}^{\#}(x)=\overline{\mathbb{P}}(\mathcal{G}_Q)>0
    \end{equation}
    and since
    \begin{equation}
        \overline{\mathbb{P}}_x(\mathcal{G}_Q)=0\qquad \text{for $\m$-a.e. $x\in \big(\X\backslash Q\big)$},
    \end{equation}
    we can clearly find $Q_{\varepsilon}\in \mathrm{B}^+_{\m}(Q)$ such that it holds, for some $c\in \R_{>0}$, that
    \begin{equation}
        c\leq \overline{\mathbb{P}}_x(\mathcal{G}_Q)\leq (1+\varepsilon)c\qquad \text{for $\m$-a.e. $x\in Q_{\varepsilon}$}.
    \end{equation}
    Let then $\mathcal{D}_{\varepsilon}$ be the derivation as in \cref{lem:ConstrDeriv} built with respect to $\mathbb{P}$ and $\tfrac{1}{c}\1_{\mathcal{G}_Q}$. With this, exploiting \eqref{eq:ConstrDeriv0}, we can write, for any $g\in \LIP_1(\X)$ and for $\m$-a.e. $x\in Q_{\varepsilon}$, that
    \begin{equation}
        \big|\Dif_{\varepsilon}(g)\big|_{\Leb_{\m}}(x)\leq \frac{1}{c}\int\limits_{\mathcal{G}_Q} \D \overline{\mathbb{P}}_x(\gamma,t)\leq 1+\varepsilon,
    \end{equation}
    the fact of which immediately implies \eqref{eq:CurvToDerivU00}. At the same time, using \eqref{eq:ConstrDeriv0} again, in combination with \eqref{eq:CurvToDerivU0}, we can write, for $\m$-a.e. $x\in Q_{\varepsilon}$, that
    \begin{equation}
        \big|\Dif_{\varepsilon}(f)\big|_{\Leb_{\m}}(x)\geq \frac{1}{c}\int\limits_{\mathcal{G}_Q} \frac{\big(\partial(f\circ \gamma)\big)(t)}{\ms_{\gamma}(t)}\D \overline{\mathbb{P}}_x(\gamma,t)\geq \zeta,
    \end{equation}
    which is exactly \eqref{eq:CurvToDerivU000}.

    The proof is complete.
    \end{proof}
\end{prop}

The proposition below is basically an analog of the previous one, where, in turn, we build, on the basis of a given system of Lipschitz functions, a suitable Weaver derivation of a ``sufficiently large'' norm that respects a somewhat reversed relation.
\begin{prop}\label{prop:CurvToDerivL}
    Let $Q\in \mathrm{B}^+_{\m}(\X)$, $p\in \mathsf{P}^{\infty}$, $\bm{f}\in \Mat^{\infty}\big(\LIP(\X)\big)$, $\zeta\in \R_{> 0}$. Consider the family
    \begin{equation}\label{eq:CurvToDerivL0}
       \mathcal{G}_Q\coloneqq \Big\{(\gamma,t)\in \ev^{-1}(Q)\cap \overline{\Fr}_{\mathrm{bL}}(\X)\bigm| \ms^{p}_{\bm{f}\circ \gamma}(t)\leq \zeta \ms_{\gamma}(t)\Big\}.
    \end{equation}
    Let $\mathbb{P}$ be a fragment plan on $(\X,\m)$ with $\overline{\mathbb{P}}(\mathcal{G}_Q)>0$, let $\varepsilon\in (0,1)$. Then there exists $\Dif_{\varepsilon}\in \mathcal{X}_{\m}$ such that, for some $Q_{\varepsilon}\in \mathrm{B}^+_{\m}(Q)$, the following properties hold:
    \begin{gather}
        |\Dif_{\e}|_{\mathcal{X}_{\m}}\geq 1-\varepsilon\qquad \text{$\m$-a.e. on $Q_{\varepsilon}$};\label{eq:CurvToDerivL00}\\
        p\big(\Dif_{\e}(\bm{f})\big)\leq \zeta \qquad \text{$\m$-a.e. on $Q_{\varepsilon}$}\label{eq:CurvToDerivL000}.
    \end{gather}
    \begin{proof}
    Let us notice that the family $\mathcal{G}_Q$ can indeed be checked to be Borel, which is necessary for the correctness of our reasoning.
    
    Let $(h_l)_{l\in \mathbb{N}}\subseteq \LIP_1(\X)$ be such that, for every $\gamma\in \Fr_{\mathrm{L}}(\X)$, we have that
    \begin{equation}\label{eq:CurvToDeriv1}
        \sup\limits_{l\in \mathbb{N}}\big(\partial(h_l\circ \gamma)\big)(t)=\ms_{\gamma}(t)\qquad \text{for $\mathcal{L}^1$-a.e. $t\in \dom(\gamma)$}.
    \end{equation}
    Fix $\varepsilon\in (0,1)$ and, for each $l\in \mathbb{N}$, consider the family
    \begin{equation}\label{eq:CurvToDeriv2}
        \mathcal{G}_{\varepsilon,l}\coloneqq \Big\{(\gamma,t)\in \mathcal{G}_Q\bigm| \big(\partial(h_l\circ \gamma)\big)(t)\geq (1-\varepsilon)^{\frac{1}{2}} \ms_{\gamma}(t)\Big\},
    \end{equation}
    which can also be seen to be Borel. With the use of the property in \eqref{eq:ExtPlan}, it is then possible to derive from \eqref{eq:CurvToDeriv1} and \eqref{eq:CurvToDeriv2} that
    \begin{equation}
        \overline{\mathbb{P}}\bigg(\mathcal{G}_Q\Big\backslash \bigcup\limits_{l\in \mathbb{N}} \mathcal{G}_{\varepsilon,l}\bigg)=0,
    \end{equation}
    whence we easily get the existence of $l_0\in \mathbb{N}$ with $\overline{\mathbb{P}}(\mathcal{G}_{\varepsilon,l_0})>0$. For brevity we put $\mathcal{G}\coloneqq \mathcal{G}_{\varepsilon,l_0}$.

    Since
    \begin{equation}
        \int\limits_{\X} \overline{\mathbb{P}}_x(\mathcal{G})\D \mathbb{P}^{\#}(x)=\overline{\mathbb{P}}(\mathcal{G})>0
    \end{equation}
    and since
    \begin{equation}
        \overline{\mathbb{P}}_x(\mathcal{G})=0\qquad \text{for $\m$-a.e. $x\in \big(\X\backslash Q\big)$},
    \end{equation}
    we can clearly find $Q_{\varepsilon}\in \mathrm{B}^+_{\m}(Q)$ such that it holds, for some $c\in \R_{>0}$, that
    \begin{equation}
        (1-\varepsilon)^{\frac{1}{2}}c\leq \overline{\mathbb{P}}_x(\mathcal{G})\leq c\qquad \text{for $\m$-a.e. $x\in Q_{\varepsilon}$}.
    \end{equation}
    Let then $\mathcal{D}_{\varepsilon}$ be the derivation as in \cref{lem:ConstrDeriv} built with respect to $\mathbb{P}$ and $\tfrac{1}{c}\1_{\mathcal{G}}$. With all this, exploiting \eqref{eq:ConstrDeriv0} and \eqref{eq:CurvToDeriv2}, we can write, for $\m$-a.e. $x\in Q_{\varepsilon}$, that
        \begin{equation}
        \begin{gathered}
            |\Dif_{\varepsilon}|_{\mathcal{X}_{\m}}(x)\geq \big(\Dif_{\varepsilon}(h_{l_0})\big)(x)=\\
            =\frac{1}{c}\int\limits_{\mathcal{G}} \frac{\big(\partial(h_{l_0}\circ \gamma)\big)(t)}{\ms_{\gamma}(t)}\D \overline{\mathbb{P}}_x(\gamma,t)\geq (1-\varepsilon),
            \end{gathered}
        \end{equation}
        which implies \eqref{eq:CurvToDerivL00}. On the other hand, bearing \eqref{eq:ConstrDeriv0} and \eqref{eq:CurvToDerivL0} in mind, with the use of standard properties of the vector-valued integration, we can write, for $\m$-a.e. $x\in Q_{\varepsilon}$, that
        \begin{equation}
            p\big(\Dif_{\varepsilon}(\bm{f})\big)(x)\leq \frac{1}{c}\int\limits_{\mathcal{G}}\frac{p\Big(\big(\partial(\bm{f}\circ \gamma)\big)(t)\Big)}{\ms_{\gamma}(t)}\D \overline{\mathbb{P}}_x(\gamma,t)\leq \zeta,
        \end{equation}
        which gives us exactly \eqref{eq:CurvToDerivL000}.
        
        This proof is complete.
    \end{proof}
    
\end{prop}

The lemma below allows one to switch from estimates from above on a given system of Lipschitz functions in terms of all possible Weaver derivations to corresponding estimates in terms of the metric speeds. In other words, exactly here we pass from (upper) algebraical relations to (upper) geometrical ones.
\begin{lem}\label{lem:FromDerivToCurU}
    Let $Q\in \mathrm{B}(\X)$, $d\in \Nat_0$, $p\in \mathsf{P}^d_+$, $\bm{f}\in \Mat^d\big(\LIP(\X)\big)$, $\zeta\in \R_{>0}$. Suppose it holds, for any $\Dif\in \mathcal{X}_{\m}$, that
    \begin{equation}\label{eq:FromDerivToCurU0}
        p\big(\Dif(\bm{f})\big)\leq \zeta |\Dif|_{\mathcal{X}_{\m}}\qquad \text{$\m$-a.e. on $Q$}.
    \end{equation}
    Let $\mathbb{P}$ be a fragment plan on $(\X,\m)$. Then it holds, for $\mathbb{P}$-a.e. $\gamma\in \Fr_{\mathrm{bL}}(\X)$, that
    \begin{equation}\label{eq:FromDerivToCurU00}
        \ms_{\bm{f}\circ \gamma}^p(t)\leq \zeta \ms_{\gamma}(t)\qquad \text{for $\mathcal{L}^1$-a.e. $t\in \gamma^{-1}(Q)$.}
    \end{equation}
    \begin{proof}
For every $\varepsilon\in (0,1)$ and for each $\bm{\xi}\in \R^d$ with $p^*(\bm{\xi})\leq 1$ we consider the family
        \begin{equation}
            \mathcal{G}_{\varepsilon,\bm{\xi}}\coloneqq \bigg\{(\gamma,t)\in \ev^{-1}(Q)\cap \overline{\Fr}_{\mathrm{bL}}(\X) \Bigm| \Big(\partial\big((\bm{\xi}\cdot \bm{f})\circ \gamma\big)\Big)(t)\geq (1+\varepsilon) \zeta \ms_{\gamma}(t)\bigg\},
        \end{equation}
        which can be seen to be Borel. The proof goes by contradiction, so we suppose the conclusion is false. From this, by looking at countable dense sets of $\varepsilon$ and $\bm{\xi}$, we can derive that $\overline{\mathbb{P}}(\mathcal{G}_{\varepsilon,\bm{\xi}})>0$ for some $\varepsilon\in (0,1)$ and some $\bm{\xi}\in \R^d$ with $p^*(\bm{\xi})\leq 1$, as well as that $\m(Q)>0$. Let then $\Dif_{\varepsilon}$ and $Q_{\varepsilon}$ be built with respect to $Q$, $\bm{\xi}\cdot \bm{f}$, $(1+\varepsilon) \zeta$, $\mathbb{P}$, and $\tfrac{\varepsilon}{2}$ as in \cref{prop:CurvToDerivU}. Now we juxtapose \eqref{eq:CurvToDerivU00}, \eqref{eq:CurvToDerivU000}, and \eqref{eq:FromDerivToCurU0}, from which we get that
        \begin{equation}
          (1+\e) \zeta\leq \big|\Dif_{\e} (\bm{\xi}\cdot \bm{f})\big|_{\Leb_{\m}} \leq \zeta |\Dif_{\varepsilon}|_{\mathcal{X}_{\m}} \leq \big(1+\tfrac{\varepsilon}{2}\big)\zeta \qquad \text{$\m$-a.e. on $Q_{\varepsilon}$}.
        \end{equation}
        As this is impossible due to the fact that $\m(Q_{\e})>0$, we can conclude the proof.
    \end{proof}
\end{lem}

The lemma below is completely analogous to the previous one, with the only difference in the direction of the estimates. So, here we pass from (lower) algebraical relations to (lower) geometrical ones.
\begin{lem}\label{lem:FromDerivToCurL}
    Let $Q\in \mathrm{B}(\X)$, $d\in \Nat_0$, $p\in \mathsf{P}^d_+$, $\bm{f}\in \Mat^d\big(\LIP(\X)\big)$, $\zeta\in \R_{> 0}$. Suppose it holds, for any $\Dif\in \mathcal{X}_{\m}$, that
    \begin{equation}\label{eq:FromDerivToCurL0}
        p\big(\Dif(\bm{f})\big)\geq \zeta |\Dif|_{\mathcal{X}_\m}\qquad \text{$\m$-a.e. on $Q$}.
    \end{equation}
    Let $\mathbb{P}$ be a fragment plan on $(\X,\m)$. Then it holds, for $\mathbb{P}$-a.e. $\gamma\in \Fr_{\mathrm{bL}}(\X)$, that
    \begin{equation}\label{eq:FromDerivToCurL00}
        \ms^{p}_{\bm{f}\circ \gamma}(t)\geq \zeta \ms_{\gamma}(t)\qquad \text{for $\mathcal{L}^1$-a.e. $t\in \gamma^{-1}(Q)$}.
    \end{equation}
    \begin{proof}
       For every $\varepsilon\in (0,1)$ consider the family
        \begin{equation}
            \mathcal{G}_{\varepsilon}\coloneqq \bigg\{(\gamma,t)\in \ev^{-1}(Q)\cap \overline{\Fr}_{\mathrm{bL}}(\X) \Bigm| \ms^{p}_{\bm{f}\circ \gamma}(t)\leq (1-\varepsilon) \zeta \ms_{\gamma}(t)\bigg\},
        \end{equation}
        which can be seen to be Borel. The proof goes by contradiction, so we suppose the conclusion is false. From this, by looking at a countable dense set of $\varepsilon$, we can derive that $\overline{\mathbb{P}}(\mathcal{G}_{\varepsilon})>0$ for some $\varepsilon\in (0,1)$, as well as that $\m(Q)>0$. Let then $\Dif_{\varepsilon}$ and $Q_{\varepsilon}$ be built with respect to $Q$, $p$, $\bm{f}$, $(1-\varepsilon) \zeta$, $\mathbb{P}$, and $\tfrac{\varepsilon}{2}$ as in \cref{prop:CurvToDerivL}. Now we juxtapose \eqref{eq:CurvToDerivL00}, \eqref{eq:CurvToDerivL000}, and \eqref{eq:FromDerivToCurL0}, from which we get that
        \begin{equation}
            \zeta\big(1-\tfrac{\varepsilon}{2}\big)\leq \zeta|\Dif_{\varepsilon} |_{\mathcal{X}_{\m}}\leq p\big(\Dif_{\varepsilon}(\bm{f})\big)\leq (1-\varepsilon)\zeta\qquad \text{$\m$-a.e. on $Q_{\varepsilon}$}.
        \end{equation}
        As this is impossible due to the fact that $\m(Q_{\e})>0$, we can conclude the proof.
    \end{proof}
\end{lem}

We have all the tools to prove \cref{lem:GeoDisAnDis}. To this end, we recall \cref{def:DistCoef}.
\begin{proof}[Proof of \cref{lem:GeoDisAnDis}]
The proofs for both parts are essentially identical. Still, we prefer to give them separately for the sake of clarity.

We start with the first part.

Arguing by contradiction, we can deduce the existence of $\zeta\in \R_{>0}$ and $Q\in \mathrm{B}_{\m}^+(\X)$ such that
    \begin{equation}\label{eq:GeoDisAnDis1}
        \GDis^{\mathrm{u}}_{p_{\mathrm{u}}}[\bm{f}]>\zeta\geq \ADis^{\mathrm{u}}_{p_{\mathrm{u}}}\big[\D_{\m} \bm{f}\big]\qquad \text{$\m$-a.e. on $Q$}.
    \end{equation}
    From the left estimate in \eqref{eq:GeoDisAnDis1}, with the use of assertion $\rm I)$ in \cref{lem:ModToPlan}, we derive the existence of a nonzero fragment plan $\mathbb{P}$ on $(\X,\m)$ such that, for $\mathbb{P}$-a.e. $\gamma\in \Fr_{\mathrm{bL}}(\X)$, one has
    \begin{equation}
       \mathcal{L}^1\bigg(\Big\{t\in \gamma^{-1}(Q)\bigm| \ms^{p_{\mathrm{u}}}_{\bm{f}\circ \gamma}(t)> \zeta\ms_{\gamma}(t)\Big\}\bigg)>0.
    \end{equation}
    At the same time, by \cref{def:DistCoef}, the right estimate in \eqref{eq:GeoDisAnDis1} yields, for any $\Dif\in \mathcal{X}_{\m}$, that
    \begin{equation}
        {p_{\mathrm{u}}}\big(\Dif(\bm{f})\big)={p_{\mathrm{u}}}\big((\D_{\m}\bm{f})(\Dif)\big)\leq \zeta|\Dif|_{\mathcal{X}_{\m}}\qquad \text{$\m$-a.e. on $Q$},
    \end{equation}
    whence, by \cref{lem:FromDerivToCurU}, it holds, for $\mathbb{P}$-a.e. $\gamma\in \Fr_{\mathrm{bL}}(\X)$, that
    \begin{equation}
         \ms^{p_{\mathrm{u}}}_{\bm{f}\circ \gamma}(t)\leq \zeta \ms_{\gamma}(t)\qquad \text{for $\mathcal{L}^1$-a.e. $t\in \gamma^{-1}(Q)$}.
    \end{equation}
    The obtained facts clearly contradict each other, as desired.

We proceed with the second part.

    Arguing by contradiction, we can deduce the existence of $\zeta\in \R_{>0}$ and $Q\in \mathrm{B}_{\m}^+(\X)$ such that
    \begin{equation}\label{eq:GeoDisAnDis2}
        \GDis^{\mathrm{l}}_{p_{\mathrm{l}}}[\bm{f}]<\zeta\leq \ADis^{\mathrm{l}}_{p_{\mathrm{l}}}\big[\D_{\m} \bm{f}\big]\qquad \text{$\m$-a.e. on $Q$}.
    \end{equation}
    From the left estimate in \eqref{eq:GeoDisAnDis2}, with the use of assertion $\rm II)$ in \cref{lem:ModToPlan}, we derive the existence of a nonzero fragment plan $\mathbb{P}$ on $(\X,\m)$ such that, for $\mathbb{P}$-a.e. $\gamma\in \Fr_{\mathrm{bL}}(\X)$, one has
    \begin{equation}
       \mathcal{L}^1\bigg(\Big\{t\in \gamma^{-1}(Q)\bigm| \ms^{p_{\mathrm{l}}}_{\bm{f}\circ \gamma}(t)< \zeta\ms_{\gamma}(t)\Big\}\bigg)>0.
    \end{equation}
    At the same time, by \cref{def:DistCoef}, the right estimate in \eqref{eq:GeoDisAnDis2} yields, for any $\Dif\in \mathcal{X}_{\m}$, that
    \begin{equation}
        {p_{\mathrm{l}}}\big(\Dif(\bm{f})\big)={p_{\mathrm{l}}}\big((\D_{\m}\bm{f})(\Dif)\big)\geq \zeta|\Dif|_{\mathcal{X}_{\m}}\qquad \text{$\m$-a.e. on $Q$},
    \end{equation}
    whence, by \cref{lem:FromDerivToCurL}, it holds, for $\mathbb{P}$-a.e. $\gamma\in \Fr_{\mathrm{bL}}(\X)$, that
    \begin{equation}
         \ms^{p_{\mathrm{l}}}_{\bm{f}\circ \gamma}(t)\geq \zeta \ms_{\gamma}(t)\qquad \text{for $\mathcal{L}^1$-a.e. $t\in \gamma^{-1}(Q)$}.
    \end{equation}
    The obtained facts clearly contradict each other, as desired.

     The proof is thus complete.
\end{proof}

\subsection{Main characterizations}\label{ss:MainCh}
In this subsection we combine all the previously obtained statements into our central results.

Throughout the remaining part, a locally finite Borel measure $\m$ on $\X$ is assumed fixed.

We start with the lemma below, which basically just combines \cref{theo:LACisApprox} and \cref{lem:GeoDisAnDis}. For the proof, we recall also the terminology introduced earlier.
\begin{lem}\label{lem:ColMainChar}
    Let $Q\in \mathrm{B}(\X)$, $d\in \mathbb{N}_0$. Suppose $\dim_{\mathcal{X}_{\m}}(Q)\leq d$. Let $\bm{\phi}\in \big[\Mat^d(\mathcal{X}^{\star}_{\m}),\mathsf{p}_{\infty}\big]_{\leq 1}$, $p_{\mathrm{u}},p_{\mathrm{l}}\in \mathsf{P}^d_+$, $\zeta\in (0,1)$. Then there exist a Borel $\m$-partition $(Q_m)_{m\in \Nat}$ of $Q$ and a family $(\bm{f}_m)_{m\in \Nat}\subseteq \Mat^d\big(\LIP_1(\X)\big)$ such that, for each $m\in \Nat$, the following properties hold:
\begin{gather}
    \zeta\GDis^{\mathrm{u}}_{p_{\mathrm{u}}}[\bm{f}_m]\leq \ADis^{\mathrm{u}}_{p_{\mathrm{u}}}[\bm{\phi}]\qquad \text{$\m$-a.e. on $Q_m$};\\
    \GDis^{\mathrm{l}}_{p_{\mathrm{l}}}[\bm{f}_m] \geq \zeta \ADis^{\mathrm{l}}_{p_{\mathrm{l}}}[\bm{\phi}]\qquad \text{$\m$-a.e. on $Q_m$}.
\end{gather}
    \begin{proof}
To begin with, let $\Phi$ denote the collection of all elements in $\mathcal{X}_{\m}^{\star}$ of the form
        \begin{equation}
            \sum\limits_{n\in \overline{1,N}}\1_{Q_n} \D_{\m} f_n,
        \end{equation}
        where $N\in \mathbb{N}$, $(Q_n)_{n\in \overline{1,N}}\subseteq \mathrm{B}(\X)$, $(f_n)_{n\in \overline{1,N}}\subseteq \LIP_1(\X)$, with the family $(Q_n)_{n\in \overline{1,N}}$ being disjoint. We then make an easy observation that the set $\Phi$ is disked and stable and that $\Phi\subseteq [\mathcal{X}^{\star}_{\m}]_{\leq 1}$. Moreover, by the very definition of $|\cdot|_{\mathcal{X}_{\m}}$, we deduce that the set $\Phi$ is norming for $\mathcal{X}_{\m}$ over $\X$.

    We now apply \cref{theo:LACisApprox}, which, after some elementary refining and assembling procedures, gives us a Borel $\m$-partition $(Q_m)_{m\in \Nat}$ of $Q$ and a family $(\bm{f}_m)_{m\in \Nat}\subseteq \Mat^d\big(\LIP_1(\X)\big)$ such that, for each $m\in\Nat$, it holds that
    \begin{equation}
       \zeta\ADis^{\mathrm{u}}_{p_{\mathrm{u}}}\big[\D_{\m} \bm{f}_m\big]\leq \ADis^{\mathrm{u}}_{p_{\mathrm{u}}}[\bm{\phi}]\qquad \text{$\m$-a.e. on $Q_m$}
    \end{equation}
    and that
    \begin{equation}
        \ADis^{\mathrm{l}}_{p_{\mathrm{l}}}\big[\D_{\m} \bm{f}_m\big]\geq \zeta \ADis^{\mathrm{l}}_{p_{\mathrm{l}}}[\bm{\phi}]\qquad \text{$\m$-a.e. on $Q_m$}.
    \end{equation}
    It remains only to juxtapose this with \cref{lem:GeoDisAnDis}.

    This proof is complete.
    \end{proof}
\end{lem}

A straightforward application of \cref{theo:MainEstADis} to the lemma above results in the following theorem that represents the main universal bounds on geometric distortion factors, the byproduct of which can be exactly observed in \cref{theo:MainRes} and \cref{theo:MainResEucl}. This statement should be treated as the conceptual core of all the analysis developed in our manuscript.
\begin{theo}\label{theo:FinDimBoundsGDis}
    Let $Q\in \mathrm{B}(\X)$, $d\in \mathbb{N}_0$. Suppose $\dim_{\mathcal{X}_{\m}}(Q)\leq d$. Let $\varepsilon\in (0,1)$. Then the assertions below are valid.
    \begin{enumerate}
        \item[$\rm I)$] There exists a Borel $\m$-partition $(Q_m)_{m\in \Nat}$ of $Q$ and a family $(\bm{f}_m)_{m\in \Nat}\subseteq \Mat^d\big(\LIP_1(\X)\big)$ such that, for each $m\in \Nat$, the following properties hold:
        \begin{gather}
        \GDis^{\mathrm{u}}_{\mathsf{p}_{\infty}}[\bm{f}_m]\leq 1 \qquad \text{$\m$-a.e. on $Q_m$};\\
            \GDis^{\mathrm{l}}_{\mathsf{p}_1}[\bm{f}_m]\geq 1-\varepsilon \qquad \text{$\m$-a.e. on $Q_m$}.
        \end{gather}
        \item[$\rm II)$] If the space $(\X,\m)$ is infinitesimally Euclidean, then there exists a Borel $\m$-partition $(Q_m)_{m\in \Nat}$ of $Q$ and a family $(\bm{f}_m)_{m\in \Nat}\subseteq \Mat^d\big(\LIP_1(\X)\big)$ such that, for each $m\in \Nat$, the following properties hold:
        \begin{gather}
        \GDis^{\mathrm{u}}_{\mathsf{p}_2}[\bm{f}_m]\leq 1 \qquad \text{$\m$-a.e. on $Q_m$};\\
            \GDis^{\mathrm{l}}_{\mathsf{p}_2}[\bm{f}_m]\geq 1-\varepsilon \qquad \text{$\m$-a.e. on $Q_m$}.
        \end{gather}
    \end{enumerate}
    \begin{proof}
        For both parts we just need to properly combine \cref{theo:MainEstADis} and \cref{lem:ColMainChar}, with the additional use of \cref{def:InfEucl} for the second part.
    \end{proof}
\end{theo}

In the next (rather technical) proposition we carry out an obvious gluing procedure with respect to given systems of Lipschitz functions, exactly as emerging in the theorem above, which results in a single Luzin--Lipschitz map that admits a control on the change of the lengths of curve fragments in terms of the corresponding geometric distortion factors.
\begin{prop}\label{prop:FromGDisToLen}
    Let $d\in \mathbb{N}_0$, $p_{\mathrm{u}},p_{\mathrm{l}}\in \mathsf{P}^d_+$, let $(Q_m)_{m\in \Nat}$ be a Borel $\m$-partition of $\X$, let $(\bm{f}_m)_{m\in \Nat}\subseteq \Mat^d\big(\LIP_1(\X)\big)$, let $\rho^{\mathrm{u}},\rho^{\mathrm{l}}\colon \X\to \R_{\geq 0}$ be Borel functions such that it holds, for each $m\in \Nat$, that
    \begin{equation}
        \GDis^{\mathrm{u}}_{p_{\mathrm{u}}}[\bm{f}_m]\leq \rho^{\mathrm{u}}\qquad \text{$\m$-a.e. on $Q_m$}
    \end{equation}
    and that
     \begin{equation}
        \GDis^{\mathrm{l}}_{p_{\mathrm{l}}}[\bm{f}_m]\geq \rho^{\mathrm{l}}\qquad \text{$\m$-a.e. on $Q_m$}.
    \end{equation}
    Then the formula
    \begin{equation}
        \varphi\coloneqq \sum\limits_{m\in \Nat} \1_{Q_m}\bm{f}_m
    \end{equation}
    defines a $1$-Luzin--Lipschitz, Borel map $\X\to \mathbb{R}^d_{\mathsf{p}_{\infty}}$ such that the following properties hold:
    \begin{gather}
    \Len_{p_{\mathrm{u}}}(\varphi\circ \gamma)\leq \int\limits_{\gamma} \rho^{\mathrm{u}} \qquad \text{for $\Mod_{\infty}^{\m}$-a.e. $\gamma\in \Fr_{\mathrm{r}}(\X)$};\label{eq:FromGDisToLen0}\\
        \Len_{p_{\mathrm{l}}}(\varphi\circ \gamma)\geq  \int\limits_{\gamma} \rho^{\mathrm{l}} \qquad \text{for $\Mod_{\infty}^{\m}$-a.e. $\gamma\in \Fr_{\mathrm{r}}(\X)$}.\label{eq:FromGDisToLen00}
    \end{gather}
    \begin{proof}
    The map $\varphi$ is indeed well posed and can be seen to be Borel and $1$-Luzin--Lipschitz. So, we are left to check the properties in \eqref{eq:FromGDisToLen0} and \eqref{eq:FromGDisToLen00}. Even though the verifications of both are quite similar, we prefer to perform them in parallel.
    
    We fix $\varepsilon\in (0,1)$ and put $\epsilon^{\mathrm{u}} \coloneqq (1+\varepsilon)$ and $\epsilon^{\mathrm{l}} \coloneqq (1-\varepsilon)$. Let us make the following observation: it suffices to check that
    \begin{equation}
         \Len_{p_{\mathrm{u}}}(\varphi\circ \gamma)\leq \epsilon^{\mathrm{u}}\int\limits_{\gamma} \rho^{\mathrm{u}} \qquad \text{for $\Mod_{\infty}^{\m}$-a.e. $\gamma\in \Fr_{\mathrm{r}}(\X)$}\label{eq:FromGDisToLen1}
    \end{equation}
    and that
     \begin{equation}
         \Len_{p_{\mathrm{l}}}(\varphi\circ \gamma)\geq \epsilon^{\mathrm{l}}\int\limits_{\gamma}\rho^{\mathrm{l}} \qquad \text{for $\Mod_{\infty}^{\m}$-a.e. $\gamma\in \Fr_{\mathrm{r}}(\X)$}\label{eq:FromGDisToLen2}.
    \end{equation}
    Indeed, the arbitrariness of $\varepsilon$ would then immediately lead us to the conclusion.
    
    To begin with, for each $m\in \Nat$, find a Borel partition $(Q_{m,l})_{l\in \Nat}$ of $Q_m$ and families $(\zeta^{\mathrm{u}}_{m,l})_{l\in \Nat}\subseteq \R_{\geq 0}$ and $(\zeta^{\mathrm{l}}_{m,l})_{l\in \Nat}\subseteq \R_{\geq 0}$ such that it holds, for each $l\in \Nat$, that
    \begin{equation}
        \rho^{\mathrm{u}}\leq \zeta^{\mathrm{u}}_{m,l}\leq \epsilon^{\mathrm{u}} \rho^{\mathrm{u}}\qquad \text{on $Q_{m,l}$}
    \end{equation}
    and that
    \begin{equation}
        \rho^{\mathrm{l}}\geq \zeta^{\mathrm{l}}_{m,l}\geq \epsilon^{\mathrm{l}} \rho^{\mathrm{l}}\qquad \text{on $Q_{m,l}$}.
    \end{equation}
    After that we put
        \begin{equation}
            Q_0\coloneqq \bigg(\X\Big\backslash \bigcup\limits_{m,l\in \Nat}Q_{m,l}\bigg),
        \end{equation}
        so that we have $\m(Q_0)=0$. Also, for each $m\in \Nat$, let $\rho^{\mathrm{u}}_m$ and $\rho_m^{\mathrm{l}}$ be nonnegative Borel $\m$-representatives, respectively, of $\GDis^{\mathrm{u}}_{p_{\mathrm{u}}}[\bm{f}_m]$ and of $\GDis^{\mathrm{l}}_{p_{\mathrm{l}}}[\bm{f}_m]$ with
        \begin{equation}
            \rho^{\mathrm{u}}_m\leq \rho^{\mathrm{u}}\qquad \text{on $Q_m$}
        \end{equation}
        and
        \begin{equation}
            \rho^{\mathrm{l}}_m\geq \rho^{\mathrm{l}}\qquad \text{on $Q_m$}.
        \end{equation}

        Let us fix $p\in \mathsf{P}^d_+$ and consider, for every $\gamma\in \Fr(\X)$, the function $\kappa_{\gamma}\colon \R^d_p\to \overline{\R}_{\geq 0}$ given for any $z\in \R^d$ by
        \begin{equation}\label{eq:FromGDisToLen3}
            \kappa_{\gamma}(z)\coloneqq \sum\limits_{m,l\in \Nat} \mathscr{N}\Big((\bm{f}_m\circ \gamma)^{-1}(z)\cap \gamma^{-1}(Q_{m,l})\Big),
        \end{equation}
        which can be easily seen to be Carath{\'e}odory $\mathcal{H}^1_p$-measurable. We now fix $\gamma\in \Fr(\X)$ and claim that
        \begin{equation}\label{eq:FromGDisToLen4}
            \Len_p(\varphi\circ \gamma)=\int\limits_{\R^d_p} \kappa_{\gamma}(z)\D \mathcal{H}^1_p(z),
        \end{equation}
        for which we clearly just need to check that
        \begin{equation}
            \kappa_{\gamma}(z)=\mathscr{N}\big((\varphi\circ \gamma)^{-1}(z)\big)\qquad \text{for $\mathcal{H}^1_p$-a.e. $z\in \R^d_p$}.
        \end{equation}
        On one hand, since the sets $Q_{m,l}$, $m,l\in \Nat$, are disjoint, we have, for any $z\in \R^d_p$, that
         \begin{equation}
        \begin{gathered}
           \kappa_{\gamma}(z)\leq \mathscr{N}\Bigg(\bigcup\limits_{m,l\in \Nat} (\bm{f}_{m}\circ \gamma)^{-1}(z)\cap \gamma^{-1}(Q_{m,l})\Bigg)\leq\mathscr{N}\big((\varphi\circ \gamma)^{-1}(z)\big).
            \end{gathered}
        \end{equation}
        On the other hand, we can write, for any $z\in \big(\R^d_p\backslash \{\o_{\R^d}\}\big)$, that
        \begin{equation}
            \mathscr{N}\big((\varphi\circ \gamma)^{-1}(z)\big)\leq \kappa_{\gamma}(z)+\mathscr{N}\big((\varphi\circ \gamma)^{-1}(z)\cap \gamma^{-1}(Q_0)\big)=\kappa_{\gamma}(z).
        \end{equation}
        The claim is thus valid. Next, we fix $\gamma\in \Fr_{\mathrm{r}}(\X)$ and $\widehat{\gamma}\in \Fr_{\mathrm{r}}(\X)$ such that the fragments $\gamma$ and $\widehat{\gamma}$ are equivalent, for which we claim that
        \begin{equation}\label{eq:FromGDisToLen5}
            \Len_p(\varphi\circ \widehat{\gamma})=\Len_p(\varphi\circ \gamma),
        \end{equation}
        for which, in turn, it is sufficient to show that
        \begin{equation}\label{eq:FromGDisToLen6}
            \kappa_{\widehat{\gamma}}(z)=\kappa_{\gamma}(z)\qquad \text{for $\mathcal{H}^1_p$-a.e. $z\in \R^d_p$}.
        \end{equation}
        Applying \cref{prop:FromRectToLip}, we find a Borel set $\widehat{Q}\subseteq \X$ with $\mathcal{H}^1_{\X}(\widehat{Q})=0$ such that
        \begin{equation}
            \mathscr{N}\big(\widehat{\gamma}^{-1}(x)\big)=\mathscr{N}(\gamma^{-1}(x)\big)\qquad \text{for any $x\in \big(\X\backslash \widehat{Q}\big)$}.
        \end{equation}
        With this, let us then notice, for all $m,l\in \Nat$ and for any $z\in \big(\R^d_p\backslash \bm{f}_m(\widehat{Q})\big)$, that
        \begin{equation}
        \begin{gathered}
            \mathscr{N}\Big((\bm{f}_m\circ \gamma)^{-1}(z)\cap \gamma^{-1}(Q_{m,l})\Big)=\sum\limits_{x\in \bm{f}_m^{-1}(z)\cap Q_{m,l}} \mathscr{N}\big(\gamma^{-1}(x)\big)=\\
            =\sum\limits_{x\in \bm{f}_m^{-1}(z)\cap Q_{m,l}} \mathscr{N}\big(\widehat{\gamma}^{-1}(x)\big)=\mathscr{N}\Big((\bm{f}_m\circ \widehat{\gamma})^{-1}(z)\cap \widehat{\gamma}^{-1}(Q_{m,l})\Big).
            \end{gathered}
        \end{equation}
        But it surely holds, for each $m\in \Nat$, that $\mathcal{H}^1_p\big(\bm{f}_m(\widehat{Q})\big)=0$ due to the fact that the map $\bm{f}_m$ is Lipschitz, whence, summing the equalities as above over $m,l\in \Nat$ and taking \eqref{eq:FromGDisToLen3} into account, we can indeed get the property in \eqref{eq:FromGDisToLen6}. The claim is therefore true.

        Let now $\G_{\mathrm{L}}$ denote the family of all $\gamma\in \big(\Fr_{\mathrm{L}}(\X)\backslash \Gamma^+_{Q_0}(\X)\big)$ such that it holds, for each $m\in \Nat$, that
        \begin{equation}
            \ms^{p_{\mathrm{u}}}_{\bm{f}_m\circ \gamma}(t)\leq \rho^{\mathrm{u}}_m(\gamma_t) \ms_{\gamma}(t)\qquad \text{for $\mathcal{L}^1$-a.e. $t\in \gamma^{-1}(Q_m)$}
        \end{equation}
        and that
        \begin{equation}
            \ms^{p_{\mathrm{l}}}_{\bm{f}_m\circ \gamma}(t)\geq \rho^{\mathrm{l}}_m(\gamma_t) \ms_{\gamma}(t)\qquad \text{for $\mathcal{L}^1$-a.e. $t\in \gamma^{-1}(Q_m)$}.
        \end{equation}
        Put also $\G^0_{\mathrm{L}}\coloneqq \big(\Fr_{\mathrm{L}}(\X)\backslash \G_{\mathrm{L}}\big)$. With this, by \eqref{eq:GeoDisFac00000} and \eqref{eq:GeoDisFac000000}, we get, with the use of \cref{prop:SubsZerMod}, that
        \begin{equation}
            \Mod_{\infty}^{\m}(\G^0_{\mathrm{L}})=0.
        \end{equation}
        Fix then $\gamma\in \G_{\mathrm{L}}$. We claim the validity of the inequalities in \eqref{eq:FromGDisToLen1} and \eqref{eq:FromGDisToLen2}. Let us verify only the lower part, as the upper part is essentially the same. So, in view of what is above, after applying \cref{prop:ArForm}, exploiting that $\gamma\not\in \Gamma^+_{Q_0}(\X)$, and performing some simple transformations, we can get the following chain:
        \begin{equation}
            \begin{gathered}
            \epsilon^{\mathrm{l}}\int\limits_{\gamma}\rho^{\mathrm{l}} =\epsilon^{\mathrm{l}}\int\limits_{Q_0} \rho^{\mathrm{l}}(x)\mathscr{N}\big(\gamma^{-1}(x)\big)\D \mathcal{H}^1_{\X}(x)+\\
                +\epsilon^{\mathrm{l}}\sum\limits_{m,l\in \Nat} \int\limits_{Q_{m,l}} \rho^{\mathrm{l}}(x)\mathscr{N}\big(\gamma^{-1}(x)\big)\D \mathcal{H}^1_{\X}(x)\leq \\
                \leq \sum\limits_{m,l\in \Nat}\zeta^{\mathrm{l}}_{m,l} \int\limits_{Q_{m,l}} \mathscr{N}\big(\gamma^{-1}(x)\big)\D \mathcal{H}^1_{\X}(x)=\\
                =\sum\limits_{m,l\in \Nat} \zeta^{\mathrm{l}}_{m,l} \int\limits_{\gamma^{-1}(Q_{m,l})} \ms_{\gamma}(t)\D\mathcal{L}^1(t)\leq\\
                \leq \sum\limits_{m,l\in \Nat}  \int\limits_{\gamma^{-1}(Q_{m,l})} \ms^{p_{\mathrm{l}}}_{\bm{f}_{m}\circ \gamma}(t)\D\mathcal{L}^1(t)=\\
                =\sum\limits_{m,l\in \Nat}  \int\limits_{\R^d_{p_{\mathrm{l}}}} \mathscr{N}\Big((\bm{f}_{m}\circ \gamma)^{-1}(z)\cap \gamma^{-1}(Q_{m,l})\Big) \D \mathcal{H}^1_{p_{\mathrm{l}}}(z).
            \end{gathered}
        \end{equation}
         Combining this with \eqref{eq:FromGDisToLen3} and \eqref{eq:FromGDisToLen4}, we achieve exactly the desired inequality.
         
          We are left to transfer the obtained properties to all rectifiable fragments. So, let $\G_{\mathrm{r}}^0$ be the family of all $\gamma\in \Fr_{\mathrm{r}}(\X)$ such that at least one of the inequalities in \eqref{eq:FromGDisToLen1} and \eqref{eq:FromGDisToLen2} does not hold. Let then $\widehat{\G}^0_{\mathrm{r}}$ be the collection of all Lipschitz fragments $\widehat{\gamma}$ such that, for some $\gamma\in \G_{\mathrm{r}}^0$, the fragments $\gamma$ and $\widehat{\gamma}$ are equivalent. From \cref{prop:FromRectToLip}, with the use of the relevant definitions, we can easily get that
        \begin{equation}\label{eq:FromGDisToLen7}
            \Mod_{\infty}^{\m}(\G_{\mathrm{r}}^0)= \Mod_{\infty}^{\m}(\widehat{\G}_{\mathrm{r}}^0).
        \end{equation}
        Now, fix $\gamma\in \G_{\mathrm{r}}^0$ and $\widehat{\gamma}\in \widehat{\G}^0_{\mathrm{r}}$ such that the fragments $\gamma$ and $\widehat{\gamma}$ are equivalent. Without loss of generality, we may assume that the inequality in \eqref{eq:FromGDisToLen2} does not hold for $\gamma$. We make the following claim: the inequality in \eqref{eq:FromGDisToLen2} does not hold for $\widehat{\gamma}$ either. Indeed, by \cref{prop:FromRectToLip} again, in combination with the definition of the line integral, we immediately deduce that
        \begin{equation}
            \int\limits_{\widehat{\gamma}} \rho^{\mathrm{l}}=\int\limits_{\gamma}\rho^{\mathrm{l}},
        \end{equation}
        which, coupled with \eqref{eq:FromGDisToLen5}, proves the impossibility of that inequality for $\widehat{\gamma}$. This, however, implies that $\widehat{\G}_{\mathrm{r}}^0\subseteq \G_{\mathrm{L}}^0$ and hence that
        \begin{equation}\label{eq:FromGDisToLen8}
            \Mod_{\infty}^{\m}(\widehat{\G}^0_{\mathrm{r}})=0.
        \end{equation}
        Juxtaposing \eqref{eq:FromGDisToLen7} and \eqref{eq:FromGDisToLen8}, we finally conclude that
        \begin{equation}
            \Mod_{\infty}^{\m}(\G_{\mathrm{r}}^0)=0,
        \end{equation}
        which is precisely what we wanted to show.

        The proof is now fully finished.
    \end{proof}
\end{prop}

We are in a position to provide the strongest versions of our main results. These will clearly imply \cref{theo:MainRes} and \cref{theo:MainResEucl} once we recall \cref{theo:FromHausToWeav}.
\begin{theo}
    Let $d\in \mathbb{N}_0$. Suppose $\dim_{\mathcal{X}_{\m}}(\X)\leq d$. Then the assertions below are valid.
    \begin{itemize}
        \item[$\rm I)$] Let $\varepsilon\in (0,1)$. Then there exists a $1$-Luzin--Lipschitz, Borel map $\varphi_{\varepsilon}\colon \X\to \mathbb{R}^d_{\mathsf{p}_{\infty}}$ such that
        \begin{equation}
            \Len_{\mathsf{p}_{\infty}}(\varphi_{\varepsilon}\circ \gamma)\leq \Len_{\X}(\gamma)\leq (1+\varepsilon)\Len_{\mathsf{p}_1}(\varphi_{\varepsilon}\circ \gamma)\qquad \text{for $\Mod_{\infty}^{\m}$-a.e. $\gamma\in \Fr(\X)$}.
        \end{equation}
        \item[$\rm II)$] Suppose the space $(\X,\m)$ is infinitesimally Euclidean. Let $\varepsilon\in (0,1)$. Then there exists a $1$-Luzin--Lipschitz, Borel map $\varphi_{\varepsilon}\colon \X\to \mathbb{R}^d_{\mathsf{p}_{\infty}}$ such that
        \begin{equation}
          \Len_{\mathsf{p}_2}(\varphi_{\varepsilon}\circ \gamma)\leq \Len_{\X}(\gamma)\leq (1+\varepsilon)\Len_{\mathsf{p}_2}(\varphi_{\varepsilon}\circ \gamma)\qquad \text{for $\Mod_{\infty}^{\m}$-a.e. $\gamma\in \Fr(\X)$}.
        \end{equation}
    \end{itemize}
    \begin{proof}
        For the proof it suffices to combine \cref{theo:FinDimBoundsGDis} with \cref{prop:FromGDisToLen}.
    \end{proof}
\end{theo}

{\bf Acknowledgments.} This study continues research begun by the author as part of his Master's Thesis at the Moscow Institute of Physics and Technology.

\printbibliography
\end{document}